\documentclass[leqno,12pt]{article}

\usepackage[all]{xy}
\usepackage{amsthm}
\usepackage{amssymb} 
\usepackage{amsmath,amscd} 
\usepackage{mathrsfs}

\def\G{\Gamma}
\def\L{\Lambda}

\def\kk{{\mathbf k}}

\def\wcdast{{\circledast\circledast}}

\def\LL{{\Bbb L}}
\def\NN{{\Bbb N}}

\def\PP{{\Bbb P}}

\def\RR{{\Bbb R}}

\def\ZZ{{\Bbb Z}}

\def\bfA{{\mathbf{A}}}
\def\bfB{{\mathbf{B}}}

\def\bfD{{\mathbf{D}}}

\def\L{\Lambda}
\def\G{\Gamma}

\def\cA{{\cal A}}
\def\cB{{\cal B}}
\def\cC{{\cal C}}
\def\cD{{\cal D}}
\def\cE{{\cal E}}
\def\cF{{\cal F}}
\def\cG{{\cal G}}

\def\cI{{\cal I}}
\def\cJ{{\cal J}}

\def\cM{{\cal M}}
\def\cN{{\cal N}}
\def\cO{{\cal O}}
\def\cP{{\cal P}}

\def\cR{{\cal R}}
\def\cS{{\cal S}}
\def\cT{{\cal T}}
\def\cU{{\cal U}}

\def\tuH{{\textup{H}}}
\def\tuZ{{\textup{Z}}}

\def\awe{{a^{\wedge}}}
\def\bwe{{b^{\wedge}}}

\def\frka{{\mathfrak{a}}}

\def\id{\operatorname{id}}

\def\pr{\operatorname{\textup{pr}}}
\def\Ob{\operatorname{Ob}}
\def\ct{\operatorname{\textup{ct}}}

\def\Mod{\operatorname{Mod}}

\def\coker{\operatorname{coker}}
\def\image{\operatorname{im}}

\def\Hom{\operatorname{Hom}}
\def\End{\operatorname{End}}
\def\Bic{\operatorname{Bic}}

\def\Ext{\operatorname{Ext}}

\DeclareMathOperator*{\colim}{colim}

\def\dga{\operatorname{\textup{dga}}}
\def\dgCat{\operatorname{\textup{dgCat}}}

\def\cCdg{\operatorname{\mathcal{C}_{\textup{dg}}}}
\def\cCss{\operatorname{\mathcal{C}_{\textup{ss}}}}

\newcommand{\cohm}{\textup{H}^{\bullet}}

\newcommand{\cone}{\textup{c}}
\newcommand{\cocone}{\textup{cc}}

\newcommand{\Path}{\textup{Path}}
\newcommand{\Cyl}{\textup{Cyl}}

\DeclareMathOperator*{\holim}{holim}
\DeclareMathOperator*{\hocolim}{hocolim}

\def\Perf{\textup{Perf}}
\def\dgPerf{\textup{Perf}_{\textup{dg}}}

\def\BBic{\operatorname{\mathbb{B}ic}}
\def\dgBic{\operatorname{\mathcal{B}ic}}

\def\sSet{\operatorname{sSet}}
\def\sCat{\operatorname{sCat}}

\def\sMod{\operatorname{sMod}}

\def\Map{\operatorname{Map}}

\def\infCat{\operatorname{\infty \textup{Cat}}}

\def\leftcone{\operatorname{{}^{\lhd}\hspace{-2pt}}}

\def\sner{\operatorname{\textup{N}}}
\def\Homr{\operatorname{\Hom^{\textup{R}}}}

\def\Ho{\operatorname{\textup{Ho}}}

\def\cofright{\rightarrowtail}
\def\trivcofright{\stackrel{\sim}{\rightarrowtail}}

\def\trivfibright{\stackrel{\sim}{\twoheadrightarrow}}

\def\ulcJ{\operatorname{\underline{\cJ}}}

\newtheorem{lemma}{Lemma}[section]
\newtheorem{proposition}[lemma]{Proposition}
\newtheorem{theorem}[lemma]{Theorem}
\newtheorem{corollary}[lemma]{Corollary}

\newtheorem{nitheorem}[lemma]{``Theorem"}
\newtheorem{nicorollary}[lemma]{``Corollary"}

\newtheorem{Assumption}[lemma]{Assumption}

{

}

\theoremstyle{definition}
\newtheorem{const}[lemma]{Construction}

\newtheorem{example}[lemma]{Example}
\newtheorem{definition}[lemma]{ Definition}

\theoremstyle{remark}

\newtheorem{remark}[lemma]{Remark}

\title{
Derived bi-duality  via homotopy limit 
}

\author{
Hiroyuki MINAMOTO
}
\date{}

\begin{document}
\maketitle

\begin{abstract}
We show that a derived bi-duality dg-module  
is quasi-isomorphic to the homotopy limit of a certain tautological functor.  
This is a simple observation,   
which seems to be true in wider context. 
From the view point of derived Gabriel topology, 
this is a derived version of results of J. Lambek 
about localization and completion of ordinary rings. 
However   
the important point is that 
we can obtain a simple formula for the bi-duality modules 
only when we come to the derived world 
from the abelian world.

We give applications. 
1. we give a generalization and  an intuitive proof of Efimov-Dwyer-Greenlees-Iyenger Theorem 
which asserts that the completion of commutative ring satisfying some conditions is obtained as a derived bi-commutator.  
(We can also  prove Koszul duality for dg-algebras with Adams grading satisfying mild conditions.) 
2. We prove that every smashing localization of dg-category is obtained as a derived bi-commutator of some pure injective module. 
This is a derived version of  the classical results in localization theory of ordinary rings. 

These applications shows that 
our formula  together with the viewpoint that 
a derived bi-commutator is a completion in some sense, 
provide us a fundamental understanding of a derived bi-duality module. 
\end{abstract}

\tableofcontents

\section{Introduction}
The following situation and its variants are ubiquitous in Algebras and Representation theory: 

Let $R$ be a ring, $J$ an $R$-module 
and $E:= \End_{R}(J)^{\textup{op}}$ the opposite ring of  the endomorphism ring of $J$ over $R$. 
Then we have the duality 
\[
(-)^{\ast} :=\Hom_{R}(-,J) : \Mod R \rightleftarrows (\Mod E)^{\textup{op}} : \Hom_{E}(-,J)=:(-)^\ast  
\]
and 
the unite map $\epsilon_{M}:M \to M^{\ast\ast}$ is given by the evaluation map:
\[
\epsilon_{M}(m): \Hom_{R}(M,J) \to J ,\, f \mapsto f(m) \textup{ for } m \in M.
\]
The bi-dual $R^{\ast\ast}$ of $R$ is called the \textit{bi-commutator} 
(or the \textit{double centralizer}) and denoted by $\Bic_{R}(J)$.  
The following is more popular expression (or the usual definition) of the bi-commutator 
\[
\Bic_{R}(J) := \End_{E}(J)^{\textup{op}}.
\]
The bi-commutator has a ring structure and 
the evaluation map $\epsilon_{R}:R \to \Bic_{R}(J)$ 
become a ring homomorphism. 
In particular, the case when 
the canonical algebra homomorphism $R \to \Bic_{R}(J)$ become an isomorphism, 
the module $J$ is said to have the \textit{double centralizer property}. 
Dualities together with  evaluation maps, bi-commutators  
and double centralizer properties are one of the central topics in Algebras and Representation theory. 
(See e.g. \cite{Jacobson,KSX,Lambek1,Lambek2,Morita loc,Tachikawa})

Recently the concern with  the \textit{derived bi-commutators} (or the \textit{derived double centralizers}) 
has been growing: 

Let $R$ a ring (or more generally dg-algebra) 
$J$ an (dg-)$R$-module
 and $\cE:= \RR\End_R(J)^{\textup{op}}$  
 the opposite dg-algebra of the  endomorphism dg-algebra of $J$. 
Then the derived bi-commutator is defined by  
\[
\BBic_{R}(J):= \RR\End_{\cE}(J)^{\textup{op}}.  
\]
There also exists a canonical algebra homomorphism $R \to \BBic_{R}(J)$. 
 Derived double centralizer property for special modules has been extensively studied 
 as a part of Koszul duality. (See e.g. \cite{LPWZ,Posi}.) 

In \cite[Section 4.16]{DGI}, Dwyer-Greenlees-Iyenger call 
a pair $(R,J)$ \textit{dc-complete}, 
in the case when $J$ has derived double centralizer property.  
They proved the following surprising and impressive theorem, which we will refer as 
\textit{completion theorem}.

\begin{theorem}[\textup{\cite{DGI},\cite{Efi}}]\label{EDGI thm}
Let $R$ be a commutative Noetherian ring and $\frka$ an ideal such that 
the residue ring $R/\frka$ is of finite global dimension. 
We denote by $\widehat{R}$ the $\frka$-adic completion.  
Then we have a quasi-isomorphism 
\[
\widehat{R} \simeq \BBic_{R}(R/\frka)
\]
where 
$\BBic_{R}(R/\frka)$ is the derived bi-commutator of $R/\frka$ over $R$.
\end{theorem}

From the view point of Derived-Categorical Algebraic Geometry (DCAG), 
all important procedure in Algebraic Geometry should have derived-categorical interpretation. 
In \cite{Konchan} Kontsevich proposed that 
 formal completion for a scheme is obtained as a derived bi-commutator. 
Following this idea, 
Efimov \cite{Efi} 
introduced derived bi-commutator of subcategory $\cJ \subset \cD(R)$ 
and  proved a scheme version of above theorem. 
Since formal completion plays an important role in Algebraic Geometry, 
completion theorem and its scheme version  are expected to become  important in  DCAG.  
Therefore it is desirable to obtain better understanding of this theorem.

In the proof of completion theorem, 
Grothedieck vanishing theorem for local cohomology is used. 
Since it is special theorem for commutative Noetherian rings, 
it is preferable to obtain more categorical proof. 
Recently Porta-Shaul-Yekutieli \cite{PSY2} generalize completion theorem 
for a commutative ring $R$ and a weakly proregular ideal $\frka$ 
based on their work \cite{PSY1} about 
the derived functors of the completion functors and the  torsion functors.
However it is still remain unclear that 
to what extent we can obtain a transcendental outcome by a homological operation 
with finite input. 
In this paper we establish a simple description of the derived bi-commutator, 
which  enable us to  give a more intuitive  proof of completion theorem. 
Actually the description is given by a certain tautological homotopy limit,  
and hence seems to state 
 that every  derived bi-commutator is completion in some sense. 
 (We can make this precise by introducing a notion of derived Gabriel topology.)
 
For this purpose, 
 we study  derived bi-duality: 
\[
(-)^{\circledast} :=\RR\Hom_{R}(-,J) : \cD(R) 
\rightleftarrows \cD(\cE)^{\textup{op}} : \RR\Hom_{\cE}(-,J)=:(-)^\circledast.  
\]
For a special class of modules $J$, 
derived bi-duality is already studied in the context of Gorenstein dg-algebras \cite{FJ,Jorgensen,MW}. 
We consider general dg-module $J$ and 
 establish a simple description of the derived bi-dual module $M^{\wcdast}$ 
via a certain tautological homotopy limit.  
This is the main theorem of this paper.
As an application other than completion theorem, 
we discuss smashing localization of dg-categories.

As mentioned above, 
derived bi-dualities, derived bi-commutator and 
derived double centralizer property  are expected to play prominent roles 
in Algebras, Representation theory, 
Derived-Categorical Algebraic Geometry. 
Our main theorem together with the view point that 
derived bi-commutators are completion in some sense, 
would have many applications. 
Moreover 
since the main theorem is proved in a formal argument,  
the same formula should hold in more wider context. 
Bi-duality is a basic operation which appears in every where of mathematics.    
So it can be expect that 
our main theorem become an indispensable tool in many area of mathematics.

Below we explain the contents of the present paper. 
First,  we give the main theorem with omitting some details. 
Next, we show that the main theorem leads to an intuitive proof of a generalization of Theorem \ref{EDGI thm}. 

\subsection{Derived bi-duality via homotopy limit}\label{926 1121}

Let $\cA$ be a dg-algebra and $J$  a dg $\cA$-module.
We denote $\cE := (\RR\End_\cA(J))^{\textup{op}}$ be the opposite dg-algebra of  the endomorphism dg-algebra. 
Then $J$ has a natural dg $\cE$-module structure. 
We obtain the dualities 
\[
(-)^{\circledast} :=\RR\Hom_{\cA}(-,J) : 
\cD(\cA) \rightleftarrows \cD(\cE)^{\textup{op}}: \RR\Hom_{\cE}(-,J)=: (-)^{\circledast}
\]
There are natural transformations $\epsilon:1_{\cD(\cA)} \to (-)^{\wcdast}$ 
induced from evaluation morphisms. 

We denote by $\langle J \rangle$ 
the smallest thick subcategory containing $J$. 
Namely $\langle J\rangle$ is 
the full subcategory of $\cD(\cA)$ consisting those objects which 
constructed from $J$ by taking cones, shifts, and direct summands finitely many times. 

Let $M$ be a dg $\cA$-module. 
We denote by $\langle J \rangle_{M/}$ 
the under category. 
Namely, the objects of $\langle J\rangle_{M/}$ are morphisms $k:M \to K$ with $K\in \langle J \rangle$ 
and the morphisms from 
$k: M \to K$ to $\ell: M \to L$ are the morphisms $\phi: K\to L$ in $\langle J \rangle$ 
such that $\ell = \phi \circ k$. 
This category $\langle J \rangle_{M/} $ comes naturally equipped with 
the co-domain functor 
$\Gamma: \langle J \rangle_{M/}\to \cD(\cA)$ 
which sends an object $k:M \to K$ to its co-domain $K$. 
\[
\Gamma:  \langle J \rangle_{M/}  \to \cD(\cA),\quad [k:M \to K] \mapsto K.
\]
An elementary observation given in  Appendix \ref{obs for bic holim thm}  
suggest 
the following simple formula for the derived bi-dual module $M^{\wcdast}$.  

\begin{nitheorem}\label{out main thm}
We have the following quasi-isomorphism 
\[
M^{\wcdast}  \simeq \holim_{\langle J \rangle_{M/}} \Gamma
\]
\end{nitheorem}

\begin{remark}
In the above theorem and the following corollary 
we omit homotopy theoretical details. 
For the rigorous statement see Theorem \ref{bic holim thm}. 
\end{remark} 

Since the bi-dual  $\cA^{\wcdast}$ of $\cA$  is naturally isomorphic to 
the derived bi-commutator $\BBic_{\cA}(J)$, 
\[
\cA^{\wcdast} = \RR\Hom_{\cE}(\RR\Hom_{\cA}(\cA,J),J) \cong \RR\Hom_{\cE}(J,J) = \BBic_{\cA}(J) 
\]

in particular,  
we have the following corollary. 

\begin{nicorollary}\label{out main cor}
\[
\BBic_{\cA}(J) \simeq \holim_{\langle J \rangle_{\cA/}}\Gamma. \]
\end{nicorollary}

These theorem and corollary provide us a fundamental understanding of derived bi-duality functors. 
We give  applications in this paper. 


\subsection{Completion via derived bi-commutator 
}

As the first application, 
we generalize the completion theorem and 
give an intuitive proof. 

Let $R$ be a ring and $\frka$ a two-sided ideal. 
An (right) $R$-module $M$ is called $\frka$-\textit{torsion} 
if for any $m\in M$ there exists $n\in \NN$ such that 
$m\frka^{n}=0$. 
We denote by $\frka\textup{-tor}$ the full subcategory of $\Mod R$ 
consisting of $\frka$-torsion modules. 
We denote by $\cD_{\frka\textup{-tor}}(R)$ the full subcategory of $\cD(R)$ 
consisting of complexes with $\frka$-torsion cohomology groups.
We denote by $\cD(\frka\textup{-tor})$ the full subcategory of $\cD(R)$ 
consisting of complexes each term of which is $\frka$-torsion module.

\begin{theorem}\label{out app 1}
Assume that the canonical inclusion  functor 
$\cD(\frka\textup{-tor}) \to\cD_{\frka\textup{-tor}} (R)$ 
gives an equivalence and that 
$R/\frka^n$ belongs to $\langle R/\frka \rangle$ for $n \geq 0$. 
We denote by $\widehat{R}$  the $\frka$-adic completion. 
Then we have a quasi-isomorphism 
\[
\BBic_{R}(R/\frka) \simeq \widehat{R}. 
\]
\end{theorem}

\begin{remark}
We put artificial conditions on the above theorem, 
in order to clarify that to what extent the derived bi-commutator gives the $\frka$-adic completion. 
\end{remark}

\noindent
``\textit{Proof}".
\begin{Assumption}\label{holim lim assu}
In this  ``Proof" we assume that 
$
\holim =\lim.
$
\end{Assumption}
We denote by $\cI$ the (non-full) subcategory of $ \langle R/\frka \rangle_{R/} $ 
which consists of objects $\pi^n: R \to R/\frka^n$ for $n \geq 1$ 
and 
of morphisms $\pi^m \to \pi^n$ induced from the 
canonical projections $\varphi^{m,n}: R/ \frka^m \to R/\frka^{n}$ for $ m\geq n$. 
In other words, 
$\cI$ is the image of the functor $(\ZZ_{\geq 1})^{\textup{op}} \to \cD(\cA)$ 
which sends an object $n$ to $\pi^{n}$ and a  morphism $m \to n$ to $\pi^{m} \to \pi^{n}$  
where we consider the ordered set $\ZZ_{\geq 1}$ as a category in the standard way.  
Then we have 
\[
\lim_{\cI}\Gamma|_{\cI} \cong \lim_{n\to \infty} R/\frka^n \cong \widehat{R}.
\]
Therefore  
by Corollary \ref{out main cor} 
 and Assumption \ref{holim lim assu}, it is enough to show that 
 $\lim_{\cI}\G|_{\cI} \cong \lim_{\langle R/\frka \rangle_{R/}} \G$. 
Therefore 
it is enough to show that $\cI$ is a left cofinal subcategory of $\langle R/\frka \rangle_{R/}$. 
Namely 
only we have to show that 
the over category $\cI_{/k}$ is non-empty and connected for each $k \in \langle R/\frka \rangle_{R/} $. 

Let $k:R \to K$ be an object of $ \langle R/\frka \rangle_{R/}$.  
Since we assume that $\cD(\frka\textup{-tor})\stackrel{\sim}{\rightarrow}\cD_{\frka\textup{-tor}} (R) $, 
$K$ belongs to $\cD(\frka\textup{-tor})$. 
It follows that  $K$ is quasi-isomorphic to  a complex $K'$ each terms of which is an $\frka$-torsion modules.  
Therefore a morphism $k :R \to  K$ canonically factors through some cyclic $\frka$-torsion module $R/\frka^n$. 
\[
\begin{xymatrix}{
R \ar[d]_{\pi^{n}} \ar[dr]^k & \\
R/\frka^n \ar[r]_{\psi} & K 
}
\end{xymatrix}\]
In other words there exists a morphism $\psi:\pi^n \to k$ in $\langle R/\frka \rangle_{R/}$. 
This proves the non-emptiness of $\cI_{/k}$.    
Since the factorization $k = \psi \circ \pi^n$ is canonical,  
we see that $\cI_{/k}$ is connected.   
This shows that $\cI $ is left co-final in $ \langle R/\frka \rangle_{R/}$ and completes the ``proof".  
\hfill $``\square"$

\subsection{Smashing localization  via derived bi-commutator} 

First we recall the following classical fact. 
\begin{theorem}[\textup{\cite[Corollary 3.4.1]{Lambek1}, \cite[Theorem 7.1]{Morita loc}}]\label{926 1140}
Let $f:R \to S$ be a (right) Gabriel localization of a ring $R$,  
that is, $f$ is an epimorphism in the category of rings and $S$ is left flat over $R$. 
Let $J$ be a co-generator of the torsion theory which corresponds to the Gabriel localization $f$. 
If we take a product $J': = J^{\kappa}$ of copies of $J$ over large enough cardinal $\kappa$, 
then we have an isomorphism 
\[
\Bic_{R}(J') \cong S. 
\]
\end{theorem}

In this section we prove a derived version. 
A morphisms $f:\cA \to \cB$ of dg-algebras is called \textit{smashing localization} 
if the restriction functor $f_{\ast}: \cD(\cB) \to \cD(\cA)$ 
is fully faithful. 
Recall that 
a ring homomorphism $R\to S$ is an epimorphism 
in the category of rings if and only if 
the restriction functor $f_{\ast}: \Mod S \to \Mod R$ is fully faithful. 
Therefore smashing localization can be considered as a dg-version of epimorphisms.  

\begin{theorem}\label{out loc thm 0}
Let  $\cA \to \cB$ be a smashing localization of dg-algebras 
and $J$  be a pure injective co-generator of $\cD(\cB)$. 
Then we have a quasi-isomorphism over $\cA$ 
\[
\BBic_{\cA}(f_{\ast}J') \simeq \cB.\]
where $J'= J^{\Pi\kappa}$ is a large enough product of $J$. 
\end{theorem}
The notion of  pure injective co-generator 
which is introduced by Krause \cite{Krause}  is a dg-version of injective co-generator for the module category.  
We will recall the definition in  Definition \ref{pure inj def}. 

A similar  theorems is   proved by Nicol\'as and Saorin \cite{NS}. 
In our way of the proof, 
an essential point is the following theorem, 
which is also intuitively proved from the view point that 
a bi-duality is a completion in some sense (See Appendix \ref{App 2}).  
\begin{theorem}\label{out loc thm}
Let $J$ a pure injective co-generator of $\cD(\cA)$ and $M$ a dg $\cA$-module.  
If we take a product $J' = J^{\Pi \kappa}$ of copies of $J$ over large enough cardinal $\kappa$, 
then the evaluation morphism is a quasi-isomorphism 
\[
\epsilon_{M}: M \stackrel{\sim}{\rightarrow} M^{\wcdast} 
\]
where the bi-dual is taken over $J'$. 
\end{theorem}
In the case when $\cA$ is  an ordinary ring and $M$ is a module, 
the same results is already   proved by Shamir \cite{Shamir}.


\subsection[Koszul duality for Adams graded dg-algebras \\(a part of joint work with A. Takahashi)]
{Koszul duality for Adams graded dg-algebras 
\\(a part of joint work with A. Takahashi)} 

The following theorem will be proved and applied in  \cite{MinTaka}. 
\begin{theorem}\label{out Koszul thm}
Let $\cA:= \cA_0 \oplus \cA_1 \oplus \cA_{2} \oplus \cdots $ be an $\NN$-Adams graded dg-algebra. 
If the $\cA_0$-modules $\cA_n$ satisfies a mild condition. 
Then we have a quasi-isomorphism 
\[
\BBic_{\cA}(\cA/\cA_{\geq 1}) \simeq \cA
\]
\end{theorem}

The proof is given in the same way of the proof of Theorem \ref{out app 1}.




\subsection{From the view point of Derived Gabriel topology}

Gabriel topology is a special class of linear topology on rings, which plays an
important role in the theory of localization of rings \cite{Sten}.
The notion of derived Gabriel topology, which is a Gabriel topology for a dg-algebra, is introduced in \cite{DGT}. 
From the view point of derived Gabriel topology,  
Theorem \ref{out main thm} says that 
the derived bi-dual $M^{\wcdast}$ equipped with ``the finite topology"  
 is the ``$J$-adic completion" of $M$.  
In this sense 
Theorem \ref{out main thm} is 
inspired by the following results of J. Lambek. 

\begin{theorem}[\textup{\cite[Theorem 4.2]{Lambek1},(See also \cite[Theorem 3.7]{Lambek2})}]\label{lambek thm}
Let $R$ be a ring and $J$ an injective $R$-module. 
For an $R$-module $M$, we denote by $Q(M)$ the module of quotients with respect to $J$. 
Assume that 
every torsionfree factor module of $Q(M)$ is $J$-divisible.  
Then the (ordinary) bi-duality $\Hom_{\End_{R}(J)}(\Hom_{R}(M,J),J)$ 
equipped with the finite topology is the $J$-adic completion of $Q(M)$.  
\end{theorem}

We remark that 
in the original Lambek Theorem,  
there are several assumptions.  
Contrary to this, 
the main theorem has no assumptions. 
This is an important difference from 
an usual derived version of some result in classical theory of rings and modules. 
So the point is that 
we can obtain a simple formula for the bi-duality modules 
only when we come to the derived world 
from the abelian world.

At the first sight, 
three theorems below  concerning on derived bi-dualities 
\begin{itemize}
\item Completion theorem 

\item Localization theorem 

\item Koszul duality 
\end{itemize}
seem to be theorems of different kind. 
However in the present paper we will see that 
these are consequences of a simple formula, which is the main theorem \ref{out main thm}. 
From the view point of derived Gabriel topology, 
these theorems are consequences of completeness of each algebras 
with respect to appropriate topologies.

\subsection*{Notations and conventions}

Unless otherwise is stated the term `` modules " means   right modules. 

Let $M$ and $N$ be modules. 
Assume that direct sum decompositions $M = \bigoplus M_{j}$ and $N = \bigoplus N_i$ are given. 
We often use a matrix style expression of morphisms $(f_{i,j}): M \to N$. 
In this paper 
no modules are given more than one direct sum decompositions. 
Hence there are no danger of confusion. 
For the sake of space, 
we often use  the transposed expression ${}^{t}(f_1,\cdots, f_n): M \to N_1 \oplus \cdots \oplus N_n$ 
of a column vector 
$\begin{pmatrix}f_1 \\ \vdots \\ f_n \end{pmatrix} : M \to N_1 \oplus \cdots \oplus N_n$ .  

The totally ordered sets $(\ZZ_{\geq 1})^{\textup{op}}$ and $(\ZZ_{\geq 0})^{\textup{op}}$ 
are always  considered as categories 
with a unique morphism $\textbf{geq}^{m,n}: m\to n$ for $m \geq n$.






\subsection*{Acknowledgment}
The author thanks 
 I. Iwanari, Y. Kimura and  S. Moriya 
  for discussions on homotopy theory. 
He also thanks  
Y. Mizuno and 
K Yamaura 
for answering questions on algebras. 
He thanks M. Wakui, who gave him an interest on bi-commutators.  
He thanks P. Nicol\'as for notifying him the paper \cite{NS}. 
This work was supported by JSPS KAKENHI Grant Number 00241065.

 
\section{Preliminaries}

\subsection{Simplicial model categories} 

For the theory of simplicial model categories, we refer to \cite{Hirsch,Hovey,HTT}.

\noindent 
{\bf $\bullet$ Model categories.}  
For a model category $\cM$, we denote by $\cM^{\circ}$ the full subcategory of $\cM$ 
consisting of fibrant-cofibrant objects. 
For a full subcategory $\cN \subset \cM$ 
we denote by $\Ho(\cN)$ the full subcategory of the homotopy category $\Ho(\cM)$ 
of $\cM$ spanned the objects of $\cN$.  

Recall 
the model structure in $\cM$ induces 
the model structure in 
 the under category $\cM_{M/}$ in the following way: 
 let $\varphi: k \to \ell$ a morphism of  $\cM_{M/}$ 
 from $k: M \to K $ to $\ell : M \to L$. 
 We denote by $\varphi':K \to L $ the morphism of $\cM$  induced from $\varphi$. 
 Then $\varphi$ is defined to be  weak equivalence (resp. fibration, cofibration) in $\cM_{/M}$ 
 if $\varphi'$ is so in $\cM$. 
 In particular 
 we see that 
 an object $k : M\to K$ is fibrant if and only if 
 the morphism $k$ is a fibration in $\cM$ 
 and that 
 $k$ is cofibrant if and only if the domain $P$ is cofibrant in $\cM$.  

In a similar way the model structure in $\cM$ induces the model structure in 
 the over category $\cM_{M/}$. 
There is a natural functor $\Ho(\cM_{M/}) \to \Ho(\cM)_{M/}$. 
In general we can't expect that this functor is full. 
However we have the following lemma.

\begin{lemma}\label{model lem}
Let $\cM$ be a model category, 
$M$ a cofibrant object, $K$ a fibrant-cofibrant object and $L$ a fibrant object. 
Assume that a cofibration $k: M \rightarrowtail K$  and a morphism $\ell: M \to L$ are  given. 
Assume moreover that 
a morphism $\overline{\varphi}: k \to \ell$ in $\Ho(\cM)_{M/}$ is given. 
In other words 
 the following commutative diagram is given in the homotopy category $\textup{Ho}(\cM)$: 
\[
\begin{xymatrix}{
M \ar[d]_{k} \ar[dr]^{\ell}& \\
K \ar[r]_{\overline{\varphi}}& L. 
}
\end{xymatrix}\]
Then 
there is a morphism $\varphi: k \to \ell$ in $\cM_{M/}$ 
such that 
the image of $\varphi$ in $\textup{Ho}(\cM)$ coincide with $\overline{\varphi}$. 
Namely 
there exists a morphism $\varphi : K \to L$ 
such that  $\ell= \varphi\circ k$ and the image of $\varphi$ in $\textup{Ho}(\cM)$ coincide with $\overline{\varphi}$.
\end{lemma}

\begin{proof}
There exists a morphism $\varphi':K  \to L $ such that $\varphi'\circ k$ and $\ell$ are homotopic. 
Therefore we have the following commutative diagram except dotted arrow: 
\[
\begin{xymatrix}{
K \ar[d]_{\varphi'}  \ar@{-->}[dr]^{\varphi"} & *++{M} \ar@{>->}_{k}[l] \ar[d] \ar[dr]^\ell & \\
L & \Path(L) \ar@{->>}[l]_-{\sim}^{\textup{pr}_1} \ar@{->>}[r]^-{\sim}_-{\textup{pr}_2} & L
}\end{xymatrix}\]
where $\Path(L)$ is a path object of $L$. 
Since $k$ is a cofibration and $\textup{pr}_1$ is a trivial fibration, 
there exists a morphism $\varphi": K \to \Path(L)$ which completes the above commutative diagram. 
Now the composite morphism $\varphi := \textup{pr}_2\circ \varphi"$ satisfies the desired property. 
\end{proof}

\vspace{8pt}

\noindent
{\bf $\bullet$ Simplicial categories.} 
Recall that a simplicial category $\cM$ is a category which is enriched over the category $\sSet$ of simplicial sets. 
For a simplicial category $\cM$, we denote by $\Map_{\cM}(K,L)$ the mapping complex for $K,L \in \cM$. 
The underlying  category $\cM$ of a simplicial category $\cM$  is a category  
the objects is the same with that of $\cM$ and 
the Hom set is defined to be $\Map_{\cM}(K,L)_0$ the set of $0$-simplexes of $\Map_{\cM}(K,L)$.  
The homotopy category $\textup{h}\cM$ of a simplicial category $\cM$  is a category  
the objects is the same with that of $\cM$ and 
the Hom set is defined to be $\Hom_{\textup{h}\cM}(K,L) := \pi_0(\Map_{\cM}(K,L))$. 
A simplicial functor $F: \cM \to \cN$ is called a \textit{Dwyer-Kan equivalence} 
if the induced functor  $\textup{h}F: \textup{h}\cM \to \textup{h}\cN$ gives an equivalence 
of categories and the induced map $\Map_{\cM}(K,L) \to \Map_{\cN}(F(K),F(L))$ is a weak equivalence 
for each pair $K,L \in \cM$. 

A simplicial category $\cM$ is called \textit{fibrant} 
if every mapping  complex $\Map_{\cM}(K,L)$ is Kan complex, 
i.e., a fibrant object with respect to the standard model structure in $\sSet$.  
Let $X$ be a simplicial set and $x\in X_{0}$ a $0$-simplex of $X$. 
We often denote by $\{x\} \to X$ 
the map of simplicial set $\Delta^{0} \to X$ corresponds to 
$x$ under the natural bijection $X_{0} \cong \Map_{\sSet}(\Delta^{0},X)$. 
In particular for a morphism $\phi: K \to L$ 
we often denote by $\{ \phi \} \to \Map_{\cM}(K,L)$ the map of simplicial sets 
$\Delta^{0} \to \Map_{\cM}(K,L)$ which corresponds to the $0$-simplex $\phi \in \Map_{\cM}(K,L)_{0}=\Hom_{\cM}(K,L)$.

Let $M$ be an object of $\cM$. 
We equip the under category $\cM_{M/}$ with a structure of simplicial category in the following way: 
let $k: M \to K $ and $\ell: M \to L$ be objects of $\cM$. 
Then the mapping complex $\Map_{\cM_{M/}}(\ell,k)$ is defined to be 
the sub simplicial set of $\Map_{\cM}(L,K)$ which fits into the following pull back diagram of simplicial sets 
\begin{equation}\label{925 1116}
\begin{xymatrix}{
\Map_{\cM_{M/}}(\ell,k) \ar@{^{(}->}[r] \ar[d] & \Map_{\cM}(L,K) \ar[d]^{\ell^{\ast}} \\
\{k\} \ar[r] & \Map_{\cM}(M,K). 
}\end{xymatrix}\end{equation}
where the right vertical arrow $\ell^{\ast}:= \Map_{\cM}(\ell,K)$ is the induced morphism.

In a similar way, we equip the over category $\cM_{/M}$ with a structure of simplicial category.

A functor $-\otimes - : \cM \times \sSet \to \cM$ is called \textit{tensor product} 
if we have  a natural  isomorphism of simplicial sets 
\[
\Map_{sSet}(X,\Map_{\cM}(K,L)) \cong \Map_{\cM}(K\otimes X,L)
\]
for $X \in \sSet$ and $K,L\in \cM$. 
A functor $\cM \times \sSet^{\textup{op}} \to \cM, (K,X) \mapsto K^{X}$ is called \textit{cotensor product} 
if we have a natural isomorphism of simplicial sets 
\[
\Map_{\sSet}(X,\Map_{\cM}(K,L)) \cong \Map_{\cM}(K,L^{X}). 
\]

Let $\cM$ be a simplicial category having a cotensor product $\cM \times \sSet^{\textup{op}} \to \cM$.  
For an object $k:M \to K$ of the under category $\cM_{M/}$  
and a  simplicial set $X$,  
we define an object $k^X$ of $\cM_{/M}$ 
to be the object given by the composition 
\[
M \xrightarrow{k} K \cong K^{\Delta^0} \xrightarrow{K^{\textbf{uni}}} K^{X}
\] 
where we denote by 
$\textbf{uni}$ a unique map $X \to \Delta^0$. 
Then we claim that 
the functor $\cM_{M/} \times \sSet^{\textup{op}} \to \cM, (k,X) \mapsto k^{X}$ 
is a cotensor product. 
Indeed we have a commutative diagram 
\[
\begin{xymatrix}{
\Map_{\sSet}(X,\Map_{\cM}(L,K)) \ar[d]_{\Map_{\sSet}(X,\ell^{\ast})} \ar@{=}[r]^{\qquad\sim} & 
\Map_{\cM}(L,K^{X}) \ar[d]^{\ell^{\ast}}\\
\Map_{\sSet}(X,\Map_{\cM}(M,K)) \ar@{=}[r]^{\qquad\sim}  & \Map_{\cM}(M,K^{X}) \\ 
\Map_{\sSet}(X,\{ k\}) \ar@{=}[r]^{\sim} \ar[u] & \{k^{X}\}, \ar[u]
}\end{xymatrix}
\]
which implies that we have a natural isomorphism 
\[
\Map_{\sSet}(X,\Map_{\cM_{M/}}(\ell,k)) \cong \Map_{\cM_{M/}}(\ell, k^{X}).
\]

In a similar way, 
for a simplicial category $\cM$ having a tensor product $-\otimes-: \cM \times \sSet \to \cM$, 
we equip the over category $\cM_{/M}$ with a tensor product in the following way: 
for an object $p : P \to M$ of $\cM_{/M}$ 
and a  simplicial set $X$ 
we define an object $p\otimes X $ of $\cM_{/M}$ 
to be the object given by the composition 
\[
P \otimes X \xrightarrow{1_{P}\otimes \textbf{uni}} P \otimes \Delta^0 \cong P \xrightarrow{p} M
\] 
where we denote by 
$\textbf{uni}$ a unique map $X \to \Delta^0$.

\vspace{8pt} 

\noindent
{\bf $\bullet$  Simplicial model categories.}
We recall that 
a simplicial model category $\cM$ is not only a simplicial category equipped with a model structure, 
but also assumed to have a tensor product and a cotensor product 
satisfying axioms which 
state that the structure of simplicial category and the model structure are consistent.   
We remained the readers that 
\begin{enumerate}
\item for a cofibration $K \rightarrowtail K'$ and a fibrant object $L$ of a simplicial model category $\cM$, 
the induce morphism $\Map_{\cM}(K',L) \to \Map_{\cM}(K,L)$ is a fibration of simplicial sets 

\item 
for a cofibrant object $K$ and a fibration $L \to L'$, 
the induced morphism  $\Map_{\cM}(K,L) \to \Map_{\cM}(K,L')$ is a fibration. 

\item the fibrant-cofibrant part $\cM^{\circ}$ is a fibrant simplicial category. 

\item the homotopy category $\Ho(\cM^{\circ})$ with respect to the model structure 
and the homotopy category $\textup{h}\cM^{\circ}$ with respect to the structure of simplicial category 
are equivalent. 
In other words the identity on the objects gives an equivalence $\Ho(\cM^{\circ}) \simeq \textup{h}\cM^{\circ}$. 
\end{enumerate}  
In particular, the last two properties are important in a relationship with $\infty$-categories.

The under category $\cM_{M/}$  of a simplicial model category $\cM$ 
fail to become a simplicial category, 
since $\cM_{M/}$ is not tensored over $\sSet$. 
However we can prove 

\begin{lemma}\label{925 1053}
Let $k: M \to K$ be a fibrant object and $\ell: M \to L$ a cofibrant object of the under category $\cM_{M/}$. 

(1) 
The mapping complex $\Map_{\cM_{M/}}(\ell,k)$ is Kan complex. 
In particular  the simplicial category $\cM^{\circ}$ is a fibrant simplicial category.

(2) 
The object  $k^{\Delta^1}$ is a path object of $p$. 
  
(3) 
Two  morphisms $f,g: \ell \to k$ are  homotopic if and only if  
then they are weak homotopy equivalence, i.e., 
the images of $f,g$ in $\pi_{0}(\Map_{\cM_{M/}}(\ell,k))$ coincide. 

(4) 
Assume that $k$ and $\ell$ are fibrant-cofibrant objects. 
Then the morphism $\phi: \ell \to k$ is weak homotopy equivalence 
if and only if $\phi$ is weak equivalence. 
In particular 
 the identity on the objects gives an equivalence  
 $\textup{h}\cM^{\circ}_{M/} \simeq \Ho(\cM_{M/}^{\circ})$. 
\end{lemma}

\begin{proof}
(1) By the definition of the model structure in $\cM_{M/}$, 
the morphism $\ell: M \to  L$ is a cofibration in $\cM$ and 
the co-domain $K$ of $k$ is a fibrant object of $\cM$. 
Hence the induced  morphism $\ell^{\ast}=\Map_{\cM}(\ell,K): \Map_{\cM}(L,K) \to \Map_{\cM}(M,K)$ 
is a Kan fibration of simplicial sets. 
By the pull back diagram (\ref{925 1116}) we conclude that the simplicial set $\Map_{\cM_{M/}}(\ell,k)$ 
is a Kan complex.

(2) Using  the characterization of fibrations and cofibrations of $\cM_{/M}$ 
and the axioms of simplicial model category, we can check that 
the morphism $k^{\textbf{uni}}:k \to k^{\Delta^{1}} $ is weak equivalence and 
the induced morphism $k^{\iota}: k^{\Delta^{1}} \to k^{\{0\}\sqcup \{1\}} \cong k\times k$ is a fibration 
where  $\iota: \{0\} \sqcup\{1\} \to \Delta^1$ is the canonical inclusion. 

(3) Since  the mapping complex $\Map_{\cM_{M/}}(\ell,k)$ is a Kan complex by (1),  
 the $0$-th homotopy is given by the co-equalizer 
\[
\pi_{0}(\Map_{\cM_{M/}}(\ell,k)) =
 \textup{coeq}\left[d_0,d_1: \Map_{\cM_{M/}}(\ell,k)_1 \rightrightarrows \Map_{\cM_{M/}}(\ell,k)_0 \right]
\]
where $d_0,d_1$ are the boundary maps. 
Let $\iota_{0}:\Delta^{0} \cong \{0\} \to \Delta^1$ and $\iota_{1}: \Delta^{0}\cong  \{1\} \to \Delta^1$ be canonical inclusions. 
We have the following commutative diagram 
\[
\begin{xymatrix}{
\Map_{\cM_{M/}}(\ell,k)_{1} \ar[d]_{d_a} \ar@{=}[r]^{ \sim} & 
\Hom_{\cM_{M/}}(\ell,k^{\Delta^{1}}) \ar[d]^{\Hom_{\cM_{M/}}(\ell,k^{\iota_a})} \\
 \Map_{\cM_{M/}}(\ell,k)_{0}  \ar@{=}[r]^{ \sim} & 
\Hom_{\cM_{M/}}(\ell,k) 
}\end{xymatrix}
\]
for $a = 0,1$. 
Therefore two morphisms  $f,g: \ell \to k$ are weak homotopy equivalent 
if and only if we have a commutative diagram 
\begin{equation}\label{926 1254}
\begin{xymatrix}{
& \ell \ar[dl]_{f} \ar[d]^{H} \ar[dr]^{g} \\
k & k^{\Delta^{1}} \ar[l]^{k^{\iota_{0}}} \ar[r]_{k^{\iota_{1}}} & k.
}\end{xymatrix}\end{equation}
for some $H: \ell \to k^{\Delta^{1}}$. 
Thus by (2) if $f$ and $g$ are weak homotopy equivalent then they are homotopic.

Conversely assume that $f$ and $g$ are homotopic. 
Then by \cite[Proposition 1.2.5]{Hovey} 
there exists a right homotopy from $f$ to $g$ using the path object $k^{\Delta^{1}}$. 
Namely we have the above commutative diagram (\ref{926 1254}) for some $H$. 
Therefore $f$ and $g$ are weak homotopy equivalent.

(4) follows from (3).

\end{proof}

In a similar way we can prove 

\begin{lemma}\label{App B lem 1}
Let $p: P\to M $ be a fibrant object  and $q: Q \to M$ a cofibrant  of $\cM_{/M}$. 

(1) 
The mapping complex $\Map_{\cM_{/M}}(p,q)$ is a Kan complex. 
 In particular the simplicial category $\cM_{/M}^{\circ}$ is a fibrant simplicial category.

(2) 
The object  $q \otimes \Delta^1$ is a cylinder object of $q$. 
  
(3) 
Two morphisms $f,g: p\to q$ are  homotopic if and only if 
 they are weak homotopy equivalence, i.e., 
the images of $f,g$ in $\pi_{0}(\Map_{\cM_{/M}}(p,q))$ coincide. 

(4) 
Assume that $p$ and $q$ are fibrant-cofibrant. 
Then 
the the morphism $\phi: p \to q$ is weak equivalence 
if and only if $\phi$ is weak homotopy equivalence. 
In particular, the identity on the objects gives an equivalence
 $\textup{h}\cM^{\circ}_{/M} \simeq \Ho(\cM^{\circ}_{/M})$
\end{lemma}

\subsection{Homotopy theory of dg-categories}

For dg-categories, we refer to \cite{Drin,ddc,Keller2,Tab,Toen,Toen2}. 

\noindent 
{\bf $\bullet$ The category $\cC(\cA)$ of dg $\cA$-modules.}

Trough out of this paper $\kk$ is a commutative ring 
and dg-categories are dg-categories over $\kk$. 
Let  $\cA$ be a small dg-category.   
For a pair of objects  $a,b\in \cA$, we denote by $\Hom_{\cA}^{\bullet}(a,b)$ the Hom complex of $\cA$. 
We denote by $\cC(\cA)$ the category of dg $\cA$-modules 
and by $\cCdg(\cA)$ the dg-category of dg $\cA$-modules. 
For simplicity we set ${}_{\cA}(M,N):=\Hom_{\cCdg(\cA)}^{\bullet}(M,N)$ for dg $\cA$-modules $M$ and $N$.

It is known that $\cC(\cA)$ admits two structure if combinatorial model category whose weak equivalences are 
the quasi-isomorphisms:  

\begin{enumerate}
\item The projective model structure, 
whose fibrations are the term-wise surjection. 

\item The injective model structure, 
whose cofibrations are the term-wise injection. 
\end{enumerate}
Note that every object of $\cC(\cA)$ is projective fibrant and injective cofibrant. 
Through out of this paper 
we equip $\cC(\kk)$ with the projective model structure. 

For an object $a\in \cA$, 
we denote by $a^{\wedge}$ the free dg $\cA$-module at $a$, 
that is, 
the dg-functor $\cA^{\textup{op}} \to \cCdg(\kk)$ 
which sends $b\in \cA$ to $\cA^{\bullet}(b,a)$. 
A dg $\cA$-module $Q$ is called \textit{relative projective} 
if, in $\cC(\cA)$, it is a direct sum of a direct summand of dg-modules of the form 
$a^{\wedge}[n], a\in \cA, n\in \ZZ$. 
We denote by ${}^{\wedge}a$ the free dg left $\cA$-module at $a$, 
that is, 
the dg-functor $\cA \to \cCdg(\kk)$ 
which sends $b\in \cA$ to $\cA^{\bullet}(a,b)$.

We fix an injective co-generator $E$ of $\Mod \kk$. We set  $\bfD:= {}_{\kk}(-,E)$.  
We define $a^{\vee}:= \bfD({}^{\wedge}a)$. 
Namely 
 $a^{\vee}$ is the dg $\cA$-module $\cA^{\textup{op}} \to \cC(\kk)$ 
which sends $b\in \cA$ to ${}_{\kk}(\cA^{\bullet}(a,b),E)$.
A dg $\cA$-module is called \textit{relative injective} 
if, in $\cC(\cA)$, it is a direct summand of a direct product of modules of the form $a^{\vee}[n],a\in \cA,n\in\ZZ$. 
We can easily check that 
if $N$ is a projectively cofibrant left dg $\cA$-module, then $\bfD(N)$ is injectively fibrant.

Let $\cA^{\#}$ be the underlying graded category of $\cA$. 
We denote by $\cG(\cA^{\#}) $ the category of graded $\cA^{\#}$ modules. 
Let  $\#: \cC(\cA) \to \cG(\cA^{\#})$ be the forgetful functor. 
A dg $\cA$-module $M$ is determined by $M^{\#}$ and the differential $d_M$ on $M^{\#}$.  
Therefore we may write $M= (M^{\#},d_{M})$. By abusing notation, we often write $M= (M, d_{M})$. 
A sequence $L \to M \to N$ of dg $\cA$-modules is called 
$\#$-\textit{exact} if the induced sequence $L^{\#} \to M^{\#} \to N^{\#}$ is 
an exact sequence of graded $\cA^{\#}$-modules. 
Since the functor $\#$ is faithful, 
we may use the morphism $f^{\#}:M^{\#} \to N^{\#}$ of the underlying modules 
to denote a morphism $f: M\to N$ of dg $\cA$-modules. 

For a morphism $f:M \to N$, 
we set the cone $\cone(f)$ and the co-cone $\cocone(f)$ to be 
\[
\begin{split}
\cone(f) &:= \left( N\oplus M[1], \begin{pmatrix} d_N & f \\ 0& -d_M \end{pmatrix} \right),\\ 
\cocone(f)& := \left( N[-1] \oplus M, \begin{pmatrix} -d_N & f \\ 0 & d_M \end{pmatrix} \right). 
\end{split}
\] 
Observe that there are following $\#$-exact sequences 
\begin{equation}\label{1010 101}
\begin{split}
&0\to N \xrightarrow{{}^{t}(1_N,0)} \cone(f) \xrightarrow{(0,1_{M[1]}) } M[1] \to 0,\\
&0\to N[-1]  \xrightarrow{{}^{t}(1_{N[-1]},0)}\cocone(f) \xrightarrow{(0,1_{M})} M \to 0. 
\end{split}
\end{equation}

By \cite[Proof of Lemma 2.3]{ddc} 
the functor $\#$ has the left adjoint $F_{\lambda}$ and the  right adjoint $F_{\rho}$. 
Therefore limits and colimits  commute with $\#$. 
In other words, 
the underlying graded module of a (co)limit dg $\cA$-module 
is obtained as the (co)limit of the underlying graded modules.  

A dg $\cA$-module $M$ is called a \textit{free} 
if it is a direct sum of modules of the form $a^{\wedge}[n]$ for $a \in \cA$ and 
$n \in \ZZ$. 
A dg $\cA$-module $P$ is called a \textit{semi-free module} 
if it has an exhaustive filtration $0=P_0 \subset P_1 \subset \cdots P$ 
such that each quotient $P_{i}/P_{i-1}$ is free. 
Note that the underlying graded $\cA^{\#}$-module $P^{\#}$ of a semi-free dg $\cA$-module $P$ is free 
and that a semi-free dg $\cA$-module $P$ is projectively cofibrant.  
It is known that every dg $\cA$-module $M$ has a semi-free resolution, i.e., 
a projectively trivial  fibration $P \stackrel{\sim}{\twoheadrightarrow} M$ with $P$ semi-free. 
(See \cite[Appendix III]{Drin}.) 
Dually let $Q$ be a semi-free left dg $\cA$-module. 
Then 
the  dg $\cA$-module $\bfD(Q)$ is injectively fibrant and 
the underlying graded $\cA^{\#}$-module $\bfD(Q)^{\#}$ is injective. 
Moreover for every dg $\cA$-module $M$ there exists a trivial injective cofibration 
$M \trivcofright \bfD(Q)$ for some semi-free left module $Q$. 
 
\begin{lemma}\label{prel 1 lem 0}
Every  projective cofibrant object $M$ is a direct summand of some semi-free module $P$. 
Consequently  
the underlying graded $\cA^{\#}$-module $M^{\#}$ is projective. 
Dually every injective fibrant object $N$ is a direct summand of $\bfD(Q)$ for some left semi-free module $Q$. 
Consequently 
 the graded $\cA^{\#}$-module $N^{\#}$ is injective.  
\end{lemma} 

\begin{proof}
We prove the first statement. The second statement is proved in a dual way. 
Let $p:P \stackrel{\sim}{\twoheadrightarrow} M$ be a semi-free resolution. 
Using the lifting property,  we see that $p$ split and $M$ is a direct summand of $P$.  
\end{proof}

\begin{lemma}\label{922 142}
Let $f: M \to N $ be a morphism of dg $\cA$-modules. 

(1) $f$ is a projective cofibration if and only if 
it is a term-wise injection with the projective cofibrant cokernel. 

(2) $f$ is a injective fibration if and only if 
it is a term-wise surjection with the injective fibrant kernel.

\end{lemma}

\begin{proof}
Using Lemma \ref{prel 1 lem 0} we can prove this lemma 
by the same methods of \cite[Lemma 2.3.9,Lemma 2.3.20]{Hovey}. 

\end{proof} 

\begin{lemma}\label{101 1250}
(1) 
Let $0\to L \to M \to N \to 0$ be a $\#$-exact sequence of dg $\cA$-modules. 
Assume that $L$ and $N$ are projectively cofibrant. 
Then so is  the middle term $M$.
Dually, 
if $L$ and $N$ are injectively fibrant, then so is  $M$.  

(2) Let $f:M\to N $ be a morphism of dg $\cA$-modules. 
If $M$ and $N$ are injectively fibrant, 
then the cone $\cone(f)$ and the co-cone $\cocone(f)$ are injectively fibrant. 
Dually, if $M$ and $N$ are projectively cofibrant, 
then the cone and the co-cone are  projectively cofibrant. 
\end{lemma}

\begin{proof}
We prove the first statements of (1) and (2).  
The second statements are  proved in dual ways.
(1)
By Lemma \ref{922 142} the morphism $L \to M$ is projective cofibration. 
Since it is a cofibration from a cofibrant object, 
the object $M$ is projectively cofibrant object. 

(2) 
It is clear that the shift functor $[-1]$ preserves projectively cofibrant object. 
Since $\cocone(f)= \cone(f)[-1]$, 
it is enough to check that $\cone(f)$ is projective cofibrant. 
From the $\#$-exact sequence (\ref{1010 101}) 
we conclude that $\cone(f)$ is projectively cofibrant. 
\end{proof}

\noindent
{\bf $\bullet$ The category $\dgCat$ of dg-categories.}  
We denote by $\tuH^0(\cA)$ the homotopy category of the dg-category $\cA$. 
Namely this is a graded category 
such that  the objects of $\tuH^0(\cA)$ is the same with that of $\cA$ 
and the Hom module is given by $\tuH^0(\cA)(a,b):= \tuH^0\left(\cA^{\bullet}(a,b)\right)$. 
A dg-functor $f:\cA \to \cB$ is called a \textit{quasi-equivalence} 
if the induced functor $\tuH^{0}(f): \tuH^{0}(\cA) \to \tuH^{0}(\cB)$ gives an equivalence of categories 
and 
the induced map $\Hom_{\cA}^{\bullet}(a,b) \to \Hom_{\cB}^{\bullet}(f(a),f(b))$ is a quasi-isomorphism.

Let $f: \cA \to \cB$ be a dg-functor. 
We have an adjoint pair 
\begin{equation}\label{926 837}
f^{\ast}:= - \otimes_{\cA}\cB : \cC(\cA) \rightleftarrows \cC(\cB) :f_{\ast} 
\end{equation}
where $f_{\ast}$ is the restriction functor. 
Recall that the restriction functor $f_\ast$ is defined to be 
$(f_\ast M)(a):= M(fa)$ for $a \in \cA$ and $M \in \cC(\cB)$. 
Therefore $f_{\ast}$ preserves quasi-isomorphisms, injective cofibrations and projective cofibrations.  
Note that we have a canonical isomorphism $f^{\ast}(\awe)\cong (fa)^{\wedge}$ for $a\in \cA$.  

\begin{lemma}\label{926 838}
If we equip both $\cC(\cA) $ and $\cC(\cB)$ with the projective model structures. 
Then the above adjoint (\ref{926 837}) become a Quillen adjoint pair. 
Moreover if $f$ is quasi-equivalence, then the above adjoint pair (\ref{926 837}) gives a Quillen equivalence. 
\end{lemma}

The tensor product $\cA \otimes \cB$ of dg-categories $\cA$ and $\cB$ is defined in the following way.  
The sets of object is given by the product 
 $\Ob(\cA \otimes \cB) = \Ob(\cA) \times \Ob(\cB)$. 
We denote by $a\otimes b$ the object of $\cA\otimes\cB$ corresponding to $(a,b) \in \Ob(\cA)\times \Ob(\cB)$. 
The Hom-complex is given by  
 $(\cA\otimes \cB)^{\bullet}(a\otimes b , a'\otimes b') := \cA^{\bullet}(a,a') \otimes \cB(b,b')$. 
The composition and the identities are defined in an obvious way.

\begin{lemma}\label{103 815} 
Let $\cA$ be a locally  cofibrant dg-category and 
$f: \cB \to \cC$ a quasi-equivalence. 
Then the induced functor 
$1\otimes f: \cA\otimes \cB \to \cA \otimes \cC$ 
is a quasi-equivalence. 
\end{lemma}

We denote by $\dgCat$ the category of small dg-categories. 
It is proved by Tabuada \cite{Tab} 
that 
the category $\dgCat$ admits a structure of cofibrantly generated model category 
whose weak equivalences are the quasi-equivalences. 
(The characterization of fibration is known. However we don't use it.)

A dg-category  $\cA$  is called \textit{locally cofibrant} 
if for $a,b\in \cA$ the Hom complex $\Hom_{\cA}^{\bullet}(a,b)$ is projective cofibrant in $\cC(\kk)$. 
Note that by \cite[Lemma 2.3]{Toen} 
cofibrant dg-category is locally cofibrant. 
Using the same method of [loc. cite], 
we can prove the following lemma. 

\begin{lemma}\label{922 257}
Let $\cA$ be a locally cofibrant small dg-category 
and $f:\cA \rightarrowtail \cB$ a cofibration in $\dgCat$. 
Then for any $b\in \cB$,  
the dg $\cA$-module $f_{\ast}(b^{\wedge})$ is projective cofibrant. 
Dually for any $b\in \cB$, 
the left dg $\cA$ module $f_{\ast}({}^{\wedge}b)$ is projective cofibrant left dg $\cA$-module. 
\end{lemma}

Note that 
if 
$f':\cA\to \cB'$ be a dg-functor 
such that 
 left dg $\cA$-module $f'_{\ast}({}^{\wedge}b)$ is projective cofibrant for  $b\in \cB'$. 
Then for any full dg-subcategory $\cB \subset \cB'$ which contains the image of $f'$, 
the factorization $f: \cA \to \cB$ of $f':\cA \to \cB'$   
 has the same property.

\begin{lemma}\label{922 258}
Let $f:\cA\to \cB$ be a morphism in $\dgCat$ 
such that 
 left dg $\cA$-module $f_{\ast}({}^{\wedge}b)$ is projective cofibrant for $b\in \cB$. 
\textup{(}e.g. the essential image of some  cofibration $\cA \to \cB'$ in $\dgCat$.\textup{)}
Then the adjoint pair induced from $f$ 
\[
f^{\ast} : \cC(\cA) \rightleftarrows \cC(\cB) : f_{\ast} 
\]
is a Quillen adjoint where both $\cC(\cA)$ and $\cC(\cB)$ are equipped with the injective model structure. 
\end{lemma}

\begin{proof}
We prove that $f_{\ast}$ preserves injective (trivial) fibrations. 
It is clear that 
the restriction functor $f_{\ast}$ preserves term-wise surjections and weakly contractibility. 
Therefore,  
in view of Lemma \ref{922 142}, it is enough to check that 
$f_{\ast}$ preserves injective fibrant objects. 
By Lemma \ref{prel 1 lem 0} it is enough to check that 
$f_{\ast}\bfD(Q)$ is injectively fibrant for each semi-free left module $Q$.  
Since the functor $f_{\ast}$ commutes with limit and the dual $\bfD$ sends colimit to limit, 
we have only to check the case when $Q = {}^{\wedge}b$ for some $b \in \cB$, 
which follows from the assumption and the fact that $f_{\ast}$ and $\bfD$ are commute.
\end{proof}

\begin{lemma}\label{103 754}
Let $\cA$ be a dg-category, $\cB$  locally cofibrant dg-category 
and $J$ an injectively fibrant dg $\cA \otimes \cB$-modules. 
Then, for any $b\in \cB$,  the $\cA$-module $J(-\otimes b) $ is injectively fibrant. 
\end{lemma}

\begin{proof}
We prove the first statement. 
We can prove the second statement in the same  way. 
We denote by $\eta$ the dg-functor $\cA \to \cA \otimes \cB,\, a \mapsto a \otimes b$. 
Then the induced adjoint pair  $\eta^{\ast}: \cC(\cA) \rightleftarrows  \cC(\cA\otimes \cE): \eta_{\ast}$  
is given by $\eta^{\ast}: N \mapsto N \otimes b^{\wedge}$ 
and $\eta_{\ast}: X \mapsto X(-\otimes b)$. 
Since we assume that the dg-category $\cB$ is locally cofibrant, 
the functor $\eta^{\ast}$ preserves  injective cofibration and trivial injective cofibration. 
Hence the adjoint pair $\eta^{\ast}\dashv \eta_{\ast}$ is a Quillen pair 
when we equip the both categories $\cC(\cA),\cC(\cA\otimes \cE)$ with the injective model structures. 
In particular the $\eta_{\ast}$ functor preserves injectively fibrant objects. 
Therefore $J(-\otimes b)$ is injectively fibrant.  
\end{proof}

\noindent 
{\bf $\bullet$ The simplicial model category $\cCss(\cA)$ of dg $\cA$-modules.}
Through out this paper we equip the category $\cC(\kk)$ with the projective model structure. 
Then $\cC(\kk)$ with a canonical tensor product become  a monoidal model category, 
and 
the category $\cC(\cA)$  either with the projective model structure or with the injective model structure  
become a $\cC(\kk)$-enriched model category in the sense of \cite[Definition A. 3.1.5]{HTT}. 
The following lemma  is deduced from this fact.

\begin{lemma}\label{103 915}
(1) 
Let $M$ be a dg $\cA$-module. 
We equip $\cC(\cA)$ with the injective model structure. 
Then the tensor-Hom adjunction 
\[
-\otimes M :\cC(\kk) \rightleftarrows \cC(\cA) : {}_{\cA}(M,-) 
\] 
is a Quillen adjunction.  
In particular, the functor ${}_{\cA}(M,-)$ preserves weak equivalences between 
injectively fibrant objects. 

(2) 
Let $M$ be a projectively cofibrant dg $\cA$-module. 
We equip $\cC(\cA)$ with the projective  model structure. 
Then the tensor-Hom adjunction 
\[
-\otimes M :\cC(\kk) \rightleftarrows \cC(\cA) : {}_{\cA}(M,-) 
\] 
is a Quillen adjunction.  
In particular, the functor ${}_{\cA}(M,-)$ preserves weak equivalences.

(3) 
The Hom complex ${}_{\cA}(M,N)$ is a cofibrant object of $\cC(\kk)$, 
if $M$ is projectively cofibrant or $N$ is injectively fibrant. 
\end{lemma}

We denote by  $\textup{Ch}_{\geq 0}(\kk)$ the category of non-negatively graded chain complexes of $\kk$-modules. 
We denote by $\sMod \kk$ the category of simplicial $\kk$-modules. 
Both categories have a canonical monoidal structure. 
Then there is a monoidal functor $\Gamma: \textup{Ch}_{\geq 0}(\kk) \to \sMod \kk$ 
which gives an equivalence of the underlying categories. (See \cite{SS}.)

We denote $\cCss(\cA)$ the associated simplicial category of the dg-category $\cCdg(\cA)$. 
Namely, the simplicial category $\cCss(\cA)$ is a simplicial category objects of  which 
are the same with $\cC(\cA)$ and the mapping complex which is denoted by $\Map_{\cA}$ 
is given by 
\[
\Map_{\cA}(M,N) := \Gamma \tilde{\tau}^{\leq 0}\left( {}_{\cA}(M,N)\right)
\] 
where $\tilde{\tau}^{\leq 0}$ is the composition of the truncation functor $\tau^{\leq 0}: \cC(\kk) \to \cC_{\leq 0}(\kk)$ 
with a canonical equivalence $\tilde{(-)}:\cC^{\leq 0}(\kk) \xrightarrow{\sim} \textup{Ch}_{\geq 0}(\kk)$. 
Note that 
for a complex $L \in \cC(\kk)$, 
we have $\G \tilde{\tau}^{\leq 0}L = \tuZ^0(L)$. 
Therefore
the underlying category of the simplicial category $\cCss(\cA)$ is the category $\cC(\cA)$.

The simplicial category $\cCss(\cA)$ is tensored and cotensored over $\sSet$ in the following way. 
For a simplicial set $X$, we denote by $\kk[X]$ the free $\kk$-module generated by $X$. 
Let $\tilde{N}: \sMod \kk \to \textup{Ch}_{\geq 0}$ be the composition 
of the normalization functor $N: \sMod \kk \to \textup{Ch}_{\geq 0}\kk$ 
with a canonical equivalence $\textup{Ch}_{\geq 0}(\kk) \xrightarrow{\sim} \cC^{\leq 0}(\kk)$. 
Then 
for $M \in \cC(\cA)$ and $X\in \sSet$, 
the tensor $M\otimes X $ and the cotensor $M^{X}$ is given by  
\[M \otimes X := M \otimes \tilde{N}(\kk[X]),\quad M^{X} := M^{\tilde{N}(\kk[X])}.
\] 
We have a natural isomorphism of simplicial sets
\[
\Map_{\sSet}(X,\Map_{\cA}(M,N)) \cong \Map_{\cA}(M \otimes X, N) \cong \Map_{\cA}(M,N^{X}).
\]

Note that since every simplicial set $X$ is cofibrant, 
the tensor product $M\otimes X$ preserves projectively cofibrant objects  
and the cotensor product $M^X$ preserves  injectively fibrant objects. 

Let  $X$ be a finite simplicial set.  
Then the corresponding dg $\kk$-module $\tilde{N}(\kk[X])$ is a strictly perfect complex, 
i.e., 
a bounded complex each term of which is finite projective $\kk$-module. 
Therefore    we can construct 
the tensor product $M \otimes X$ and the cotensor product $M^{X}$ 
 from $M$ by taking shift, cone and direct summands.

The complexes  $\kk \otimes \Delta^0 = \tilde{N}(\kk[\Delta^0])$
 and  $\kk \otimes \Delta^1 = \tilde{N}(\kk[\Delta^1])$ which correspond to 
the standard $0$-simplex and the standard $1$-simplex are  given by  
\[
  \tilde{N}(\kk[\Delta^0]) = \kk, \quad 
 \tilde{N}(\kk[\Delta^1]) =\left(\kk\oplus \kk \oplus \kk[1], 
\begin{pmatrix} 0 & 0 & 1_{\kk} \\ 0 & 0 & -1_{\kk} \\ 0 & 0& 0 \end{pmatrix}  \right).
\]
where we consider $\kk$ as the complex concentrated at $0$-th degree. 
We denote by  
$\iota_a$ the canonical inclusion $\Delta^0 \cong \{a\} \to \Delta^1$ for $a=0,1$ 
and by $\textbf{uni}$ a unique map $\Delta^1 \to \Delta^0$. 
 Then the corresponding maps 
 are given by 
 \[
 \begin{split}
 \kk\otimes \iota_0 = {}^{t}(0, 1_{\kk}, 0), \kk\otimes \iota_{1} = {}^{t}(1_{\kk},0,0): 
 &\kk \otimes \Delta^0 \to \kk\otimes\Delta^1\\
 \kk \otimes \textbf{uni} = (1_{\kk},1_{\kk},0) : 
 & 
 \kk \otimes \Delta^1 \to \kk\otimes \Delta^0
 \end{split}
 \]

\begin{lemma}\label{106 1146}
Let $M$ be a dg $\cA$-module. 
Assume that in the under category $\cC(\cA)_{M/}$ 
two morphism $s_1: k_1 \to k_0$ and $s_2: k_{2} \to k_0$ are given. 
Then there exist morphisms $t_1: k_3 \to k_1$ and $t_2: k_3 \to k_2$ 
such that 
$s_1\circ t_1$ and $s_2\circ t_2 $ are weak homotopy equivalent. 
In other words, 
 any  diagram $k_1 \to k_0 \leftarrow k_2$ in $\cC(\cA)_{M/}$ 
 can be completed in a homotopy commutative diagram 
 \[
 \begin{xymatrix}{
 k_3 \ar[d] \ar[r] & k_2 \ar[d] \\
 k_1 \ar[r] & k_0 
 }\end{xymatrix}.
 \] 

Dually for morphisms 
$s_{1}: p_0 \to p_1 $ and $s_{2} : p_0 \to p_2$ 
 in the over category $\cC(\cA)_{/M}$ 
 there exist a morphism $t_1: p_1 \to p_3$ and $t_2: p_2 \to p_3$ 
 such that 
 $t_1 \circ s_1$ and $t_2 \circ s_2$ are weak homotopy equivalent. 
\end{lemma}

\begin{proof}
We prove the first statement. 
The second statement is proved in a dual way. 

We write $k_{i}:R \to K_{i}$ for $i=0,1,2$. 
We denote by the same symbol $s_{i}$ the morphisms $K_{i} \to K_{0}$ between co-domains 
induced from $s_{i}: k_{i} \to k_{0}$ for $i=1,2$. 
We define $K_{3}$ to be the co-cone $\cocone(-s_{1},s_{2})$ of the 
morphism $(-s_{1},s_{2}): K_{1} \oplus K_{2} \to K_{0}$. 
We define 
a morphism $k_{3}:M \to K_{3}$ of graded modules to be  
$k_{3}: = {}^{t}(0,k_{1},k_{2}): R \to K_{3}= K_{0}[-1]\oplus K_{1} \oplus K_{2}$. 
Using the calculation 
 $(s_{1},-s_{2}) \circ {}^{t}(k_{1},k_{2})=0$, 
 we check that $k_{3}$ is a chain map. 
We denote by $t_{i}: k_{3}\to k_{i}$ the morphism in $\cU$ induced from 
the $i+1$-th projection $\pr_{i+1}: K_{3}= K_{0}[-1]\oplus K_{1} \oplus K_{2} \to K_{i}$ for $i= 1,2$.   
Observe that 
\[
K_{0}^{\Delta^1} = \left( K_{0}[-1] \oplus K_{0} \oplus K_{0},  
\begin{pmatrix} 
-d_{K_{0}} & -1_{K_{0}} & 1_{K_{0}} \\
 0 & d_{K_{0}}  & 0 \\
 0 & 0& d_{K_{0}} 
 \end{pmatrix}
\right).
\]
Recall that 
the cotensor $k_0^{\Delta^1}$ is defined to be 
$M \xrightarrow{k_0} K_0 \xrightarrow{K_0^{\textbf{uni}}} K_0^{\Delta^1}$. 
We denote by $H$ the morphism 
$\textup{diag}(1_{K_{0}[-1]},s_{1},s_{2}): K_{3} \to K_{0}^{\Delta^1}$. 
Then it is a straightforward calculation to check the equation 
$k_0^{\Delta^1} = H \circ k_3$ of morphisms in $\cC(\cA)$. 
In other words, 
the morphism  $H$ induces a morphism $k_3 \to k_0$ in the under category $\cC(\cA)_{M/}$, 
which will be also denoted by $H$. 

Now we can check that 
 the following  diagram in $\cU$ is commutative: 
\[
\begin{xymatrix}{ 
k_{1} \ar[d]_{s_{1}} & k_{3} \ar[l]_{t_{1}} \ar[d]^{H} \ar[r]^{t_{2}} & k_2 \ar[d]^{s_{2}} \\
k_{0}  & k_{0}^{\Delta^1} \ar[l]^{k_0^{\iota_0}} \ar[r]_{k_0^{\iota_1}}  & k_{0},  
}
\end{xymatrix}\]
By Lemma \ref{925 1053} this diagram gives the desired homotopy commutative diagram.
\end{proof}

\begin{remark}
From the construction of $K_3$ in the above proof, 
we  see the following things:  
If the co-domains of $k_0,k_1,k_2$ are injectively fibrant, 
we can take $k_3$ whose co-domain is also injectively fibrant. 
if one of the morphism $k_a: M \to K_a$ (a= 1,2) is injectively cofibration, 
then 
we can take an object $k_3 \in \cC(\cA)_{M/}$ 
such that the corresponding morphism $k_3: M \to K_3$ in $\cC(\cA)$ is injectively cofibration. 

The similar things hold for the second statement. 
\end{remark}

\subsection{Homotopy limits for simplicial functors}

For the theory of $\infty$-categories,  we refer  to \cite{HTT}. 
Let $\bfA$ be a combinatorial simplicial model category and $\cI$ be a small simplicial category. 
We denote by $\bfA^{\cI}$ the category of simplicial functors $F: \cI \to \bfA$. 
Then by \cite[Proposition A.3.2.2 ]{HTT} the category $\bfA^{\cI}$ admits two structures of Quillen model category 
whose weak equivalences are the term-wise weak equivalences:
\begin{enumerate}
\item The projective model structure, whose fibrations are the term-wise fibrations. 

\item The injective model structure, whose cofibrations are the term-wise cofibrations. 
\end{enumerate} 

For an object $M \in \bfA$, 
we denote by $\ct_{\cI} M$ the constant functor $\cI \to \bfA$ 
with the value $M$. 
If the diagram category $\cI$ is clear from the context, 
we often omit the subscript $\cI$ and denote by $\ct M$. 

Let  $\ct: \bfA \to \bfA^{\cI}$ be the functor which sends $M $ to $\ct M$. 
Then by \cite[Proposition A. 3.3.8]{HTT} 
we have Quillen adjunction 
\[
\ct: \bfA \rightleftarrows \bfA^{\cI} : \lim 
\]
where 
we equip $\bfA^{\cI}$ with the injective model structure.  

Let $F$ be an object of $\bfA^{\cI}$ and $X$ be an object of $\bfA$. 
We say that 
a morphism $\eta: \ct X \to F$ in $\bfA^{\cI}$ \textit{exhibits} 
$X$ \textit{as a homotopy limit of} $F$ if,  
for some equivalence $F \xrightarrow{\sim} F'$ where $F'$ is injective fibrant in $\bfA^{\cI}$,  
the composite map $X \to \lim F \to \lim F'$ is weak equivalence.

Since $\lim$ preserves a weak equivalence between injectively fibrant objects, this condition is independent to the choice of $F'$. 
We can easily prove the following lemma.

\begin{lemma}\label{923 017}
(1) 
If $F$ is injective fibrant object of $\bfA$, 
then the canonical map $\ct \lim F \to F$ exhibits $\lim F$ as a homotopy limit of $F$. 

(2) 
Let $\psi:X' \to X$ be a weak equivalence of $\bfA$. 
Then the morphism $\eta:\ct X \to F$ exhibits $X$ as a homotopy limit of $F$ 
if and only if the composite morphism $\eta\circ \ct(\psi):\ct  X' \to \ct X \to F$ 
exhibits $X'$ as a homotopy limit of $F$.

(3) 
Let $F \xrightarrow{\sim} G$ be a weak equivalence in $\bfA^{\cI}$. 
Assume that a morphism $\eta: \ct X \to F $ is given for some $X$.  
Then the morphism 
$\eta$ exhibits $X$ as a homotopy limit of $F$ if and only if 
the composite morphism $\ct X \to F \to G$ exhibits $X$ as a homotopy limit of $G$. 
\end{lemma}

When we study limits in category theory, two important notions are  cofinality of functors and filtered categories. 
For the definition of filtered $\infty$-category, we refer to \cite[Definition 5.3.1.7]{HTT}. 
We call an $\infty$-category $\cC$ \textit{co-filtered} if the opposite $\infty$-category $\cC^{\textup{op}}$ is filtered.  
A $\infty$-category version of cofinality was  introduced by Joyal \cite[Definition 4.1.1.1]{HTT}. 
In the present paper we mainly use the dual notion left cofinality. 
An $\infty$-functor $F:A \to B$ is called \textit{left cofinal} 
if  the opposite functor 
$F^{\textup{op}}: A^{\textup{op}} \to B^{\textup{op}}$ 
 is a cofinal $\infty$-functor.
A simplicial functor $\phi: \cI \to \cJ$ between fibrant simplicial categories 
is  called \textit{homotopy left cofinal}  
if the simplicial nerve $\sner(\phi): \sner(\cI) \to \sner(\cJ)$ is left cofinal. 
We note that Dwyer-Kan equivalence of fibrant simplicial categories is an important example of homotopy left cofinal functor. 

For  simplicial functors $\phi: \cI \to \cJ$ and $F: \cJ \to \bfA$, 
we often denote by $F|_{\cI}$ the composite functor $F\circ \phi: \cI \to \cJ \to \bfA$.

\begin{lemma}\label{923 043}
Let $\cI \to \cJ$ be a homotopy left cofinal functor between fibrant simplicial categories. 
Assume that every object of $\bfA$ is cofibrant. 
Then 
the morphism $\eta: \ct_{\cJ}X \to F$ exhibits $X$ as a homotopy limit of $F$ 
if and only if $\eta|_{\cI}: \ct_{\cI} X \to F|_{\cI}$ exhibits $X$ as a homotopy limit of $F|_{\cI}$. 
\end{lemma}

\begin{proof}
First note that 
in the case when $X$ and the every value of $F$ are fibrant-cofibrant in $\bfA$ 
this lemma is proved by  \cite[Proposition 4.1.1.8, Theorem 4.2.4.2]{HTT}.
We also note that 
by the same consideration of \cite[Remark A.2.6.8]{HTT} 
we see that 
each value $F(i)$ of an injectively fibrant object $F$ of $\bfA^{\cJ}$ is a fibrant object of $\bfA$. 
Hence it follows from the assumption that each value of injectively fibrant object $F$ is a fibrant-cofibrant object of $\bfA$.

We assume that 
$\eta$ exhibits $X$ as a homotopy limit of $F$. 
We take an injective trivial cofibration $F \stackrel{\sim}{\rightarrowtail} F'$ 
 with $F'$ injectively fibrant. 
 The map $X \to \lim F'$ induced from the composite map $\ct X \to F \to F'$ 
 a weak equivalence. 
Let $F|_{\cI} \xrightarrow{\sim} G$ be a weak equivalence with $G$ injectively fibrant. 
Since the restriction functor $\bfA^{\cJ} \to \bfA^{\cI}$ preserves injective  trivial cofibrations, 
by the lifting property 
we obtain the commutative diagram 
\[\begin{xymatrix}{
\ct_{\cI}X \ar[d]_{\wr} \ar[r] & *++{F|_{\cI}} \ar@{>->}[d]^{\wr} \ar[r]^{\sim} & 
G \ar@{->>}[d]\\ 
\ct_{\cI}(\lim_{\cJ} F') \ar[r] & F'|_{\cI} \ar[r] \ar[ur]^{\sim} & \emptyset 
}\end{xymatrix}\] 
Since $\ct_{\cJ} (\lim_{\cJ} F') \to F'$ exhibits $\lim F'$ as a homotopy limit of $F'$, 
by the known results mentioned above 
the map $\ct_{\cI}(\lim_{\cJ} F') \to F|_{\cI}$ exhibits $\lim F'$ as a homotopy limit. 
Hence the map $\lim_{\cJ} F' \to \lim_{\cI} G$ induced from the composite map $\ct_{\cI}\lim F' \to F'|_{\cI} \to G$ 
is weak equivalence. 
Consequently we see that the map $\ct_{\cI} X \to \lim_{\cI} G$ is weak equivalence.

We assume that 
$\eta|_{\cI}$ exhibits $X$ as a homotopy limit of $F|_{\cI}$. 
It is easy to see that 
we have only to check the case when 
a weak equivalence $F \to F'$  is an injective trivial cofibration. 
Take a weak equivalence  $F|_{\cI} \stackrel{\sim}{\rightarrowtail} G$ 
with $G$ injectively fibrant 
in $\bfA^{\cI}$. 
Then as before 
by the lifting property 
we obtain the commutative diagram 
\[\begin{xymatrix}{
\ct_{\cI}X \ar[d] \ar[r] & *++{F|_{\cI}} \ar@{>->}[d]^{\wr} \ar[r]^{\sim} & 
G \ar@{->>}[d]\\ 
\ct_{\cI}(\lim_{\cJ}  F') \ar[r] & F'|_{\cI} \ar[r] \ar[ur]^{\sim} & \emptyset 
}\end{xymatrix}\] 
Since $F'$ is taken to be injectively  fibrant, 
the canonical map $\ct_{\cJ}(\lim_{\cJ} F') \to F'$ exhibits $\lim F'$ as a homotopy limit of $F'$. 
Therefore by the known result mentioned above, 
the canonical map $\ct_{\cI}(\lim_{\cJ}F') \to F|_{\cI}$ exhibits $\lim_{\cJ} F'$ as a homotopy limit of $F'|_{\cI}$. 
Hence  the map $\lim_{\cJ} F' \to \lim_{\cI} G$ induced from the above diagram 
is weak equivalence. 
By the assumption 
the map $X \to \lim_{\cI} G$ is weak equivalence. 
Consequently, the map $X \to \lim_{\cJ}F'$ is weak equivalence as desired. 
\end{proof}

Let $f: \bfA \rightleftarrows \bfB :g$ be a Quillen adjunction between 
combinatorial simplicial model categories. 
We define $f^{\cI}: \bfA \to \bfB$ to be the functor which sends 
$G\in \bfA$ to $f\circ G$ and define $g^{\cI}: \bfB^{\cI} \to \bfA^{\cI}$ similarly. 
Then 
$f^{\cI} : \bfA^{\cI} \rightleftarrows \bfB^{\cI} : g^{\cI}$ is a Quillen adjunction 
where we equip $\bfA$ and $\bfB$ with the injective model structures.

\begin{lemma}\label{921 456}

Assume that the simplicial functor $g$ preserves weak equivalence. 
If a morphism $\eta:\ct_{\cI}X \to F$ in $\bfA^{\cI}$ exhibits $X$ as a homotopy limit of $F$, 
Then  $g^{\cI}(\eta) : \ct_{\cI}g(X) \to g^{\cI}(F)$ exhibits $g(X)$ as a homotopy limit of $g^{\cI}F$.  
\end{lemma}

\begin{proof}
We take an injectively trivial cofibration  
 $g^{\cI}(F) \stackrel{\sim}{\rightarrowtail} G $  with $G$ injectively fibrant. 
 Let $F\xrightarrow{\sim} F'$ be a weak equivalence in $\bfA^{\cI}$ with $F'$ injectively fibrant. 
Since $g^{\cI}(F')$ is injectively fibrant, 
by the lifting property 
we obtain the commutative diagram 
\begin{equation}\label{921 449}
\begin{xymatrix}{
\ct_{\cI}g(X) \ar[r] \ar[dr] & *++{g^{\cI}(F)} \ar@{>->}[d]_{\wr} \ar[r]^{\sim} & g^{\cI}(F') \\ 
&G \ar[ur]& 
}
\end{xymatrix}  
\end{equation}
 By the assumption on $g$, the induced morphism $g^{\cI}(F) \to g^{\cI}(F')$ 
 is weak equivalence. 
Therefore the morphism $G \to g^{\cI}(F')$ in the above diagram is also weak equivalence. 
Since $\lim$ is a right Quillen functor, 
it preserves a weak equivalence between fibrant object. 
Thus taking limits of the above diagram (\ref{921 449}) we see that the induced morphism 
$g(X) \to G$ is weak equivalence. 
\end{proof}

In the case when a diagram category $\cI$ is of special type, 
we can easily compute homotopy limits of functors $F: \cI \to \bfA$. 
We set $\cI: = (\ZZ_{\geq 1})^{\textup{op}}$ 
where we consider the totally ordered set $\ZZ_{\geq 1}$ 
as a category in the standard way. 
Namely,  
the objects of $\cI$ are positive integers 
and  morphisms $\textbf{geq}^{m,n}: m \to n$ are $m \geq  n$.  
 Using \cite[Theorem 15.10.12]{Hirsch} 
 we prove the following lemma. 
 Here we equip $\cC(\cA)$ with the projective model structure.

\begin{lemma}\label{921 514}
Let $\cA$ be a small dg category. 
Let $F$ be an object of $\cCss(\cA)^{\cI}$.  
Assume that for $n\geq 1$ 
the  morphism $F(\textup{\textbf{geq}}^{n+1,n}): F(n+1) \to F(n) $ is a projective fibration. 
Then the canonical map $\ct\lim F \to F$ exhibits $\lim F$ as a homotopy limit.  
\end{lemma}

We often use the following characterization of homotopy left cofinality. 
\begin{theorem}[Theorem 4.1.3.1 \cite{HTT}]\label{cofinal thm}
Let $f: \cC \to \cD$ be a map of simplicial sets, where $\cD$ is an $\infty$-category. 
The following conditions are equivalent: 

\begin{enumerate}
\item[(1)] $f$ is  left cofinal. 

\item[(2)] For every $D\in \cD$, 
the simplicial set $f_{/D}:=\cC \times_{\cD} \cD_{/D}$ is weakly contractible.
\end{enumerate}
\end{theorem}

\begin{lemma}\label{cofinal lem}
Let $\cC$ be a co-filtered $\infty$-category 
and $\cD\subset \cC$ a full sub $\infty$-category. 
Assume that for each object $c$ of  $\cC$ there exists a morphism $d\to c$ in $\cC$ 
the domain of which is an object of $\cD$. 
Then $\cD$ is left cofinal in $\cC$.  
\end{lemma}

\begin{proof}
By the above Theorem \ref{cofinal thm} 
it is enough to prove that the simplicial set  $\cD_{/c}$ is weakly contractible 
for each $c \in \cC$. 
By the assumption 
$\cD_{/c}$ is non-empty. 
Therefore we have only to show that 
any map $\sigma:K \to \cD_{/c} $ of simplicial sets from a finite simplicial set $K$ 
factors through the canonical map $\iota:K \to {}^{\lhd}K$.
\[\begin{xymatrix}{
K \ar[d]_{\iota} \ar[r]^{\sigma} & \cD_{/c} \\
\leftcone K \ar@{-->}[ur] &  
}\end{xymatrix}\] 
Recall that $\cD_{/c} = \cD \times_{\cC} \cC_{/c}$. 
We denote by $\sigma_{1}, \sigma_{2}$ the composite morphism $\textup{pr}_{1} \circ \sigma,\textup{pr}_{2}\circ \sigma$ 
where $\textup{pr}_{1},\textup{pr}_{2}$ are the first and the second projections. 
\[\begin{xymatrix}{
K \ar@/^1pc/[drr]^{\sigma_2} \ar@/_1pc/[ddr]_{\sigma_1} \ar[dr]^{\sigma} && \\ 
&\cD_{/c} \ar[r]^{\pr_2}\ar[d]_{\pr_1} & \cC_{/c}\ar[d]^{\Upsilon} \\
&\cD \ar[r] & \cC 
}\end{xymatrix}\]

Since $\cC$ is co-filtered, so is $\cC_{/c}$. 
Therefore $\sigma_{2}$ factors through $\iota$. 
In other word there exists a map $\psi: \leftcone K \to \cC_{/c}$ of simplicial sets 
such that $\sigma_{2} = \psi \circ \iota$. 
We denote by $f:c'\to c$ the image of the cone point $-\infty $ of $\leftcone K$ by $\psi$. 
We may consider $f$ as a morphism of $\cC$ the co-domain of which is $c$. 
We denote by $c'$ the domain of the morphism $f$. 
By assumption there exists a morphism $g: d \to c'$ in $\cC$ the domain of which belongs to $\cD$.
We may consider $g$ gives a morphism in $\cC_{/c}$ from some object $h : d \to c$ to $c' \to c$. 
By abusing notation we denote by $g$ the corresponding map $g: \Delta^{1} \to \cC_{c/}$  
of simplicial sets. 

Since the restriction $g|_{\{1\}}$ to $\{1\} \subset \Delta^{1}$ coincide with 
the restriction $\sigma|_{\{-\infty\}}$  to $\{-\infty\} \subset \leftcone K$, 
the maps $\sigma_{2}$ and $g$ induce the map 
$\Sigma: \Delta^{1} \bigsqcup_{\{1\}} \leftcone K \to \cC_{/c}$ of simplicial sets.
By \cite[Lemma 2.4.3.1]{HTT} the inclusion $\eta: \Delta^{1} \bigsqcup_{\{1\}} \leftcone K \to \Delta^{1} \star K$ 
is inner anodyne. 
Therefore there exists a map $\Sigma':\Delta^{1} \star K \to \cC$ 
such that $\Sigma = \Sigma' \circ \eta$. 
We denote by $\tau_2$ the composite map of $\Sigma"$ with the canonical map $\leftcone K \cong \{0\} \star K \to \Delta^{1} \star K$. 

Since the restriction $\tau_2|_{K}$ to $K$ of $\tau$ coincide with $\sigma_{2}$
 and the subcategory $\cD$ of $\cC$ is full, 
the image of  composite map $\Upsilon \circ \tau_{2}$ of $\tau_{2}$ with  the domain functor $\Upsilon:\cC_{/c} \to \cC$ 
contained in $\cD$. 
Therefore the map $\Upsilon \circ \tau_{2}$ is decomposed 
into some map $\tau_{1}: \leftcone K \to \cD$ followed by the inclusion $\cD \subset \cC$. 
Since the restriction $\tau_2|_{K}$ to $K$ of $\tau$ coincides with $\sigma_{2}$, 
the restriction $\tau_{1}|_{K}$ to $K$ coincide with $\sigma_1$. 
Therefore the maps $\tau_1,\tau_2$ gives a map $\tau : \leftcone K \to \cD_{/c}$  
such that $\tau\circ \iota = \sigma$. 
\end{proof}

\section{Derived bi-duality via homotopy limit}

\subsection{Derived bi-commutator categories and dualities} 
Let $\cA$ be a small locally cofibrant dg-category and   
 $\cJ \subset \cD(\cA)$  a small full subcategory. 
Following Efimov \cite{Efi} we construct the derived bi-commutator category $\dgBic_{\cA}(J)$ 
in the following steps.   

\begin{const}\label{103 444}

\begin{enumerate}

\item[(I).] 
First we choose 
an injectively fibrant 
representative $J$ of each object of $\cJ$, 
and denote by $\ulcJ_1$ the full subcategory of $\cC(\cA)$ consisting of those representatives.  
Then   
we define the opposite of endomorphism dg-category 
$\cE:= \cE\text{nd}_{\cA}(\ulcJ_1)^{\textup{op}}$ to be 
\[
\textit{Ob}(\cE) := \textit{Ob}(\ulcJ_1), \quad \Hom_{\cE}([J],[I]) := {}_{\cA}(I,J). 
\]
where $[I]$ and $[J]$ are objects of $\cE$ which correspond to $I,J \in \textit{Ob}(\ulcJ_1)$ respectively. 
Note that by \cite[Lemma 3.1]{Efi} the quasi-equivalence class of $\cE$ is only depends on $\cJ$ 
and is independent to the choice of the representative $\ulcJ_1$.
Note also that $\cE$ is a locally cofibrant dg-category. 

\item[(II).] 
Next 
we define a dg $\cA \otimes \cE$-module $\underline{\cJ}_{2}$ to be 
$\underline{\cJ}_{2}(a\otimes [J]) := J(a)$ for $a \in \cA$ and $[J] \in \cE$. 
We call $\ulcJ_2$ the \textit{diagonal module} associated to $\ulcJ_1$. 
Note that 
the diagonal module $\ulcJ_2$ is determined by $\ulcJ_1$.

\item[(III).] 
Finally 
we take a fibrant replacement  
$\ulcJ_2 \stackrel{\sim}{\rightarrow} \underline{\cJ}$ of $\underline{\cJ}_{2}$ in $\cC(\cA \otimes \cE)$. 
\end{enumerate} 

\end{const}

\begin{definition}
We define the bi-commutator category 
$\dgBic_{\cA}(\cJ)$ 
of $\cJ$ over $\cA$ to be 
\[
\textit{Ob}\left(\dgBic_{\cA}(\cJ) \right)  := \textit{Ob}\cA, \quad 
{\dgBic_{\cA}(\cJ)}(a,b):= {}_{\cE}(\underline{\cJ}(b\otimes-),\underline{\cJ}(a\otimes -)).
\]
\end{definition} 
The identity on the objects induces a canonical functor $\iota: \cA \to \dgBic_{\cA}(\cJ)$.  
We denote by $\BBic_{\cA}(\cJ)$ the object of $\Ho(\dgCat)$ 
corresponding to $\dgBic_{\cA}(\cJ)$. 
Note that 
by \cite[Lemma 3.2]{Efi} the object $\iota : \cA \to \BBic_{\cA}(\cJ)$ 
of $\Ho(\dgCat_{\cA/})$ only depends on 
the subcategory $\cJ \subset \cD(\cA)$.

We introduce the duality functors induced by the $\cA\otimes \cE$-module $\underline{\cJ}$: 
\[
\begin{split}
D:= {}_{\cA}(-,\underline{\cJ}): \cC(\cA) \to \cC(\cE){}^{\textup{op}}\\
D':={}_{\cE}(-,\underline{\cJ}): \cC(\cE) \to \cC(\cA)^{\textup{op}}
\end{split}
\]
We set $S := D'D$.  
 Note that  $S(M)$ computes the derived bi-duality of $M$ over $\cJ$  
(Corollary \ref{103 824}). 
Namely 
the module $S(M)$ is isomorphic to $\RR\Hom_{\cE}(\RR\Hom_\cA(M,\cJ),\cJ)$ in the derived category $\cD(\cA)$. 

There exist natural transformations 
$
\epsilon: 1_{\cC(\cA)} \to S 
$
induced from the evaluation morphisms. 
Namely,  
 for $a \in \cA $ and $M \in \cC(\cA)$, 
the morphism \[
\epsilon_{M}(a): M(a) \to  S(M)(a)= 
{}_{\cE}({}_{\cA}(M,\underline{\cJ}),\underline{\cJ}(a\otimes -))\]
 of cochain complexes 
is given 
for a homogeneous elements $m \in M(a)$ 
by 
\[
[\epsilon_{M}(a)(m)](e) : {}_{\cA}(M,\underline{\cJ}(-\otimes e)) \to \underline{\cJ}(a\otimes e),
\quad f \mapsto (-1)^{(\deg m)( \deg f) }f(m) 
\] 
where
$e \in \cE$ and 
 $f \in {}_{\cA}(M,J(-\otimes e))$ is a homogeneous element.

We can check that the family of morphisms
\[
\cA(a,b) = b^{\wedge}(a) \xrightarrow{\epsilon_{b^{\wedge}}(a)} 
S(b^{\wedge})(a) = {}_{\cE}(\ulcJ(b\otimes-),\ulcJ(a\otimes -)) = \dgBic_{\cA}(J)(a,b)
\]
gives a dg-functor $\iota: \cA \to \dgBic_{\cA}(\cJ)$. 
We immediately see the following lemma. 
\begin{lemma}\label{1010 321}
The functor $\iota$ is quasi-equivalence if and only if $\epsilon_{b^{\wedge}}$ is quasi-isomorphism for $b\in \cA$. 
\end{lemma}

%
%

To state the main theorem of this  paper, we prepare notations. 
We denote by $\langle \cJ \rangle^{\cD}$ the smallest thick subcategory of $\cD(\cA)$ 
containing $\cJ$. 
Namely 
it is  
the full subcategory of $\cD(\cA)$ which consists of those objects 
which are obtained by taking direct summands, shifts and cones finitely many times 
from the image of $\cJ$ by the canonical functor $\cC(\cA) \to \cD(\cA)$. 
We denote by $\langle \cJ \rangle^{\cC}$ the full subcategory of $\cC(\cA)$ 
consisting those objects which are sent to  $\langle \cJ \rangle^{\cD}$ 
by the canonical functor $\cC(\cA) \to \cD(\cA)$. 

We fix a dg $\cA$-module $M$. 
We define a  full subcategory $\cU:= \cU^{M}_{J}$ of
 $\langle \cJ \rangle^{\cC}_{M/}$ to be the full subcategory $(\langle \cJ \rangle^{\cC}_{M/})^{\circ}$ 
 consisting of fibrant-cofibrant objects. 
Namely $\cU$ is the full subcategory of $\langle \cJ\rangle^{\cC}_{M/}$ 
consisting those object $k:M \to K$ such that 
the morphism  $k$ is a injective cofibration and the co-domain $K$ is injective fibrant. 
Let $\Gamma: \cU \to \cC(\cA)$ be the co-domain functor. 
\[
\Gamma: \cU \to \cC(\cA), \quad [k: M \to K] \mapsto K.
\]
We observe 
that the assignment $\kappa: k \mapsto k$ 
gives 
a natural transformation $\kappa: \ct_{\cU} M \to \G$
where 
$\ct_{\cU}M:\cU \to \cC(\cA), k \mapsto M$ is the constant functor with the value $M$. 
\[
\begin{xymatrix}{
(\ct_{\cU}M)(k) \ar[d]_{\kappa_k} \ar@{=}[r] & M \ar[d]^{k} \\
\G(k) \ar@{=}[r] & K
}\end{xymatrix}
\]
 
We have the commutative diagram 
\[
\begin{xymatrix}{
\ct_{\cU} M \ar[d]_{\kappa} \ar[r]^{\epsilon_{M}} & \ct_{\cU} S(M) \ar[d]^{S(\kappa)} \\
\Gamma \ar[r]_{\epsilon_{\Gamma}}^{\sim} & S\circ \Gamma 
}\end{xymatrix}\]

Then 
the main theorem  is the following. 

\begin{theorem}\label{bic holim thm}
Let $\cA$ be a locally cofibrant dg-category, $M$ a dg $\cA$-module 
and $\cJ$ a set of objects of $\cD(\cA)$.
We construct a derived bi-duality functor $S(-)$ over $\cJ$ 
by the construction \ref{103 444}. 
Assume that the evaluation map $\epsilon_M: M \to S(M)$ 
is injectively cofibrant. 
Then  we have a quasi-isomorphism
\[
S(M) \simeq \holim_{ \cU }\Gamma
\]
More precisely, 
the morphism $S(\kappa) : \ct_{\cU}S(M) \to S\circ \G$ exhibits 
$S(M)$ as a homotopy limit of $S\circ \G$.  
Since the morphism $\epsilon_{\G}: \G \to S\circ \G$ is a weak equivalence 
(Lemma \ref{main mis lemma 2}.(1)), 
we  obtain the above quasi-isomorphism. 
\end{theorem}

The assumption that $\epsilon_M$ is injective cofibration is harmless, 
because of the following lemma. 

\begin{lemma}\label{103 450}
Let $\cA $ and $\cJ$ as in the above theorem. 
We fix a set $\cM$ of objects of $\cC(\cA)$. 
Then 
by the construction \ref{103 444} 
we can construct a derived bi-duality functor $S(-)$ over $\cJ$ 
such that the evaluation map $\epsilon_M$ is injective cofibration for $M \in \cM$.  
\end{lemma} 

\begin{proof}
We claim that 
there exists 
a injective cofibration  $ M \hookrightarrow  N_M$ 
with $N_M$ injectively fibrant and weakly contractible. 
Indeed 
first we take a fibrant replacement  $M \stackrel{\sim}{\rightarrowtail} M'$ of $M$ 
by a injectively cofibration. 
Let $M' \to \cone(1_{M'})$ be the canonical embedding. 
Then by Lemma \ref{101 1250} the composite map $M \to M' \to \cone(1_{M'})$ 
satisfies the desired property.

In the step (I) 
first we  choose injectively fibrant representatives of $\cJ$ and construct $\ulcJ_1'$.  
Next  for some $J_0' \in \ulcJ_1'$  we set $J_0:= J_0'\oplus \bigoplus_{M \in \cM}  N_M$.  
Then we obtain  $\ulcJ_1$  
such that  for each $M \in \cM$,  
 there exists an injective cofibration $\iota: M \rightarrowtail J_0$ for some object $J_0 $ of $\ulcJ_1$.
Let $\ulcJ_2$ is the diagonal module of $\ulcJ_0$ and 
$\ulcJ$ an injectively fibrant replacement 
by an injective trivial cofibration $\ulcJ_2 \stackrel{\sim}{\rightarrowtail} \ulcJ$. 
Then for any $M\in \cM$ 
there exists an injective cofibration $M \cofright \ulcJ(-\otimes [J_0])$.  
Now it is easy to see by the explicit formula of the evaluation map 
that 
$\epsilon_M: M \to S(M)$ is an injectively cofibration. 
\end{proof}

\subsection{A proof of Theorem \ref{bic holim thm}}

First we study basic properties of duality $D,D'$. 

\begin{lemma}\label{main mis lemma 1}
(1) For each $e \in \cE$ the dg $\cA$-module $\ulcJ(- \otimes e) $ is injectively fibrant. 
Dually for each $a\in \cA$ the dg $\cE$-module $\ulcJ(a\otimes -)$ is injectively fibrant. 
 
(2) 
The functors $D,D'$ preserves quasi-isomorphisms. 

(3) 
The functor $S$ preserves quasi-isomorphisms, projectively fibrations and injective cofibrations.   

(4)
If we equip $\cC(\cA)$ with the injective model structure and 
$\cC(\cE)^{\textup{op}}$ with the opposite model structure of  the  projective model structure in $\cC(\cE)$, 
then 
the adjoint pair 
\[D: \cC(\cA) \rightleftarrows \cC(\cE)^{\textup{op}}:D'\] 
is a Quillen adjunction. 
 
\end{lemma}

\begin{proof}
(1) 
follows from Lemma \ref{103 754}.  

(2) and (3) follow from (1) and Lemma \ref{prel 1 lem 0}.

(4) 
By (2) it is enough to prove that $D$ sends injective cofibration  to 
projective fibration.  
This follows (1) and Lemma \ref{prel 1 lem 0}.  
\end{proof}

\begin{corollary}\label{103 824}
For any dg $\cA$-module $M$,  
the module $S(M)$ computes the derived bi-duality over $\cJ$.  
\[S(M) \simeq \RR\Hom_{\cE}(\RR\Hom_{\cA}(M,\cJ),\cJ).\] 
\end{corollary}

We denote by $\dgPerf \cE$ the category $\langle e^{\wedge} \mid e \in \cE \rangle^{\cC}$ of perfect dg $\cE$-modules. 
There exists a natural transformation $\epsilon': 1_{\cC(\cE) } \to DD'$ 
induced from evaluation maps.

\begin{lemma}\label{main mis lemma 2}
(1) If a dg $\cA$-module $K$ belongs to $\langle \cJ \rangle^{\cC}$, 
then the morphism $\epsilon_{K}: K \to S(K)$ is a quasi-isomorphism. 

(2) 
If a dg-$\cE$-module $P$ belongs to $\dgPerf\cE$, 
then the morphism $\epsilon'_P: P \to DD'(P)$ is a quasi-isomorphism. 

(3) 
The functors $D,D'$  gives a contravariant Dwyer-Kan equivalence 
between $\left(\langle \cJ \rangle^{\cC}\right)^{\circ}$
 and $(\dgPerf \cE)^{\circ}$. 
\end{lemma}

\begin{proof}
(1) 
It is enough to prove the case when $K =J$ for some $J \in \ulcJ_1$.  
It is easy to see  that 
the dg $\cE$-module 
${}_{\cA}(J,\underline{\cJ}_{2})$ is isomorphic to $[J]^{\wedge}$
 and that the evaluation morphism 
\[
\varepsilon_{J}: J \to {}_{\cE}({}_{\cA}(J,\underline{\cJ}_{2}),\underline{\cJ}_{2}) 
\]
is an isomorphism of dg $\cA$-modules. 
Let 
$\phi: 
\underline{\cJ}_2 \xrightarrow{\sim} \underline{\cJ} $  be the replacement morphism. 
We can check that 
the morphism
 $\psi_1: {}_{\cE}({}_{\cA}(J,\underline{\cJ}_{2}),\underline{\cJ}_{2}) \to {}_{\cE}({}_{\cA}(J,\underline{\cJ}_{2}),\underline{\cJ})$ 
 of dg $\cA$-modules 
 induced from $\phi$ is isomorphic to 
 the morphism $\phi(-\otimes J): \underline{\cJ}_{2}(-\otimes [J])  \to \underline{\cJ}(-\otimes [J])$ 
 under the isomorphisms 
 \[
 {}_{\cE}({}_{\cA}(J,\underline{\cJ}_{2}),\underline{\cJ}_{2}) \cong \underline{\cJ}_{2}(-\otimes[J]),\quad
 {}_{\cE}({}_{\cA}(J,\underline{\cJ}_{2}),\underline{\cJ}) \cong \underline{\cJ}(-\otimes [J]). 
 \]
 This implies that $\psi_1$ is a quasi-isomorphism. 
Let $J$ be an object of $\ulcJ_1$. 
By the definition of the diagonal module $\ulcJ_2$, 
the $\cA$-module $\ulcJ_2(-\otimes [J])$ is isomorphic to the $\cA$-module. 
Recall that all objects of $\ulcJ_1 \subset \cC(\cA)$ are  taken to be injectively fibrant. 
Hence the  $\cA$-module $\ulcJ_{2}(-\otimes e)$ is injectively fibrant for each $e\in \Ob(\cE)$. 
By Lemma   \ref{main mis lemma 1}, 
 the  $\cA$-module $\ulcJ(-\otimes e)$ is injectively fibrant for each $e\in \Ob(\cE)$. 
Therefore 
by Lemma \ref{103 915}   
the induced map ${}_{\cA}(J,\ulcJ_2) \to {}_{\cA}(J,\ulcJ)$ is quasi-isomorphism of $\cE$-modules.  
By Lemma \ref{main mis lemma 1}, 
we see that 
the induced morphism $\psi_2:{}_{\cE}({}_{\cA}(J,\underline{\cJ}) \ulcJ) \to {}_{\cE}({}_{\cA}(J,\ulcJ_{2}),\ulcJ)$ 
is a quasi-isomorphism. 

We can check that $\psi_2 \circ \epsilon_{J} = \psi_{1} \circ \varepsilon_{J}$.  
\begin{equation}\label{103 552}
\begin{xymatrix}{
J \ar[d]_{\varepsilon_{J}}^{\sim} \ar[r]^{\epsilon_{J}} & {}_{\cE}({}_{\cA}(J,\ulcJ),\ulcJ) \ar[d]_{\sim}^{\psi_{2}}\\
{}_{\cE}({}_{\cA}(J,\ulcJ_{2}),\ulcJ_{2}) \ar[r]_{\psi_{1}}^{\sim} & {}_{\cE} ({}_{\cA}(J, \ulcJ_{2}),\ulcJ)
}\end{xymatrix}\end{equation} 
Since the morphisms $\varepsilon_J,\psi_1,\psi_2$ are quasi-isomorphisms, 
we conclude that $\epsilon_J$ is a quasi-isomorphism. 

(2) It is enough to check the case $P = [J]^{\wedge}$ for some $J \in \ulcJ_1$. 
In the similar way as in  (1), 
we obtain the  commutative diagram similar to the above diagram (\ref{103 552}) 
 \[
\begin{xymatrix}{
[J]^{\wedge} \ar[d]_{\varepsilon'_{[J]^{\wedge} }}^{\sim} \ar[r]^{\epsilon'_{[J]^{\wedge}}} & 
{}_{\cA}({}_{\cE}([J]^{\wedge} ,\ulcJ),\ulcJ) \ar[d]_{\sim}^{\psi'_{2}}\\
{}_{\cA}({}_{\cE}([J]^{\wedge} ,\ulcJ_{2}),\ulcJ_{2}) \ar[r]_{\psi'_{1}}^{\sim} & 
{}_{\cA} ({}_{\cE}([J]^{\wedge}, \ulcJ_{2}),\ulcJ)
}\end{xymatrix}\]
By Yoneda Lemma we have an isomorphism of $\cA$-modules  
${}_{\cE}([J]^{\wedge}, \ulcJ_2) \cong \ulcJ_2(-\otimes[J]) \cong J$. 
Observe that 
the left vertical arrow $\varepsilon_{[J]^{\wedge}}$ 
coincide with the isomorphism 
\[
[J]^{\wedge} \cong {}_{\cA} (J, \ulcJ_2) \cong {}_{\cA}({}_{\cE}([J]^{\wedge} ,\ulcJ_2),\ulcJ_2).     
\]
Hence it  is an isomorphism.  
Since $[J]^{\wedge}$ is projectively cofibrant,   
the induced map ${}_{\cE}([J]^{\wedge}, \ulcJ_2) \to {}_{\cE}([J]^{\wedge}, \ulcJ)$ 
is a quasi-isomorphism by Lemma \ref{103 915}. 
Hence by Lemma \ref{main mis lemma 1} the left vertical morphism $\psi_2'$ is a quasi-isomorphism. 
We can prove that 
the bottom arrow $\psi_1'$ is a quasi-isomorphism 
in the same way of the proof that  $\psi_2$ is quasi-isomorphism in (1). 
Thu we conclude that 
the evaluation map $\epsilon_{[J]^{\wedge}}'$ is quasi-isomorphism. 
Now we complete the proof of (2).  

(3) follows from (1) and (2). 
\end{proof}

Now we start to  prove Theorem \ref{bic holim thm}. 
We equip $\cC(\cE)$ with the projective model structure. 
We define the full sub category $\cO$ 
to be $(\dgPerf\cE_{/DM})^{\circ}$. 
Namely 
$\cO$ is the full subcategory of $\dgPerf\cE_{/DM}$ 
consisting of those object $p:P \to DM$ such that 
the morphism $p$ is a projective fibration and the domain $P$ is projective cofibrant in $\cC(\cE)$. 

We construct a simplicial functor $\widetilde{D'}: \cO^{\textup{op}} \to \cU$. 
The pair $D,D'$ of dg-functors is a contravariant adjunction. 
Namely there exists a natural isomorphism of Hom complexes 
\[
\widetilde{D'}_{N,M}: {}_{\cE}(N,DM) \xrightarrow{\cong} {}_{\cA}(M,D'N)  
\]
for $M \in \cC(\cA)$ and $N \in \cC(\cE)$. 

For an object $p: P \to DM$ of $\cO$, 
we define $\widetilde{D'}(p)$  
to be 
$\widetilde{D'}(p) := \widetilde{D'}_{P,M}(p)$. 
More explicitly it is given by 
$\widetilde{D'}(p)= D'(p)\circ \epsilon_M$. 
\begin{equation}\label{main mis diagram 1}
[p:P \to DM] \mapsto [\widetilde{D'}(p) : M \xrightarrow{\epsilon_{M}} D'D(M) \xrightarrow{D'(p)} D'P]
\end{equation}
By Lemma \ref{main mis lemma 1}.(2) D'(p) is an injective cofibration. 
Since we assume that $\epsilon_M:M \to S(M) = D'D(M)$ is an injective cofibration, 
 the composite morphism $\widetilde{D'}(p) = D'(p) \circ \epsilon_m$ is an injectively cofibration. 
By Lemma \ref{main mis lemma 1}(4) the object $D'(P)$ is  injectively fibrant in $\cC(\cA)$, 
hence  
the object $\widetilde{D'}(p)$ of $\langle \cJ \rangle^{\cC}_{M/}$ belongs to $\cU$. 

By abusing notation 
we  denote by $\widetilde{D'}$
the isomorphism of the mapping complexes   
\[\widetilde{D'}: 
\Map_{\cC(\cE)}(P,DM) \to \Map_{\cC(\cA)}(M,D'P)  
\] 
induced from the above isomorphism $\widetilde{D'}_{P,M}$ of Hom-complexes.

Let $p:P \to DM$ and $q:Q \to DM$ be  objects of $\cO$. 
Recall that the mapping complex $\Map_{\cO}(p,q)$ is defined by the pull-back diagram 
\begin{equation}\label{diagram 1}
\begin{xymatrix}{
\Map_{\cO}(p,q) \ar[r] \ar[d] & \Map_{\cC(\cE)}(P,Q) \ar[d]^{q_{\ast}}\\
\{p\} \ar[r] & \Map_{\cC(\cE)}(P,DM) 
}\end{xymatrix}\end{equation} 
where $q_{\ast}:=\Map_{\cC(\cE)}(P,q)$. 
The mapping complex $\Map_{\cU}(\widetilde{D'}(q),\widetilde{D'}(p))$ is given by the pull-back diagram 
\begin{equation}\label{diagram 2}
\begin{xymatrix}{
\Map_{\cU}(\widetilde{D'}(q),\widetilde{D'}(p)) \ar[r] \ar[d] & 
\Map_{\cC(\cA)}(D'Q,D'P) \ar[d]^{\widetilde{D'}(q)^{\ast}}\\ 
\{\widetilde{D'}(p)\} \ar[r] & \Map_{\cC(\cA)}(M,D'P).
}\end{xymatrix}\end{equation}
where $\widetilde{D'}(q)^{\ast}: = \Map_{\cC(\cE)}(\widetilde{D'}(q),D'P)$. 
It is easy to check that 
the following diagram is commutative 
\begin{equation}\label{diagram 3}
\begin{xymatrix}{
\Map_{\cC(\cE)}(P,Q) \ar[d]_{q_{\ast}} \ar[r]^{D'} & \Map_{\cC(\cA)}(D'Q,D'P) \ar[d]^{\widetilde{D'}(q)^{\ast}}\\
\Map_{\cC(\cE)}(P,DM) \ar[r]_{\widetilde{D'}} & \Map_{\cC(\cA)}(M,D'P)\\
\{p\} \ar[u] \ar@{=}^{\sim}[r] & \{\widetilde{D'}(p)\} \ar[u]
}
\end{xymatrix}
\end{equation}
Therefore the map $D':\Map_{\cC(\cE)}(P,Q) \to \Map_{\cC(\cA)}(D'Q,D'P)$ 
induces the map 
 $\Map_{\cO}(p,q) \to \Map_{\cU}(\widetilde{D'}(q),\widetilde{D'}(p))$, 
which will be denoted by $\widetilde{D'}_{p,q}$. 
Then  
we can check that 
the assignment $\Ob\cO \to \Ob \cU,  p \mapsto \widetilde{D'}(p)$ 
and  the collection of maps $\{\widetilde{D'}_{p,q}\}$ 
gives a simplicial functor $\widetilde{D'}: \cO^{\textup{op}} \to \cU$.

\begin{lemma}
The simplicial functor $\widetilde{D'}$ is a Dwyer-Kan equivalence. 
\end{lemma}

\begin{proof}
First we show that 
the functor $\textup{Ho}(\widetilde{D'}):\textup{Ho}(\cO^{\textup{op}}) \to \textup{Ho}(\cU)$ 
is an essentially surjective.
Let $k: M \rightarrowtail K$ be an object of $\cU$. 
Then $Dk: DK \twoheadrightarrow DM$ is a projective fibration of dg-$\cE$-modules by Lemma \ref{main mis lemma 1}.
We take a projective cofibrant replacement $g: P \stackrel{\sim}{\twoheadrightarrow} DK$ of $DK$. 
Then the composite morphism $p:= Dk\circ g: P \twoheadrightarrow DM$ belongs to $\cO$. 
We set $\psi:= D'g \circ \epsilon_K$. 
The following commutative diagram shows that 
we have a weak equivalence $\psi :k \xrightarrow{\sim} \widetilde{D'}(p)$: 
\[
\begin{xymatrix}{
M \ar@{=}[d] \ar[rr]^{k} && K\ar[d]^{\wr}_{\epsilon_K}\ar[dr]^{\sim \psi} & \\
M \ar[r]_{\epsilon_M } & D'DM \ar[r]_{D'Dk} & D'DK \ar[r]^{\sim}_{D'g} &D'P
}\end{xymatrix}\]
By Lemma \ref{925 1053}, $\psi$ is a weak homotopy equivalence in $\cU$.  
Therefore we see the desired result.

We prove that the map $\widetilde{D'}_{p,q}$ of simplicial sets is weak homotopy equivalence. 
The top arrow $D': \Map_{\cC(\cE)}(P,Q) \to \Map_{\cC(\cA)}(D'Q,D'P)$ in the above diagram (\ref{diagram 3}) 
is a weak homotopy equivalence. 
The bottom arrow 
$\widetilde{D'}: \Map_{\cC(\cE)}(P,DM) \to \Map_{\cC(\cA)}(M,D'P)$ in the above diagram (\ref{diagram 3}) 
is an isomorphism of simplicial sets.  
Since the maps $q_{\ast}$ and $\widetilde{D'}(q)^{\ast}$ are fibrations of simplicial sets 
and the standard model structure of $\sSet$ is proper, 
the above diagrams (\ref{diagram 1},\ref{diagram 2}) are homotopy pull-back diagrams. 
Therefore we see the desired result. 
\end{proof}

We denote by $\Upsilon$ the domain functor 
\[\Upsilon: \cO \to \cC(\cE),\quad [p:P \to DM] \mapsto P.\] 
Then we have 
 $\Gamma \circ \widetilde{D'} = D' \circ \Upsilon^{\textup{op}}.$ 
\[
\begin{xymatrix}
{
\cO^{\textup{op}} \ar[d]_{\widetilde{D'}} \ar[r]^{\Upsilon^{\textup{op}}} 
& \cC(\cE)^{\textup{op}} \ar[d]^{D'}\\
\cU \ar[r]_{\Gamma} & \cC(\cA)
}
\end{xymatrix}
\]
We denote by $\Phi:\Upsilon \to  \ct_{\cO}(D(M)) $ the tautological natural transformation 
given by the assignment $\Phi: p\mapsto p$.  
By the construction of $\widetilde{D'}$, 
we have the commutative diagram
\[
\begin{xymatrix}{
D'\circ (\ct_{\cO^{\textup{op}}} DM) \ar@{=}[rr] \ar[d]_{D'(\Phi^{\textup{op}})}&&
\ct_{\cO^{\textup{op}}}S(M)  \ar[d]^-{S(\kappa_{\widetilde{D'}})} \\
 D' \circ \Upsilon^{\textup{op}} \ar@{=}[r] &\G \circ \widetilde{D'} \ar[r]^-{\sim}_{\epsilon_{\G}} & S\circ \G \circ \widetilde{D'}.
}\end{xymatrix}
\]

Since the functor $\widetilde{D'}$ is a Dwyer-Kan equivalence, in particular homotopy left cofinal, 
in view of Lemma \ref{923 043} 
it is enough to show that 
the morphism $S(\kappa_{\widetilde{D'}})$ exhibits 
$S(M)$ as a homotopy limit of $S\circ \G \circ \widetilde{D'}$. 
By Lemma \ref{923 017} it is enough to show that  
 $D'(\Phi^{\textup{op}})$  
exhibits $\ct_{\cO^{\textup{op}}}S(M)$ as a homotopy limit. 
Since the contravariant simplicial functor $D'= {}_{\cE}(-,\ulcJ)$ preserves quasi-isomorphisms and 
sends colimit to limit, 
in view of Lemma \ref{921 456}
it is enough to show that 
the map $\Phi: \Upsilon \to \ct_{\cO}(D(M))$ exhibits $D(M)$ as a homotopy colimit. 
It is proved in next section Theorem \ref{Perf hocolim thm}.
We finish the proof of Theorem \ref{bic holim thm}. 

\section{A dg-module is obtained as a tautological filtered homotopy colimit of perfect modules.}\label{AppA}

In this section we prove the following theorem.

Let $\cA$ be a small dg-category and $M$ be a dg $\cA$-module. 
We equip $\cC(\cA)$ with the projective model structure. 
We denote by $\cO$ the full sub simplicial category $\left( \dgPerf \cA_{/M}\right)^{\circ}$ 
of $\dgPerf \cA_{/M}$. 
Namely the object $p : P \to M$ is such that the morphism $p$ of $\cC(\cA)$ is a projective fibration 
and the domain $P$ of $p$ is a projectively  fibrant object which belongs to $\dgPerf \cA$.  
We denote by $\Upsilon:= \Upsilon_{M}$ the domain functor 
\[
\Upsilon: \cO \to \cC(\cA) ,\quad [p: P \to M ] \mapsto P=: \Upsilon_{p}
\]
We denote by $\Upsilon_{p}$ the value of $\Upsilon$ at $p\in \cO$.
Observe that 
the assignment $\Phi: p \mapsto p$ gives a natural transformation 
$\Phi: \Upsilon \to \ct_{\cO}M $. 
\[
\begin{xymatrix}{
\Upsilon_p \ar[d]_{\Phi_p} \ar@{=}[r] & P \ar[d]^{p}  \\
(\ct_{\cO} M)_p \ar@{=}[r] & M 
}\end{xymatrix}\]

\begin{theorem}\label{Perf hocolim thm}
The canonical morphism 
\[
\hocolim_{\cO} \Upsilon \to M
\] 
is a quasi-isomorphism. 
More precisely the morphism $\Phi$ exhibits $M$ as a homotopy colimit of $\Upsilon$. 
Namely for any cofibrant replacement $\theta \xrightarrow{\sim} \Phi$ in the functor category $\cC(\cA)^{\cO}$ 
equipped with the projective model structure, 
the composite map $\Theta \to \Upsilon \to \ct_\cO M$ 
induces a quasi-isomorphism $\colim \Theta \to M$. 
\end{theorem}

In the rest of this section  we devote to prove this theorem.
For $i\in \ZZ$ and $a \in \cA$,  
taking $i$-th cohomologies at $a$,  
we obtain a natural transformation 
$\tuH^{i}(\Phi(a)) : \tuH^{i}(\Upsilon(a)) \to \tuH^{i}(\textup{ct}(M)(a))= \tuH^{i}(M(a))$ 
of the functors from $\cO $ to $ \Mod \kk$. 
We denote by $f$ the morphism $\colim\tuH^{i}(\Phi(a))$.   
\[
f := \colim\tuH^{i}(\Phi(a)) : \colim_{\cO}\tuH^{i}(\Upsilon(a) ) \to \tuH^{i}(M(a))
\]

\begin{lemma}\label{107 216}
The morphism $f$ is an isomorphism.  
\end{lemma}

We set  $\alpha := a^{\wedge}[-i]$. 
Recall that by Yoneda Lemma,  
for a dg $\cA$-module $N$, there exists a natural morphism 
\[
q:= q^{N}:  \tuZ^{i}(N(a))  \xrightarrow{\cong} \Hom_{\cC(\cA)}(\alpha,N). 
\]
We denote by $u$ the element of $\tuZ^{i}(\alpha(a))$ such that $q_{u}= 1_{\alpha}$. 
Then the  morphism $q_{n} : \alpha \to N$ in $\cC(\cA)$ is characterized by the property 
that 
$\tuZ^{i}(q_{n}(a))(u) = n$.

\begin{proof}
We set $L := \colim_{\cO}\tuH^{i}(\Upsilon(a))$. 
First we prove that 
$f$ is surjective. 
Let $\overline{m}$ be an element of $\tuH^{i}(M(a))$. 
We choose a representative element $m$ in $\tuZ^{i}(M(a))$ of $\overline{m}$.
We decompose $q_{m}: \alpha \to M$ into a trivial cofibration
  $s:\alpha \stackrel{\sim}{\rightarrowtail} \beta$ 
  followed by a fibration $p: \beta \twoheadrightarrow M$. 
  Since $\alpha$ is cofibrant, $\beta$ is also cofibrant. 
  Hence $p$ is an object of $\cO$. 
\[
\begin{xymatrix}{
*++{\alpha} \ar[dr]_{q_{m}} \ar@{>->}[r]^{\sim s} & \beta\ar@{->>}[d]^{p} \\
& M }\end{xymatrix}
\]
Now the following commutative diagram tells us that $m$ is in the image of $f$
\[
\begin{xymatrix}{
\tuH^{i}(\alpha(a)) \ar[dr]_{\tuH^i(q_{m}(a))}\ar[r]^{\cong} & 
\tuH^{i}(\beta(a)) \ar[d]^{\tuH^{i}(p(a))} \ar[rr] & & L\ar@/^1pc/[dll]^{f} \\
& \tuH^{i}(M(a)) &
}\end{xymatrix}\]

Next we prove that 
$f$ is injective. 
Let $\overline{\xi}_{1},\overline{\xi}_{2}$ 
be elements of $L$ such that $f(\overline{\xi}_{1})=f(\overline{\xi}_{2})$. 
For $s=1,2$ 
there exists object $p_{s}: P_{s} \to M$ of $\cO$ 
such that 
there exists an element $\overline{x}_{s}$ of $\tuH^{i}(P_{s}(a))$ such that 
$\eta_s(\overline{x}_s) = \overline{\xi}_s$ 
where 
$\eta_{s}$ is the canonical morphism $ \tuH^{i}(P_s(a))= \tuH^{i}(\Upsilon_{p_s}(a)) \to L$. 
We choose a representative $x_s$ in $\tuZ^{i}(P_s(a))$ of $\overline{x}_s$ for $s=1,2$. 
Then 
the morphisms $p_1\circ q_{x_1}$ and $p_2\circ q_{x_2}$ are homotopic. 
Therefore we have the following commutative diagram 
\[\begin{xymatrix}{
& \alpha \ar[dl]_{q_{x_1}} \ar[dr]^{\lambda_{1}} && \alpha \ar[dl]_{\lambda_{2}} \ar[dr]^{q_{x_2}} & \\
P_{1} \ar[drr]_{p_{1}} & & \Cyl(\alpha) \ar[d]^{ H} && P_{2} \ar[dll]^{p_2} \\
&& M && 
}\end{xymatrix}\]
where $\Cyl(\alpha)$ is a cylinder object of $\alpha$, e.g., $\Cyl(\alpha) = \alpha\otimes \Delta^1$. 
In other words, we have the following  diagram in $\dgPerf\cA_{/M}$ 
\[
\begin{xymatrix}{
& \mu_1 \ar[dr]^{\lambda_1} \ar[dl]_{q_{x_1}}&& \mu_2 \ar[dl]_{\lambda_2}\ar[dr]^{q_{x_2}}& \\
p_1  && H && p_2
}\end{xymatrix}\]
where we set $\mu_s : = p_s \circ q_{x_s} = H \circ \lambda_s$ for $s= 1,2$. 
Using  the following Lemma \ref{App A lem 1}   
we obtain the following homotopy commutative diagram 
in $\dgPerf\cA_{/M}$ such that $r_1,r_2,r_3$ belong to $\cO$ 
\[
\begin{xymatrix}{
& \mu_1 \ar[dr]^{\lambda_1} \ar[dl]_{q_{x_1}}&& \mu_2 \ar[dl]_{\lambda_2}\ar[dr]^{q_{x_2}}& \\
p_1\ar[dr]  && H  \ar[dl]\ar[dr] && p_2 \ar[dl] \\
& r_1 \ar[dr] && r_2 \ar[dl] & \\
&& r_3 &&
}\end{xymatrix}\]
Applying domain functor and taking cohomology groups to this diagram,  
we obtain  the following strictly commutative diagram in $\Mod \kk$ 
\[
\begin{xymatrix}{
& \tuH^{i}(\alpha(a) ) \ar@{=}[rr] \ar[dr] \ar[dl]&&
 \tuH^{i}(\alpha(a))  \ar[dl]\ar[dr] & \\
\tuH^{i}(\Upsilon_{p_1}(a) ) \ar[dr]  
&& \tuH^{i}(\Cyl(\alpha)(a))  \ar[dl]\ar[dr] 
&& \tuH^{i}(\Upsilon_{p_2}(a))  \ar[dl] \\ 
& \tuH^{i}(\Upsilon_{r_1}(a)) \ar[dr] && \tuH^{i}(\Upsilon_{r_2}(a)) \ar[dl] & \\
&& \tuH^{i}(\Upsilon_{r_3}(a))  &&
}\end{xymatrix}\]
Chasing the homotopy class of the element $u \in \tuZ^{i}(\alpha(a))$, we see that 
 $\overline{\xi}_1=\eta_1(\overline{x}_1) = \eta_2(\overline{x}_2) = \overline{\xi}_2$. 
\end{proof}

We denote by $(\dgPerf\cA^{\textup{p-cof}})_{/M}$
the subcategory of 
$\dgPerf \cA_{/M}$ 
consisting those objects $p: P \to M$ 
such that the domain $P$ is projectively cofibrant 
perfect $\cA$-module. 

\begin{lemma}\label{App A lem 1}

(1) 
Let $p:P \to M$ be an object of $(\dgPerf\cA^{\textup{p-cof}})_{/M}$ and $q:Q \to M $ a object of $\cO$. 
Then there exists a diagram $p \rightarrow r \leftarrow q$ in $(\dgPerf\cA^{\textup{p-cof}})_{/M}$ 
such that $r$ belongs to $\cO$.

(2)
Let $p:P \to M $ be an object of $(\dgPerf\cA^{\textup{p-cof}})_{/M}$ and $q:Q \to M$ a object of $\cO$.
For any parallel two morphisms $f,g:p \rightarrow q$ in $\cO$ 
there exists an morphism $h : q \to r$ such that the composite morphisms 
$h\circ f$ and $h \circ g$ are homotopic. 
\end{lemma}

\begin{proof}
(1)
It is enough to set $R= P \bigoplus Q$ and $r= (p, q)$. 
\[
\begin{xymatrix}{
P \ar[rr]^{{}^{t}(1,0)} \ar[drr]_{p} && P \bigoplus Q \ar[d]^{(p,q)} && 
Q \ar[ll]_{{}^{t}(0,1)} \ar[dll]^{q} \\
&& M &&
}\end{xymatrix}\]

(2) 
follows from Lemma \ref{106 1146} and the following remark. 
\end{proof}

Let $\textbf{rep}: \Theta \xrightarrow{\sim} \Upsilon $ be a cofibrant replacement of $\Upsilon$. 
Then  we have the following commutative diagram
\[\begin{xymatrix}{ 
\colim_{\cO}\tuH^{i}(\Theta(a)) \ar[d]_{\colim_{\cO} \tuH^{i}(\textbf{rep})} \ar[rr]^{\textbf{can}} && 
\tuH^{i}(\colim_{\cO}\Theta(a)) \ar[d]^{ \tuH^{i}(\colim_\cO \Phi\circ \textbf{rep})} \\
\colim_{\cO}\tuH^{i}(\Upsilon(a) ) \ar[rr]_{\colim_\cO \tuH^{i}(\Phi(a))} && \tuH^{i}(M(a))  
}
\end{xymatrix}\]
where $\textbf{can}$ is the canonical morphism. 
We need to show that the left vertical arrow is an isomorphism. 
By Lemma \ref{107 216} the bottom arrow is an isomorphism. 
Since the morphism $\textbf{rep}$ is a weak equivalence, 
the right vertical arrow is also an isomorphism. 
Therefore we only have to show that 
the top arrow $\textbf{can}$ is an isomorphism. 
It follows from the following two lemmas: 
the $\infty$-category $\sner(\cO)$ is filtered (Lemma \ref{107 225}) and 
filtered homotopy colimit commutes with taking cohomology group (Lemma \ref{hom com lem}).

\begin{lemma}\label{107 225}
The $\infty$-category $\sner(\cO)$ is filtered.
\end{lemma}

\begin{proof}
Since $\cO$ is a fibrant simplicial category, 
by \cite[Proposition 5.3.1.13, Definition 5.3.1.1]{HTT}
it is enough to prove the following two conditions: 
\begin{enumerate}
\item[(1)]
For every finite set $\{p_{i}\}_{i=1}^{n}$ of objects of $\cO$, 
there exists an objects $p\in \cO$ and morphisms $\phi_i:p_{i} \to p$. 

\item[(2)]
For every pair $p,q$ of objects of $\cO$, 
every finite simplicial set $X$ and every morphism $X \to \Map_{\cO}(p,q)$ 
there exists a morphism $q \to r$ such that 
the induced map $X \to \Map_{\cO}(p,r)$ is null homotopic. 
\end{enumerate}

(1) 
For $p_{i}:P_{i} \to M$, $i= 1,2,\dots,n$, 
it is enough to set $P:= (P_{1},\dots,P_{n}):  \oplus_{i=1}^{n}P_{i} \to M$ 
and
 $\phi_{i}: P_i \to P$ be the morphism in $\cO$ 
induced from the $i$-th canonical injection  $P_i \to \bigoplus_{i=1}^{n} P_{i}$.  

(2) 
For simplicity 
we denote by $[-,+]= \Hom_{\textup{h}\sSet}(-,+)$ the Hom set of the homotopy category $\textup{h}\sSet$.

Let $f$ be an element of  $\Hom_{\sSet}(X,\Map_{\cO}(p,q)) $. 
We denote by the same symbol $f \in [X,\Map_{\cO}(p,q)]$ the homotopy class of $f$  
The problem is to show that 
there exists a morphism $g: q \to r$ 
such that 
the image of $f$ of the morphism 
$[X,\Map_{\cO}(p,q)] \to  [X,\Map_{\cO}(p,r)]$ 
induced from $g$ 
belongs to the image of 
the morphism 
$[\ast,\Map_{\cO}(p,r)] \to [X,\Map_{\cO}(p,r)]$ 
induced from a unique map $X \to \ast$.  

Recall that we have a natural isomorphism 
\[
[X,\Map_{\cO}(p,q)] \cong \Hom_{\textup{h}\cO}(p\otimes X, q). 
\]
Therefore 
the problem is equivalent to show that 
there exists a morphism $g: q \to r $ in $\cO$ 
such that 
we have a homotopy commutative diagram
\[
\begin{xymatrix}{
p\otimes X \ar[d]_-{\alpha} \ar[r]^-{f}  & q \ar[d]^{g} \\
p \ar[r] & r
}\end{xymatrix}
\]
where $\alpha : p\otimes X \to p$ is the morphism 
induced from a unique morphism $X \to \ast$. 
However this is a consequence of Lemma \ref{106 1146} and the following remark.
\end{proof}

\begin{lemma}\label{hom com lem}
Let $V$ be a filtered  $\infty$-category  and  $F : V \to \sner(\cC(\cA)^{\circ})$ an $\infty$-functor.
Then the canonical morphism 
\[
\colim_{V}\tuH^{i}(F(a)) \to \tuH^{i}(\hocolim_{V}F(a))
\]
is an isomorphism for $ i \in \ZZ$ and $a\in \cA$. 
\end{lemma}

\begin{proof}
By \cite[Proposition 5.3.16]{HTT} 
there there exists a filtered partially ordered set $A$ and a cofinal map $\iota: \sner(A) \to V$. 
Therefore we may replace $V$ with $\sner(A)$ and $F$ with $F\circ \iota$. 
Since we have an isomorphism 
\[
\Hom_{\textup{Ho}(\sCat)}(A,\cC(\cA)^{\circ}) 
\cong \Hom_{\textup{Ho}(\infCat)}(\sner(A),\sner(\cC(\cA)^{\circ})), 
\]
where we denote by $\infCat$ the model  category $\sSet$ equipped with the Joyal model structure, 
we may replace $F\circ \iota$ with $\sner(f)$ for some functor $f:A\to \cC(\cA)$. 
Now the result follows from the well-known fact 
that filtered colimit commute with taking cohomology group.    
\end{proof}

This finish the proof of Theorem \ref{Perf hocolim thm}.


\section{Miscellaneous results}

\subsection{On the choice in Construction \ref{103 444}.(III) }\label{103 1010}
In this subsection
we will show that the conclusion of Theorem \ref{bic holim thm} 
is independent to the choice in Construction \ref{103 444}.(III). 
More precisely,  
in Construction \ref{103 444}.(III)
we take a different choice $\ulcJ_2 \xrightarrow{\sim} \ulcJ'$ of a fibrant replacement 
and denote by $S'$ the corresponding bi-duality functor.

\begin{proposition}\label{104 603}
The canonical morphism $S(\kappa):ct S(M) \to S\circ \G$ exhibits $S(M)$ as a homotopy limit 
of $S\circ \G$ if and only if 
the canonical morphism $S'(\kappa):ct S'(M) \to S'\circ \G$ exhibits $S'(M)$ as a homotopy limit 
of $S'\circ \G$.  
\end{proposition}

 Corollary \ref{103 824} claim that
 the quasi-isomorphism class of $S(M)$ is independent to the choice of $\ulcJ$ 
in the step (III). 
However we need finer information for our purpose. 
For latter reference, 
we consider the following more general situation. 

Let $\cF$ be a locally cofibrant dg-category. 
Assume that there exists a quasi-equivalence $g: \cF \xrightarrow{\sim} \cE$. 
To prove Proposition \ref{104 603} it is enough to  set $\cF = \cE$ and $g= 1_\cE$. 
By Lemma \ref{103 815} the dg-functor $\tilde{g} : = 1_\cA \otimes g$ 
is a quasi-equivalence from $\cA\otimes \cF$ to $\cA \otimes \cE$.
Note that for $\cA \otimes \cE$-module $X$ and for $\cA$-module $N$, 
we have  natural isomorphisms 
\begin{equation}\label{1010 247}
g_{\ast}\left( X(a\otimes -) \right) \cong \left(\tilde{g}_{\ast}X\right) (a\otimes -),
\quad 
g_{\ast}{}_{\cA}(N,X) \cong {}_{\cA}(N,\tilde{g}{\ast}X).
\end{equation} 
We assume that 
the fibrant replacement $\ulcJ$ of the diagonal module $\ulcJ_2$ 
is taken by a injectively trivial cofibration $\ulcJ_2 \trivcofright \ulcJ$.  
Then the induced morphism $\tilde{g}_\ast \ulcJ_2 \to \tilde{g}_{\ast} \ulcJ$ 
is also injectively trivial cofibration. 
Let $\tilde{g}_\ast \ulcJ_2 \xrightarrow{\sim} \ulcJ'$ be a fibrant replacement. 
Then by the lifting property, 
there exists a quasi-isomorphism $\alpha: \tilde{g}_\ast\ulcJ \to \ulcJ'$. 

We define the endofunctors $S'$ and $T$ of $\cC(\cA)$ to be 
\[
S'(-) := {}_{\cF}({}_{\cA}(- , \ulcJ'), \ulcJ'), \qquad
 T(-) := {}_{\cF}({}_{\cA}(-,\tilde{g}_{\ast} \ulcJ), \ulcJ'). 
 \] 
 We denote by $\epsilon: 1_{\cC(\cA)} \to S'$ the natural transformation induced from the  evaluation map.  
Then  as in the proof of Lemma \ref{main mis lemma 2}, 
the morphism $\alpha$ induces the following commutative diagram of endofunctors of $\cC(\cA)$ 
\[
\begin{xymatrix}{
1_{\cC(\cA) } \ar[d]_{\epsilon'} \ar[r]^{\epsilon} & 
S \ar[d]^{\rho} \\
S' \ar[r]_{\lambda} & T
}\end{xymatrix}\]
We claim that the morphisms $\lambda,\rho$ are weak equivalences of endofunctors. 
Namely, for $N\in \cC(\cA)$ the morphisms $\lambda_N,\rho_N$ are quasi-isomorphisms. 
As in the same way   
that we prove $\psi_1$ is quasi-isomorphism in  the proof of  Lemma \ref{main mis lemma 2}.(1) ,  
we can prove that $\lambda_N$ is a quasi-isomorphism. 
The morphism $\rho_N$ evaluated at $a \in \cA$ is given in the following way 
\[
\begin{split}
{}_{\cE}({}_{\cA}(N,\ulcJ), \ulcJ(a\otimes -)) 
\xrightarrow{g_{\ast}} 
 {}_{\cF}(g_{\ast} {}_{\cA}(N,\ulcJ), g_{\ast}(\ulcJ(a\otimes -))) 
& \cong {}_{\cF}({}_{\cA}(N,\tilde{g}_{\ast}\ulcJ) , (\tilde{g}_{\ast}\ulcJ)(a\otimes-)) \\
& \xrightarrow{\tilde{\alpha}_a} {}_{\cF}({}_{\cA}(N,\tilde{g}_{\ast}\ulcJ), \ulcJ'(a\otimes-)).  
\end{split}
\]
The last morphism $\tilde{\alpha}_a$ is induced 
from the quasi-isomorphism $\alpha(a) : (\tilde{g}_\ast \ulcJ)(a\otimes- ) \xrightarrow{\sim} \ulcJ'(a\otimes-)$. 
Since $\ulcJ(a\otimes - )$ is injectively fibrant, 
the Hom complex ${}_{\cE}({}_{\cA}(N,\ulcJ), \ulcJ(a\otimes -)) $  compute 
the derived Hom complex on the top of the following diagram.  
\[\begin{xymatrix}{
\RR\Hom_{\cE}\left(  {}_{\cA}(N,\ulcJ), \ulcJ(a\otimes -)\right) \ar[d]^{g_\ast} \\
 \RR\Hom_{\cF}\left(g_{\ast} {}_{\cA}(N,\ulcJ), g_{\ast}(\ulcJ(a\otimes -))\right).
}\end{xymatrix}\]
On the other hands, 
the morphism $\alpha(a) $ is a fibrant replacement of $(\tilde{g}_\ast \ulcJ)(a\otimes- ) $ by Lemma \ref{103 754}. 
Hence the Hom complex ${}_{\cF}({}_{\cA}(M,\tilde{g}_{\ast}\ulcJ), \ulcJ'(a\otimes-))$ 
compute 
the derived Hom complex in the bottom of above diagram. 
Observe that the morphism $\rho_N$ is quasi-isomorphic to 
the morphism between the  above derived Hom complexes  induced from $g_{\ast}$. 
By Lemma \ref{926 838} 
the restriction functor $g_{\ast}: \cD(\cE) \to \cD(\cF)$ gives an equivalence. 
Hence we conclude that $\rho_N$ is quasi-isomorphism.

Compositing the above diagram of endofunctors of $\cC(\cA)$ with 
the morphism $\kappa: \ct_{\cU} M  \to \G$, 
we obtain the commutative diagram 
\[
\begin{xymatrix}{ 
\ct_{\cU} S'(M) \ar[d]_{S'(\kappa)} \ar[r]^{\sim}  & 
\ct_{\cU} T(M) \ar[d]_{T(\kappa)} & 
\ct_{\cU} S(M) \ar[d]^{S(\kappa)} \ar[l]_{\sim} \\ 
S'\circ \G \ar[r]^{\sim} & 
T\circ \G & 
S \circ \G \ar[l]_{\sim}.  
}\end{xymatrix}\]
By Lemma \ref{923 017}, we see that 
$S(\kappa)$ exhibits $S(M)$ as a homotopy limit of $S\circ \G$ 
if and only if 
$S'(\kappa)$ exhibits $S'(M)$ as a homotopy limit of $S'\circ \G$.

We define a dg-category  
$\cB$ to be 
\[
\textit{Ob}(\cB)  := \textit{Ob}\cA, \quad 
{\cB}(a,b):= {}_{\cF}(\ulcJ'(b\otimes-),\ulcJ'(a\otimes -)).
\]
The identity on the objects induces a dg-functor $\tilde{\iota}: \cA \to \cB$.   
Note that 
by \cite[Lemma 3.2]{Efi} 
in the homotopy category $\Ho(\dgCat_{\cA/})$ 
the object $\iota : \cA \to \BBic_{\cA}(\cJ)$ 
is isomorphic to $\tilde{\iota}: \cA \to \cB$. 
In other words, we can compute the derived bi-commutator category $\BBic_{\cA}(\cJ)$ 
by using $\cF$ and $\ulcJ'$.

\subsection{Derived bi-commutators over equivalent subcategories}

\begin{proposition}\label{1010 314}
Let $f: \cA \to \cB$ be a dg-functor 
and $\cJ \subset  \cD(\cB)$ a small subcategory.
Assume that 
the restriction functor $f_{\ast} : \cD(\cB) \to \cD(\cA)$ 
induces an equivalence between $\cJ$ and $f_{\ast}\cJ$. 
Then 
in the homotopy category $\Ho(\dgCat)$ we have a quasi-fully faithful functor  
$f^{\heartsuit}:\BBic_{\cA}(f_{\ast}\cJ) \to \BBic_{\cB}(\cJ)$ 
and 
the commutative diagram 
\[\begin{xymatrix}{
\cA \ar[d]_{\iota_{f_{\ast}\cJ}} \ar[r]^{f} & \cB \ar[d]^{\iota_{\cJ}} \\
\BBic_{\cA}(f_{\ast} \cJ) \ar[r]_-{f^{\heartsuit}} & \BBic_{\cB}(\cJ) .
}\end{xymatrix}
\]
\end{proposition}
\begin{remark}
Since the canonical morphisms $\iota_{f_{\ast}\cJ}$ and $\iota_{\cJ}$ 
induces identities on objects,  
if the functor $f$ is a quasi-essentially surjective, 
then the functor $f^{\heartsuit}$ gives a quasi-equivalence of the derived bi-commutator categories. 
\end{remark}

We may assume that 
the functor $f: \cA \to \cB$ is a cofibration. 
We construct $\ulcJ_1 \subset \cC(\cB)$ and  $\cF := \cE\textup{nd}_{\cA}(\ulcJ_1)^{\textup{op}}$ 
by using Construction \ref{103 444}. 
By Lemma \ref{922 258}, 
the $\cA$-module  $f_{\ast}\ulcJ_1$ is injectively fibrant. 
Hence the dg-categories 
$\cE:= \cE\textup{nd}_{\cA}(f_\ast \ulcJ_1)^{\textup{op}}$ 
is  locally cofibrant. 
By assumption 
the dg-functor $g: \cF \to \cE$ 
induced from $f_{\ast}$
is a quasi-equivalence  of locally cofibrant dg-categories.  

Let   $\ulcJ_2$ and $\underline{f_\ast \cJ}_2$ 
be the diagonal modules of $\ulcJ_1$ and $f_\ast \ulcJ$ respectively. 
We set $\tilde{f}:= f\otimes 1_{\cF} : \cA \otimes \cF \to \cB \otimes \cF$ 
and $\tilde{g}:= 1_{\cA} \otimes g : \cA \otimes \cF \to \cA \otimes \cE$. 
Then we can check  
an isomorphism 
$\tilde{f}_{\ast}\ulcJ_2 \cong \tilde{g}_{\ast}(\underline{f_\ast \cJ}_2)$ 
of $\cA \otimes \cF$-modules. 
Let $\ulcJ_2 \trivcofright \ulcJ$ be an injectively fibrant replacement of $\ulcJ_2$. 
By Lemma \ref{922 258} $\tilde{f}_{\ast}\ulcJ$ is an injective fibrant replacement of $\tilde{f}_\ast \ulcJ_{2}$. 
Therefore 
according to  the consideration in Section \ref{103 1010} 
we can compute the derived bi-dual over $f_{\ast}\cJ \in \cD(\cA)$ by using $\tilde{f}_\ast\ulcJ$.   
\[
S_{f_{\ast}\cJ}(-):= {}_{\cF}({}_{\cA}(-,\tilde{f}_{\ast}\ulcJ),\tilde{f}_{\ast}\ulcJ). 
\]
The restriction functor $f_{\ast}:\cC(\cB) \to \cC(\cA)$ 
induces the morphism ${}_{\cB}(f^{\ast}M,\ulcJ) \to {}_{\cA}(f_{\ast}f^{\ast}M,\tilde{f}_{\ast}\ulcJ)$ 
of dg $\cE$-modules. 
We denote by $\alpha_{M}:S_{f_{\ast}J\cJ}(f_{\ast}f^{\ast}M) \to f_{\ast}S_{\cJ}(f^{\ast}M)$ 
which is induced from the above morphism 
\[
\alpha_{M}:S_{f_{\ast}\cJ}(f_{\ast}f^{\ast}M)= 
{}_{\cF}({}_{\cA}(f_{\ast}f^{\ast}M,\tilde{f}_{\ast}\ulcJ),\tilde{f}_{\ast}\ulcJ)
\to {}_{\cF}({}_{\cB}(f^{\ast}M,\ulcJ),\tilde{f}_{\ast}\ulcJ)\cong  f_{\ast}S_{\cJ}(f^{\ast}M).
\]
where for the last isomorphism we use an isomorphism (\ref{1010 247}). 
Let  $\eta_{M}: M \to f_{\ast}f^{\ast}M$ be the unit map. 
Then we have the following commutative diagram 
\[
\begin{xymatrix}{
M \ar[d]_{\epsilon_{M}} \ar[rr]^{\eta_{M}} &&
 f_{\ast}f^{\ast}M \ar[d]_{\epsilon_{f_{\ast}f^{\ast}M}} \ar@/^4pc/[dd]^{f_{\ast}\epsilon_{f_{\ast}M}} \\ 
S_{f_{\ast}\cJ}(M) \ar[rr]_{S_{f_{\ast}\cJ}(\eta_{M})} \ar@/_1pc/[drr]_{\beta_M}
 && S_{f_{\ast}\cJ}(f_{\ast}f^{\ast}M) \ar[d]_{\alpha_{M}} \\ 
&& f_{\ast}S_{\cJ}(f^{\ast}M) 
}\end{xymatrix}\]
where we denote by $\beta_{M}$  the composition  $\beta_M:= \alpha_{M}\circ S_{f_{\ast}\cJ}(\eta_{M})$. 
However 
the morphism $\beta_M$ is nothing but the morphism 
induced from the adjunction isomorphism 
${}_{\cA}(M, \tilde{f}_{\ast}\ulcJ) \cong {}_{\cB}(f^{\ast}M,\ulcJ)$. 
In particular it is an isomorphism.

In the case when $M = a^{\wedge}$ for $a\in \cA$, 
the morphism $\beta_{a^{\wedge}}(b)$ gives  an isomorphism  
\[
S_{f_{\ast}\cJ}(a^{\wedge})(b) ={}_{\cE}(\tilde{f}_{\ast}\ulcJ(b), \tilde{f}_{\ast}\ulcJ(a))
 \cong {}_{\cE}(\ulcJ(fb) , \ulcJ(fa)) = f_{\ast}S_{\cJ}(f^{\ast}a^{\wedge})(b)
\]
of dg $\kk$-modules for $b\in \cA$. 
The family of  the morphisms $\beta_{a^{\wedge}}$ with $a \in \cA$ 
induces a fully faithful dg-functor 
$
f^{\heartsuit}: \dgBic_{\cA}(f_{\ast}\cJ) \hookrightarrow \dgBic_{\cB}(\cJ). 
$
It is straightforward to check that 
we have the commutative diagram 
\[\begin{xymatrix}{
\cA \ar[d]_{\iota_{f_{\ast}\cJ}} \ar[r]^{f} & \cB \ar[d]^{\iota_{\cJ}} \\
\dgBic_{\cA}(f_{\ast} \ulcJ) \ar[r]_-{f^{\heartsuit}} & \dgBic_{\cB}(\ulcJ) .
}\end{xymatrix}
\]

\section{Completion via derived bi-commutator}

\subsection{Completion theorems}

\begin{theorem}\label{comp thm 1}
Let $R$ be a ring and 
 $\mathfrak{a}$  a two-sided ideal of $R$ such that 
 the canonical functor 
$ \cD(\frka\textup{-tor}) \to \cD_{\frka\textup{-tor}} (R) $ gives an equivalence. 
Assume that
$R/\frka^n$ belongs to $\langle R/\frka \rangle^{\cD}$ for $n \geq 0$. 
\textup{(}e.g. $R/\frka$ has finite global dimension and $R$ is Noetherian.\textup{)}  
We denote by $\widehat{R}$  the $\frka$-adic completion. 
Then in the homotopy category $\Ho(\dgCat)$, there exists an isomorphism 
$\BBic_{R}(R/\frka) \cong \widehat{R}$ 
which fits into the commutative diagram 
\[
\begin{xymatrix}{
R \ar[rr]^{\iota} \ar[drr]_{\textup{\textbf{comp}}} && \BBic_{R}(R/\frka) \ar@{=}[d]^{\wr}\\
&& \widehat{R} 
}\end{xymatrix}\]
where the slant arrow $\textup{\textbf{comp}}: R \to \widehat{R}$ is the  canonical completion morphism.
\end{theorem}

A proof is given in the next section \ref{926 125}. 

\begin{theorem}\label{comp thm 2}
Let $R$ be a ring and $\frka$ an two-sided ideal such that 
the canonical functor $\cD(\frka\textup{-tor}) \to \cD_{\frka\textup{-tor}}(\cR)$ 
 gives an equivalence. 
Let $K$ be a compact generator of $\cD_{\frka\textup{-tor}}(R)$. 
Then in the homotopy category $\Ho(\dgCat)$, there exists an isomorphism 
$\BBic_{R}(K) \cong \widehat{R}$ 
which fits into the commutative diagram 
\[
\begin{xymatrix}{
R \ar[rr]^{\iota} \ar[drr]_{\textup{\textbf{comp}}} && \BBic_{R}(K) \ar@{=}[d]^{\wr}\\
&& \widehat{R} 
}\end{xymatrix}\]
where the slant arrow $\textup{\textbf{comp}}: R \to \widehat{R}$ is the  canonical completion morphism.

\end{theorem}

This theorem is proved in the same way of the proof of 
Theorem \ref{comp thm 1} by the following Lemma \ref{uuuuuuuu}.  

Let $R$ be a ring. 
Let $K$ be an object of $\cC(R)$. 
We denote by $\overline{\langle K \rangle}^{\cD}$ the smallest localizing subcategory of $\cD(R)$ 
containing $K$. 
We denote by $\overline{\langle K \rangle}^{\cC}$ the pre-image of $\overline{\langle K \rangle}^{\cD}$ 
by the homotopy functor $\cC(R) \to \cD(R)$. 
We set $\cU :=(\langle K \rangle^{\cC}_{R/})^{\circ} $ and 
 $\overline{\cU} := \left(\overline{\langle K \rangle}^{\cC}_{R/} \right)^{\circ}$. 
Let  $\overline{\G} : \overline{\cU} \to  \cC(\cA)$ and  $\G: \cU \to \cC(\cA)$ 
be the co-domain functors. 
Then the restriction $\overline{\G}|_{\cU} $ to $\cU$ equals to $\G$.

\begin{lemma}\label{uuuuuuuu} 
Assume that 
$K$ is compact in $\overline{\langle K \rangle}^{\cD}$. 
Then the canonical map $\ct_{\overline{\cU}} S(R)\to S\circ \overline{\G} $ 
exhibits $S(R)$ as a homotopy limit  of $S \circ \overline{\G}$. 
\end{lemma}

\begin{proof}
It is enough to prove that $\cU$ is a left homotopy co-final subcategory of $\overline{\cU}$. 
In the same way of the proof of Lemma \ref{107 225} 
we prove that 
the simplicial category  $\left(\overline{\langle K \rangle}^{\cC}_{R/} \right)^{\circ}$ is homotopy left filtered. 
Therefore 
by Lemma \ref{cofinal lem} and Lemma \ref{model lem},
 it is enough to show that 
for any morphism $\overline{\ell}: R \to \overline{L}$ in $\cD(R)$ 
with $\overline{L} \in \overline{\langle K \rangle}^{\cD}$ 
there exists a factorization $\overline{\ell} = \psi \circ \ell$ 
such that $\ell : R \to L$ is a morphism with $L \in \langle K \rangle^{\cD}$.

We take a projective cofibrant replacement $\widetilde{K} \xrightarrow{\sim} K$. 
We set $S: = \End_{R}^{\bullet}(\widetilde{K})^{\textup{op}}$. 
Then by \cite[Theorem 3.5]{PSY2} 
the adjoint pair 
\[
-\otimes_{S} \widetilde{K} : \cC(S) \rightleftarrows \cC(R) : {}_{R}(\widetilde{K}, -) 
\]
induces an equivalence between $\cD(S) $ and $\overline{\langle K \rangle}^{\cD}$. 
Therefore 
$\overline{L}$ is quasi-isomorphic to a dg $R$-module of form $P \otimes_{S} \widetilde{K}$ 
for some semi-free dg $S$-module $P$.  
Therefore we may assume that 
the underlying complex $\overline{L}^{\#}$ of $\overline{L}$ 
is a direct sum of a direct summand of modules of the form $\widetilde{K}[a], a\in \ZZ$ 
and the  differential $d_{\overline{L}}$ is given by one-sided twist.  
Now we can easily see that 
for any cocycle $\ell$ of $\overline{L}$ 
there exists a sub-complex $L$ of $\overline{L}$ such that 
$\ell\in L$ and $L \in \langle K \rangle$.    
\end{proof}

It is proved by  \cite[Corollary 3.31]{PSY1} 
that if a ring  $R$ is commutative and an ideal $\frka$ is weakly pro-regular, 
then the canonical functor $\cD(\frka\textup{-tor}) \to \cD_{\frka\textup{-tor}}(R)$  
gives an equivalence. 
Therefore 
as a corollary of Theorem \ref{comp thm 2}, we reprove \cite[Theorem 4.2]{PSY2}.

\begin{corollary}\label{comp cor 2}
Let $R$ be a commutative ring and $\frka$ a weakly pro-regular ideal. 
Let $K$ be a compact generator of $\cD_{\frka\textup{-tor}}(R)$. 
Then we have the following quasi-isomorphism of dg-algebras under $R$. 
\[
\BBic_{R}(K) \simeq \widehat{R}.
\] 
\end{corollary}

The condition that 
 the canonical functor 
$ \cD(\frka\textup{-tor}) \to \cD_{\frka\textup{-tor}} (R) $ gives an equivalence
is satisfied 
if the subcategory $\frka\textup{-tor}$ is closed under taking injective hull. 
This condition is satisfied if a right ideal $\frka$ has the Artin-Rees property.

Let  $R$ be a  Noetherian algebra. 
Namely  the center $Z=\textup{Cen}(R)$  is a Noetherian ring and $R$ is a finitely generated as  $Z$-module. 
If  a two-sided ideal $\frka$ is of form $\tilde{\frka}R$ for some ideal $\tilde{\frka}$ of $Z$, 
then  it has the  Artin-Rees property.
Let $\{a_1 , \dots, a_n\}$ be a generator of the ideal  $\tilde{\frka}$ of $Z$. 
It is also a generator of the ideal $\frka$ of $R$. 
Then 
we can prove that 
the Koszul complex $K=K(R;a_1,\dots,a_n)$ is a compact generator of $\cD_{\frka{\textup{-tor}}}(R)$  
by the same method of the proof of \cite[Proposition 6.1]{BN}.
Therefore we obtain the following corollary. 

\begin{corollary}
Let $R$, $\frka$ and $K$ be as above. 
Then we have the following quasi-isomorphism of dg-algebras under $R$. 
\[
\BBic_{R}(K) \simeq \widehat{R}.
\] 
\end{corollary}

\subsection{A proof of Theorem \ref{comp thm 1}. }\label{926 125}

We set  $\cI: =(\ZZ_{\geq  1})^{\textup{op}}$.  
Let  $\tilde{\varpi}^{n}: R \to R/\frka^n$ be the canonical projection for $n \geq 1$ 
and $\psi^{m,n}: R/\frka^m \to R/\frka^n$ the canonical projection for $m \geq n$. 
We take a cofibrant replacement $\rho: R' \xrightarrow{\sim} R$ of $R$. 
We denote by $\Phi: \cI \to \dga$ the functor 
which sends $n\geq 1$ to $R/\frka^n$ and $m\geq n$ to the canonical projection $R/\frka^m \to R/\frka^n$. 
The family of maps $\tilde{\varpi}^n \circ \rho:R' \to R/\frka^n$ induces the morphism 
$ \ct_{\cI}(R') \to \Phi$ in the functor category $\dga^{\cI}$. 
We decompose this morphism into  a cofibration $\varpi^{\bullet}: \ct_{\cI} R' \rightarrowtail \Phi'  $ 
followed by a trivial fibration $\Phi' \stackrel{\sim}{\twoheadrightarrow} \Phi$.  
Note that since the canonical projection $R/\frka^{n+1} \to R/\frka^n$ is fibration, 
we may assume that $\Phi'(\textbf{geq}^{m,n})$ is also fibration. 
By abusing notation 
we denote by the same symbol $\Phi'$ the functor $\cI \to \cC(R')$ which sends $n$ to $\Phi'_{n}$. 
We take a fibrant replacement  $\lambda: \Phi' \stackrel{\sim}{\twoheadrightarrow} \Phi"$ 
in the functor category $\cC(R')^{\cI}$ with the injective model structure where 
the category $\cC(R')$ is equipped with the injective model structure. 
We denote by $\pi^{\bullet}: \ct_{\cI}(R') \to \Phi"$ 
the composition $\ct_{\cI}(R') \to \Phi'$ with $\Phi' \to \Phi"$.  
We denote by $\Phi"_{n}$ the image of $n \in \cI$ by $\Phi"$.  
\[
\begin{xymatrix}{
R \ar@{->>}[d]_{\tilde{\varpi}^n} & 
*++{R'}\ar@{>->}[d]_{\varpi^n} \ar@{>->}@/^2pc/[dd]^{\pi^n} \ar@{->>}[l]_{\sim}
 \\
R/\frka^n & *++{\Phi'_n} \ar@{->>}[l]_{\sim}  \ar@{>->}[d]^{\wr} 
\\
& \Phi"_n 
}\end{xymatrix}\]

 We set $\cU$ to be $(\langle \Phi"_{1} \rangle^{\cC}_{R'/})^{\circ}$. 
We denote by $\iota: \cI \to \cU$ the functor which sends $n$ to $\pi^n: R' \to \Phi"_{n}$. 
We denote by $\tilde{\iota}: \cI \to \langle \Phi"_{1} \rangle^{\cC}_{R'/}$ the functor which sends $n$ to $\varpi^n: R' \to \Phi'_{n}$. 
By abuse of notations 
we denote by $\Gamma$ the co-domain functors $\cU \to \cC(R')$ and $\langle \Phi'_{1} \rangle_{R'/}^{\cC}$. 
We denote by $\kappa: \ct(R') \to \Gamma$ the canonical morphism. 
Then we have 
isomorphisms 
$\Gamma\circ \tilde{\iota} \cong \Phi'$ and $\Gamma \circ \iota \cong \Phi"$. 
We can check that under these isomorphisms we have $\kappa_{\tilde{\iota}} \cong \varpi^{\bullet}$ 
and $\kappa_{\iota} \cong \pi^{\bullet}$. 
\begin{equation}\label{101 1103}
\begin{xymatrix}{
\ct_{\cI}(R') \ar[d]_{\varpi^{\bullet}} \ar@/_2pc/[dd]_{\pi^{\bullet}} \ar@{=}[r] &
\ct_{\cI}(R') \ar[d]^{\kappa_{\tilde{\iota}}} \ar@/^2pc/[dd]^{\kappa_\iota } \\
\Phi' \ar[d] \ar@{=}[r]^{\sim}  & \Gamma \circ \tilde{\iota} \ar[d] \\
\Phi" \ar@{=}[r]^{\sim} & \G \circ \iota  
}\end{xymatrix}\end{equation}

We construct a bi-dual $S$ functor over $\Phi"_1$ by using Construction \ref{103 444}. 
Since there exists a injective cofibration $\pi^1: R' \cofright \Phi"_1$, 
we may assume that the evaluation map $\epsilon_{R'} : R' \to S(R')$ is an injective cofibration. 
Hence by Theorem \ref{bic holim thm} 
the canonical morphism  $S(\kappa) : \ct_\cU  S(R') \to S \circ\G$ exhibits $S(R')$ as a homotopy limit of $R'$. 
 
We have the commutative diagram 
\begin{equation}\label{comp diagram}
\begin{xymatrix}{
\ct_{\cI}(R') \ar[r] \ar[d]_{\kappa_{\tilde{\iota}}} \ar@/_2pc/[dd]_{\kappa_{\iota}} & 
\ct_{\cI}(S(R')) \ar[d]^{S(\kappa_{\tilde{\iota}})} \ar@/^3pc/[dd]^{S(\kappa_{\iota})}\\ 
\Gamma \circ \tilde{\iota} \ar[d] \ar[r] & S\circ \Gamma \circ \tilde{\iota} \ar[d] \\
\Gamma \circ \iota \ar[r] & S \circ \Gamma \circ \iota 
}\end{xymatrix}\end{equation}
where the horizontal morphisms are induced by the unit map $\epsilon: 1 \to S$. 
The key of the proof is the following proposition. 

\begin{proposition}\label{comp key prop}
The functor $\iota : \cI \to \cU$ is homotopy left cofinal. 
\end{proposition}

Before giving a proof of the proposition, 
we complete the proof of Theorem \ref{comp thm 1} 
with assuming Proposition \ref{comp key prop}. 

We take the limits of the above diagram (\ref{comp diagram}) 
\[
\begin{xymatrix}{
 & S(R') \ar@/^3.5pc/[dd]^{\simeq^{(1)}} \ar@{-->}[d]\\
\lim\Gamma \circ \tilde{\iota} \ar[d]_{\simeq^{(2)}} \ar@{-->}[r]   & 
\lim S\circ \Gamma \circ \tilde{\iota} \ar[d]^{\simeq^{(2)}}  \\
\lim\Gamma \circ \iota \ar[r]^-{\simeq^{(2)}} & \lim S\circ \Gamma \circ \iota  
}\end{xymatrix}
\]
By Proposition \ref{comp key prop} we check the quasi-isomorphism $(1)$ in the above diagram. 
By construction, for $n \geq 1$ 
the morphisms $\G\circ \iota(\textbf{geq}^{n+1,n}) = \Phi"(\textbf{geq}^{n+1,n})$ 
and $\G\circ \tilde{\iota}(\textbf{geq}^{n+1,n}) = \Phi'(\textbf{geq}^{n+1,n})$ 
is projective fibration. 
By Lemma \ref{main mis lemma 1}.(3) 
the same things hold for $S\circ \G \circ\iota $ and $S \circ \G \circ \tilde{\iota}$. 
Therefore by Lemma \ref{921 514} limits of those functors compute homotopy limits. 
Since homotopy limit is unique up to weak equivalence (Lemma \ref{923 017}), 
we check the quasi-isomorphisms $(2)$ in the above diagram. 
Consequently we see that 
the dotted arrows in the above diagram are quasi isomorphisms. 
Therefore the following morphisms are quasi-isomorphisms of dg $R'$-modules  
\begin{equation}\label{1016 158}
\lim_{n\to \infty }\Phi'_{n} \xrightarrow{\epsilon \, \sim } 
\lim_{n \to \infty}  S(\Phi'_{n}) \xleftarrow{S(\kappa) \,\sim} S(R'). 
\end{equation} 
Note that 
since 
the quasi-isomorphisms $\Phi'_{n} \xrightarrow{\sim} R/\frka^n$ are taken as  morphisms of dg-algebras, 
 therefore the morphism $\lim_{n\to \infty} \Phi'_{n} \to \lim_{n\to \infty} R/\frka^n = \widehat{R}$ 
 is also quasi-isomorphisms  of dg-algebras. 
Note also that 
the middle term $\lim_{n \to \infty} S(\Phi'_n)$ in the above diagram 
has no obvious dg-algebra structure. 

We denote by $\cB$ the bi-commutator $\dgBic_{R'}(\Phi"_1)= S(R')$. 
Observe that the middle term $W := \lim_{n \to \infty} S(\Phi'_n)$ has 
a left $\widehat{R'}:= \lim_{n \to \infty }\Phi'_n$-module structure.  
Indeed 
the left multiplication  $\lambda_r : \Phi'_n \to \Phi'_n$ 
of a element $r$ of  the dg-algebra $\Phi'_n$ 
is a homomorphism of $R'$-modules. 
Therefore $S(\Phi'_n)$ canonically has a structure of a left $\Phi'_n$-module, 
which is compatible with the projective systems. 
Hence $W$ has a left $\widehat{R'}$-module structure. 
On the other hands, 
$S(\Phi'_n)$ has a structure of a right  $\cB$-module structure 
which is  compatible with the projective system. 
Hence $W$ has a right $\cB$-module structure. 

Taking cohomology groups of diagram (\ref{1016 158}), 
we obtain isomorphism of graded modules 
\[
\widehat{R} \xrightarrow{ \alpha \, \cong } \cohm(W) \xleftarrow{ \cong \,\,\beta} \cohm(\cB).
\]
We can easily check that 
these morphisms satisfy the condition of the following lemma. 
Hence the map $\beta^{-1} \circ \alpha: \widehat{R} \to \cohm(\cB)$ 
is an algebra isomorphism.  
Moreover it is easy to see that 
this is compatible with the canonical homomorphisms 
$\textbf{comp}: R \to \widehat{R}$ and $\cohm(\iota): R \to \cohm(\cB)$. 

\begin{lemma}\label{1016 227} 
Let $A$ and $B$ be rings and $U$ a 
$A^{\textup{op}}\otimes B$-module. 
(i.e.  
a abelian group 
which has both 
left $A$-module structure and  a right $B$-module structure such that $(am)b= a(mb)$ 
for $a \in A $, $m \in M$ and $b\in B$.)  
Assume that a left $A$-isomorphism $\alpha: A \to U $ and a right $B$-isomorphism 
$\beta: B \to U$ are given. 
If $\alpha(1_A) = \beta(1_B)$, then the map $\beta^{-1} \circ \alpha  :A \to B$ is an algebra isomorphism. 
\end{lemma}

A proof is a straightforward calculation and left to the readers. 

Let $\tau^{\leq 0}: \dgCat \to \dgCat$ be the smart truncation functor.  
Then we have canonical natural trans formations 
 $\nu: \tau^{\leq 0} \to 1_{\dgCat}$ and $\nu': \tau^{\leq 0} \to \tuH^0(-)$. 
 Observe that $\nu$ and $\nu'$ are quasi-isomorphisms for dg-categories 
such that the cohomology groups of all Hom-complex  are concentrated in degree $0$.  
By the above isomorphism we see that  the cohomology of $\BBic_{R}(R/\frka) \simeq \cB$ 
is concentrated in degree $0$. 
Since the dg-algebra $R'$ is quasi-isomorphic to the ring $R$, 
the cohomology group of $R'$ is also concentrated in degree $0$. 
Therefore by the above natural transformations 
the canonical morphism $\iota: R' \to \BBic_R(R/\frka)$ corresponds to 
$\tuH^0(\iota) : R \to \tuH^{0}(\BBic_R(R/\frka))$. 
Hence 
we obtain  the following desired commutative diagram 
in the homotopy category $\Ho(\dgCat)$ 
\[
\begin{xymatrix}{ 
& R \ar[dl]_{\textbf{comp}}  \ar[d]_{\tuH^0(\iota)}  \ar[dr]^{\iota} & \\
 \widehat{R} \ar@{=}[r]^-{\sim}  & 
\tuH^0(\BBic_{R}(R/\frka )) & 
\BBic_{R}(R/\frka) \ar@{=}[l]_{\sim}. 
}\end{xymatrix}
\]

In the rest of this section we devote to prove Proposition \ref{comp key prop}. 
Namely we prove that the $\infty$-functor $\sner(\iota): \sner(\cI) \to \sner(\cU)$ is left cofinal. 

We denote by $V$ the full sub $\infty$ category of $\sner(\cU)$ 
consisting of $\pi^{n}$. 
Then we have $\sner(\iota)= \xi \circ \eta$
where  $\eta: \sner(\cI) \to V$ and $\xi : V\to \sner(\cU)$ 
be the canonical inclusions. 
By \cite[Proposition 4.1.1.3]{HTT} a composition of left cofinal functors is left cofinal.
Therefore  it is enough to prove that $\xi$ and $\eta$ are left cofinal. 
We prove these in Lemma \ref{comp lem 1} and Lemma \ref{comp lem 2}.

\begin{lemma}\label{comp lem 1}
The functor $\eta$ is left cofinal. 
\end{lemma}
 
\begin{proof}
By Theorem \ref{cofinal thm}
 it is enough to show that 
$\eta_{/n}$ is weakly contractible for each $n \geq 1$. 
To prove this, 
it is enough to show that 
any simplicial map $f: X \to \eta_{/n}$ from a finite simplicial set $X$  
factors through $\ast$ in the homotopy category $\textup{Ho}(\sSet)$. 

Recall that $\eta_{/n} = \sner(\cI) \times_{V} V_{/n}$. 
We set $f_{1} = \textup{pr}_1 \circ f, f_{2} = \textup{pr}_{2} \circ f$ 
where $\textup{pr}_{1} , \textup{pr}_{2}$ are the first and second projection. 
Since $\cI \cong (\ZZ_{\geq 1})^{\textup{op}}$, the simplicial map $f_{1}: X \to \sner(\cI)$ 
is homotopic to some constant map. 
More precisely, 
there exists a simplicial map $H' : X\times \Delta^1 \to \sner(\cI)$ 
such that 
the restriction $H'|_{X \times \{1\}}$ to $X \times \{1\} \cong X$ is $f_{1}$ 
and 
the restriction $H'|_{X \times \{0\}}$ to $X\times \{0\}$  is the constant functor 
with the value $m := \max\{f_{1}(x) \mid x\in X\}$.

Since the simplicial map $\{1\} \to \Delta^1$ is isomorphic to the horn inclusion 
$\Lambda^{1}_{1} \to \Delta^1$,  
by \cite[Corollary 2.1.2.7]{HTT} the simplicial map $X \times \{1\} \to X \times \Delta^1$ is right anodyne. 
Since by \cite{HTT} the co-domain functor $V_{/n} \to V$ is a right fibration, 
there exists a simplicial map $H":X \times \Delta^1 \to V_{/n}$ which complete the following commutative diagram:
\[
\begin{xymatrix}{
X \ar[d]\ar[rr]^{f_2} && V_{/n} \ar[d] \\
X\times \Delta^1 \ar[r]_{H'}\ar[urr]^{H"}  &\sner(\cI) \ar[r]_{\eta} & V 
}\end{xymatrix}
\]
Then the functor $H',H"$ gives the functor $H:X \times \Delta^1 \to \eta_{/n}$ 
such that the restriction $H|_{X \times \{1\}}$ is $f$. 
We denote by $g$ the restriction $H|_{X \times \{0\}}$. 
The map $H$ gives a homotopy between $f$ and $g$. 
Therefor it is enough to prove that 
$g$ factors through $\ast$ in the homotopy category $\textup{Ho}(\sSet)$.

Since  $g\circ \textup{pr}_1=H'|_{X\times \{0\}} $ is the constant map with the value $m$, 
the map $g$ factors through $\Homr_{V}(m,n)$, 
(Recall that $\Homr_{V}(m,n) = \{m \} \times_{V}V_{/n}$) 
\[\begin{xymatrix}{
X \ar@/^1pc/[drr]^{g} \ar@/_1pc/[ddr]  \ar@{-->}[dr] &&& \\
&\Homr_{V}(m,n) \ar[d]\ar[r] & \eta_{/n} \ar[d]_{\textup{pr}_1} \ar[r]^{\textup{pr}_2}  & V_{/n} \ar[d] \\
 & \{m\} \ar[r] & \sner(\cI) \ar[r]_{\eta} & V 
 }\end{xymatrix}\]
By  \cite[2.2]{HTT}, the complex  $\Homr_{V}(m,n) $ is weak homotopy equivalent to $\Map_{\cU}(\pi^m,\pi^n)$. 
We finishes the proof by showing that $\Map_{\cU}(\pi^m,\pi^n)$ is weakly contractible. 
We check that $\pi_i(\Map_{\cU}(\pi^m,\pi^n))=0$ for $i \geq 0$. 
Recall that the mapping complex $\Map_{\cU}(\pi^m,\pi^n)$ fits into the pull-back diagram 
\[
\begin{xymatrix}{
\Map_{\cU}(\pi^m,\pi^m) \ar[d] \ar[r] &
\Map_{R'}(\Phi"_m,\Phi_n") \ar[d]^{\textup{ind}} \\
\{\pi^n\} \ar[r] & 
\Map_{R'}(R',\Phi_n")  
}\end{xymatrix}\]
where $\textbf{ind}$ is the induced morphism $\textbf{ind} := \Map(\pi^m,\Phi_n")$. 
Since $\pi^m$ is a injective cofibration and $\Phi_n"$ is injectively fibrant, 
the map $\Map(\pi^m,\Phi")$ is a fibration of simplicial set. 
Therefore 
we the homotopy groups $\pi_i(\Map_{\cU}(\pi^m,\pi^n))$ fits into the homotopy long exact sequence.
\[
\begin{split}
 \pi_{i+1} (\Map_{R'}(\Phi"_m,\Phi_n") ) \xrightarrow{\pi_{i+1}(\textbf{ind})} &
\pi_{i+1}(\Map_{R'}(R',\Phi_n")) \to \pi_i(\Map_{\cU}(\pi^m ,\pi^n)) \\
\to& \pi_{i} (\Map_{R'}(\Phi"_m,\Phi_n") ) \xrightarrow{\pi_i(\textbf{ind})}
\pi_{i}(\Map_{R'}(R',\Phi_n")) 
\end{split}\]

The canonical exact sequence $0 \to \frka^m \to R \xrightarrow{\tilde{\varpi}^m} R/\frka^m \to 0$ 
induces an exact triangle 
\[
\RR\Hom_{R}(\frka^m[1], R/\frka^n) \to 
\RR\Hom_R(R/\frka^m,R/\frka^n) \xrightarrow{\widetilde{\textbf{ind}}}
 \RR\Hom_R(R,R/\frka^n) \to . 
\]
where $\widetilde{\textbf{ind}}$ is the morphism induced from $\tilde{\varpi}^n$. 
 Since we have  
\[
\textup{H}^{-i}(\RR\Hom_{R}(\frka^m[1],R/\frka^n)) \cong \Ext^{-i-1}(\frka^m,R/\frka^n)=0, 
\]
the morphism $\tuH^{0}(\widetilde{\textbf{ind}})$ is injective and 
the morphisms $\tuH^{i}(\widetilde{\textbf{ind}})$ are isomorphisms for $i \geq 1$. 
On the other hands, 
 we have the commutative diagram for $i\geq 0$ 
\[
\begin{xymatrix}{
\pi_i(\Map_{R'}(\Phi"_m,\Phi"_n)) \ar[d]_{\pi_i(\textbf{ind})} \ar@{=}[r]^-{\sim}  & 
\tuH^{-i}(\RR\Hom_R(R/\frka^m,R/\frka^n)) \ar[d]^{\tuH^{-i}(\widetilde{\textbf{ind}})} \\
\pi_i(\Map_{R'}(R',\Phi"_n)) \ar@{=}[r]^-{\sim} & 
\tuH^{-i}(\RR\Hom_R(R,R/\frka^n)) 
}\end{xymatrix}
\]
Hence 
from the above homotopy long exact sequence, 
we conclude that 
all the homotopy groups of $\Map_\cU(\pi^m,\pi^n)$ vanish. 
This finishes the proof. 
\end{proof}

\begin{lemma}\label{comp lem 2}
The functor $\xi$ is left cofinal. 
\end{lemma}

\begin{proof}
We can prove that 
the $\infty$-category $\sner(\cU)$ is co-filtered 
as in the same way of the proof of Lemma \ref{107 225}. 
In view of Lemma \ref{cofinal lem} 
it is enough to show that 
for any morphism  $k: R' \rightarrowtail K$ in  $\cC(R')$ 
with $K \in \left(\langle \Phi"_1 \rangle^{\cC}\right)^{\circ}$  
there exists a morphism $\pi^{n} \to k$ for some $n \geq 1$.  
By Lemma \ref{model lem} 
it is enough to prove the same statement in the derived category $\cD(R')$. 

The quasi-isomorphism $R' \xrightarrow{\sim} R$ induces the equivalence $\cD(R') \simeq \cD(R)$. 
Therefore the problem is reduced to prove the same statement in $\cD(R)$. 
Observe that  
under the equivalence, the object $\Phi"_1$ is isomorphic to $R/\frka$ 
and hence 
the thick subcategory $\langle \Phi"_{1} \rangle^{\cD} $ is equivalent to 
the thick subcategory $\langle R/\frka \rangle^{\cD}$.  
By the assumption that $\cD_{\frka\textup{-tor}}(R) \simeq \cD(\frka\textup{-tor})$, 
an object $K$ of $\langle R/\frka \rangle$ is quasi-isomorphic to
 a complex  each terms  of which is a $\frka$-torsion $R$ module. 
Hence we immediately see that 
 in the derived category $\cD(R)$ 
the morphism $k: R \to K$ with $K\in \langle R/\frka \rangle^{\cD}$ 
factors through the canonical projection $R \to R/\frka^{n}$ for some $n$.  
By Lemma \ref{model lem} we obtain a morphism $\pi^n \to k$ in $\cU$.   
\end{proof}

Now we finish the proof of Theorem \ref{comp thm 1}.

\section{Smashing localization  via derived bi-commutator} 

\subsection{Smashing localization of dg-categories}

A functor $F: \cT \to \cS$ between triangulated categories is called a \textit{smashing localization} 
if $F$ has a fully faithful right adjoint $G:\cS \hookrightarrow \cT$ which preserves direct sums. 
A dg-functor $f:\cA \to \cB$ of small dg-categories is called \textit{smashing localization (or homological epimorphisms)} 
if the restriction functor $f_{\ast}: \cD(\cB) \to \cD(\cA)$ is fully faithful. 
In the study of smashing localization, 
we mainly interest in the functors $\cD(\cA) \to \cD(\cB)$ 
between derived categories.(See e.g. \cite{Krause, NS}.)
Therefore we consider smashing localizations $f: \cA \to \cB$ of a dg-category $\cA$ 
up to Morita equivalence of $\cB$. 

Let $f: \cA \to \cB'$ be a smashing localization. 
We decompose $f$ into a cofibration $g: \cA \cofright \cB$ 
followed by a trivial fibration $h : \cB \trivfibright \cB'$. 
Then $f$ and $g$ induce the same smashing localization of 
the derived category $\cD(\cA)$. 
Hence in the study of smashing localization, 
we may assume that $f:\cA \to \cB$ is a cofibration of dg-categories. 
Moreover by the following lemma, 
we may replace $f$ with the essential image of a cofibration.

\begin{lemma}\label{smashing lem 1}
Let $f: \cA \to \cB'$ be a smashing localization.
We decompose $f$ into an essentially surjection $g: \cA \to \cB$ 
 followed by a fully faithful dg-functor  $h: \cB \to \cB'$. 
Then  $h$ induces Morita equivalence, 
i.e., $h_{\ast}:\cD(\cB') \to \cD(\cB) $ gives an equivalence,  
and hence $g$ is also smashing localization. 
\end{lemma}

\begin{proof}
First we show that $h_{\ast}$ is fully faithful. 
Since $f_{\ast}$ is fully faithful, the counit morphism $f^{\ast}f_{\ast} N \to N$ is an isomorphism for any $N \in \cD(\cB')$. 
In particular every object of $\cD(\cB')$ is of form $f^{\ast}K$ for some $K \in \cD(\cA)$. 
Therefore it is enough to prove that 
the map $\Hom_{\cD(\cB')}(f^{\ast}K,f^{\ast}L) \to \Hom_{\cD(\cB)}(h_{\ast}f^{\ast}K,h_{\ast}f^{\ast}L),\psi \mapsto h_{\ast}(\psi)$ 
is an isomorphism for $K,L \in \cD(\cA)$.  
However it is easily shown by using the equation $f^{\ast}= h^{\ast}g^{\ast}$
 and the natural isomorphism $h^{\ast}h_{\ast} h^{\ast} \cong h^{\ast}$.

Since $h:\cB \to \cB'$ is fully faithful, 
the unite map $b^{\wedge} \to h_{\ast}h^{\ast}(b^{\wedge})$ is an isomorphism. 
Recall that we have a canonical isomorphism $g^{\ast}(a^{\wedge}) \cong (ga)^{\wedge}$. 
Therefore 
we have $h_{\ast}f^{\ast}(a^{\wedge}) \cong h_{\ast}h^{\ast}g^{\ast}(a^{\wedge}) \cong g^{\ast}(a^{\wedge})$. 
Since $g$ is essentially surjective, the set $\{g^{\ast}(a^{\wedge}) \mid a \in \cA\}$ 
of modules equals to the set $\{ b^{\wedge} \mid b \in \cB \}$. 
Therefore the modules $h_{\ast}(f^{\ast}(a^{\wedge}))$ form a set of compact generators of $\cD(\cB)$. 
On the other hands 
by \cite[Lemma in page 1231]{NS} the modules $(fa)^{\wedge} \cong f^{\ast}(a^{\wedge})$ 
form a set of compact generators of $\cD(\cB')$. 
Since the functor  $h_{\ast}$ is exact and commutes with arbitrary direct sum, 
we conclude that $h_{\ast}$ is essentially surjective.
\end{proof}

A notion of pure monomorphism and pure injective object for compactly generated triangulated category 
were introduced  by Krause \cite{Krause} to study smashing localizations.  
\begin{definition}\label{pure inj def}
(1) 
A morphism $M\to N$ in $\cD(\cA)$ is said to be a \textit{pure monomorphism} 
if the induced map $\Hom_{\cD(\cA)}(P,M) \to \Hom_{\cD(\cA)}(P,N)$ is a monomorphism 
for all $P \in \Perf \cA$. 
 
(2) 
A morphism $M\to N$ in $\cD(\cA)$ is said to be a \textit{cohomologically monomorphism} 
if the induced map $\tuH^{\ast}(M)\to \tuH^{\ast}(N)$ is a monomorphism.

(3) 
An object $J$ of $\cD(\cA)$ is called \textit{pure injective} 
(resp. \textit{cohomologically injective} (C.I. for short.)) 
if $\psi:M\to N$ is pure monomorphism (resp. cohomologically monomorphism),  
then every morphism $M \to J$ factors through $\psi$. 
\end{definition}

We call an object $E$ of $\cD(\cA)$ a \textit{co-generator} if, for $M\in \cD(\cA)$, 
the condition $\Hom_{\cD(\cA)}(M,J)=0$ implies $M =0$.  
(We remark that 
this definition is not a direct opposite version of generators for triangulated categories. See e.g. \cite[3.1.1.]{Ro}. ) 
We call an  object $J$ of $\cC(\cA)$ a pure injective, a C.I. , 
a co-generator 
if it is so as an object of $\cD(\cA)$. 

We can easily prove the following lemma. 
\begin{lemma}
(1) A pure injective mono morphism is a cohomologically   monomorphism. 

(2) A C.I.  object is a pure injective object. 
\end{lemma}

\begin{example}\label{101 941}
Let $R$ be a ring and $J$ be an injective $R$-module. 
Then $J$ is a C.I.  object. 
Moreover if $J$ is an injective co-generator of $\Mod R$, 
then $\Pi_{n \in \ZZ} J[n]$ is a C.I.  co-generator of $\cD(R)$. 

\end{example}

Note that by \cite[Proposition 2.6]{Krause} if $f: \cA \to \cB$ is a smashing localization of dg-categories 
then the restriction functor $f_{\ast}: \cD(\cB) \to \cD(\cA)$ sends  pure-injective objects to pure-injective object. 
The following is main theorem of this section, 
which  says that 
every smashing localization is obtained as a derived bi-commutator of some pure-injective object.

\begin{theorem}\label{loc thm 0}
Let $f: \cA \to \cB $ be an essentially surjective smashing localization 
such that the induced adjoint pair 
\[
f^{\ast} : \cC(\cA) \rightleftarrows \cC(\cB) : f_{\ast}
\]
is a Quillen pair where 
both $\cC(\cA)$ and $\cC(\cB)$ are equipped with the injective model structures. 
Let $J$ be a pure injective co-generator of $\cD(\cB)$. 
Then we have quasi-equivalence of dg-categories over $\cA$ 
\[
\cB \simeq \BBic_{\cA}(f_{\ast}J')
\]
where 
$J'$ is a large enough product $J^{\Pi \kappa}$ of $J$. 
\end{theorem}

\begin{remark}
Nicol\'as and Saorin \cite{NS} 
 proved that 
for any  smashing localization $F:\cD(\cA) \to \cS$  
there exists  a subcategory $\cI \subset \cD(\cA)$ 
such that 
the functor $\LL \iota_{\cI}^{\ast} : \cD(\cA) \to \cD(\dgBic_{\cA}(\cI))$ 
induced from the canonical morphism $\iota_{\cI}: \cA \to \dgBic_{\cA}(\cI)$ 
is equivalent to $F$. 

\end{remark}
   
In a proof of Theorem \ref{loc thm 0} 
an essential part is the following theorem.

\begin{theorem}\label{loc thm 1}
 Let  $J$ be an injective fibrant cohomologically  injective co-generator. 
 We fix a set  $\cM$ of objects of $\cC(\cA)$.    
If we take a product $J':= J^{\Pi \kappa}$ of copies of $J$ over a large enough cardinal $\kappa$, 
then for each $M$ in $\cM$, the bi-dual $M^{\wcdast}$ of $M$ taken over $J'$ 
is isomorphic to $M$. 

More precisely,
if we choose an injectively fibrant  representative of $J'$ appropriately, 
then  
the evaluation map $\epsilon_M: M \to S(M) $ is a quasi-isomorphism. 
\end{theorem}

\begin{proof}[A proof of Theorem \ref{loc thm 0} with assuming Theorem \ref{loc thm 1}]
By Theorem \ref{loc thm 1}, 
there exists an injective fibrant representative of a fixed  injective co-generator $J \in \cD(\cB)$,  
which will  be  also  denoted by $J' \in \cC(\cB)$,  
such that if we take a large enough cardinal $\kappa$, 
the the product $J':= J^{\Pi \kappa}$ satisfies the following property: 
let $S_{J'}$  be the bi-dual taken over $J'$,  
for any $b\in \cB$ 
the evaluation map $\epsilon_{b^{\wedge}}: \bwe \to S_{J'}(\bwe)$ is a quasi-isomorphism. 
Then by Lemma \ref{1010 321} the canonical functor $\iota_{J'}: \cB \to \BBic_{\cB}(J')$ 
is quasi-equivalence. 
By Proposition \ref{1010 314} 
we have the commutative diagram 
in $\Ho(\dgCat)$ 
\[
\begin{xymatrix}{
\cA \ar[rr]^{f} \ar[drr]_{\iota_{J}} && \cB  \ar@{=}[d]^{\wr}\\
&& \BBic_{\cA}(f_{\ast}J').
}\end{xymatrix}\]
Now we finish the proof. 
\end{proof}

\subsection{A proof of Theorem \ref{loc thm 1}.}

We denote by $\cohm(\cA)$ the homotopy category of the dg-category $\cA$. 
Namely this is a graded category 
such that  the objects of $\cohm(\cA)$ is the same with that of $\cA$ 
and the Hom-graded module is given by $\cohm(\cA)(a,b):= \cohm\left(\cA^{\bullet}(a,b)\right)$.

By a version of \cite[Theorem 1.8. Corollary 1.9]{Krause} (see also \cite{GP}), we can easily deduce the following lemma. 
\begin{lemma}\label{101 1201}
(1) 
Let $M$ be an object of $\cD(\cA)$ and $J$ a cohomologically  injective object of $\cD(\cA)$.
Then the map $\Hom_{\cD(\cA)}(M,J) \to \Hom_{\cG(\cohm(\cA))}(\cohm(M),\cohm(J))$ 
induced from the functor $\cohm: \cD(\cA) \to \cG(\cohm(\cA))$ is an isomorphism. 

(2) 
An object $J \in \cD(\cA)$ is a C.I.  object (resp. injective co-generator) 
if and only if the cohomology group $\tuH^{\bullet}(J)$ is a injective object (resp. injective co-generator) of $\cG(\tuH^{\bullet}(\cA))$. 

(3) 
Let $J$ be a C.I.  co-generator of $\cD(\cA)$. 
Then for any object $M \in \cD(\cA)$ there exists a C.I.  morphism $M \to J^{\Pi \kappa}$ 
for some cardinal $\kappa$. 

(4) Let $J$ be a C.I.  co-generator of $\cD(\cA)$.
 The graded  $\tuH^{\bullet}(\cA)$-module $\tuH^{\bullet}(M)$ admits a resolution 
\begin{equation}\label{101 607}
0 \to \cohm(M) \to \cohm(J^{\Pi \kappa_{0}^M}) \to \cohm(J^{\Pi \kappa_1^M}) \to \cohm(J^{\Pi \kappa_2^M}) \to \cdots. 
\end{equation}
for some cardinal $\kappa_{i}^M$ for $i \geq 0$. 
\end{lemma}

Let $\cM$ be a set of objects of $\cD(\cA)$. 
The following theorem tells us that 
if we set $\kappa := \sup\{ \kappa_{i}^M \mid M \in \cM ,\,i \geq 0\}$ 
where $\kappa^M_{i}$ is a cardinal appearing in the above resolution (\ref{101 607}),  
then the bi-duality $S(M)$ taken over $J^{\Pi\kappa}$ is quasi-isomorphic to $M$. 

\begin{theorem}\label{loc thm 2} 
Let $J$ be a cohomological injective object of $\cD(\cA)$.  
We fix a set $\cM$ of objects of $\cC(\cA)$.  
Assume that 
for each $M\in \cM$ 
there exist objects $J^{i}_M$ for $i \geq 0$ 
such that 
the graded  $\tuH^{\bullet}(\cA)$-module $\tuH^{\bullet}(M)$ admits a resolution 
\begin{equation}\label{101 1200}
0 \to \cohm(M) \xrightarrow{ \tau_M} 
\cohm(J^0_M) \xrightarrow{d^0_M} \cohm(J^1_M) \xrightarrow{d^{1}_M} \cohm(J^2_M) \xrightarrow{d^{2}_M} \cdots 
\end{equation}  
 whose each term $J^i_M$ is a direct summand of a finite direct sum of $J_M$.
Then 
the evaluation map $\epsilon_{M}: M \to S_J(M)$ is a quasi-isomorphism. 

More precisely,
if we choose an injectively fibrant  representative of $J$ appropriately, 
then  
the evaluation map is a quasi-isomorphism 
\[
\epsilon_M: M \to S(M)
\]

\end{theorem}

By Example \ref{101 941} we obtain the following corollary. 
\begin{corollary}\label{loc cor}
Let $R$ be a ring and $J$ an injective $R$-module. 
Assume that
an $R$-module $M$ admits an injective resolution 
\[
0 \to M \to J^0 \to J^1 \to J^2 \to \cdots 
\]
such that each term $J^i$ is a direct summand of finite direct sum of $J$. 
Then 
the evaluation map  $\epsilon_{M}: M \to M^{\wcdast}$ is a quasi-isomorphism. 
\end{corollary}

\begin{remark}
This corollary is already obtained by Shamir \cite{Shamir} 
in a different way. 
\end{remark}
In the rest of this section we devote to prove Theorem \ref{loc thm 2}.

We proceed a proof of Theorem \ref{loc thm 2}. 
Since the derived bi-duality functor $S(-)$ preserves quasi-isomorphisms,  
we may assume that all objects $M$ of $\cM$  are injective fibrant. 

Using the following Lemma \ref{106 1614}  
we inductively  
construct a $\#$-exact sequence of dg $\cA$-modules for each $M \in \cM$ 
\begin{equation}\label{105 740} 
0\to M \xrightarrow{\delta^{-1}_M} \tilde{J}^{0}_M \xrightarrow{\delta^0_M} \tilde{J}^{1}_M \xrightarrow{\delta^{1}_M} 
\tilde{J}^{2}_M \xrightarrow{\delta^{2}_M} \cdots
\end{equation}
such that 
\begin{enumerate}

\item 
 cohomology $\cohm(-)$ coincide  with 
the given resolution (\ref{101 1200}), 
(hence $\tilde{J}^i_M$ is an injectively  fibrant representative of $J^i_M$), 

\item
each $\coker(\delta^{i}_M)$ is injectively fibrant for $i \geq -1$.  
\end{enumerate}

\begin{lemma}\label{106 1614}
Let $\tilde{N}$ be an injectively fibrant object of $\cC(\cA)$. 
Assume that a morphism $r: \cohm(\tilde{N}) \to \cohm(L)$ is given 
with some C.I.  object $L$ of $\cD(\cA)$. 
Then 
there exists an injectively fibrant replacement $\tilde{L}$ of $L$ 
and a morphism $\underline{r}: \tilde{N} \to \tilde{L}$ 
such that 
$\cohm(\underline{r}) = r$ and 
$\coker(\underline{r} )$ is injectively fibrant. 
\end{lemma} 

\begin{proof}
By Lemma \ref{101 1201} 
there exists a morphism $r': \tilde{N} \to L$ in $\cD(\cA)$ such that $\cohm(r') = r$. 
We pick an injectively fibrant replacement $\tilde{L}'$ 
and a representative $\underline{r}': \tilde{N} \to \tilde{L}'$ of $r'$. 
Let  $C := \cone(1_{\tilde{N}})$ be the cone of the identity $1_{\tilde{N}}$ 
and $\iota:={}^{t}(1_{\tilde{N}},0): \tilde{N} \to C $  the canonical inclusion. 
We set $\tilde{L} := \tilde{L}' \oplus C$ and $r := {}^{t}(\underline{r}', \iota): \tilde{N} \to \tilde{L}$.
Then we have the following commutative diagram 
\[
\begin{xymatrix}{ 
0\ar[r] & \tilde{N} \ar@{=}[d] \ar[r]^{\underline{r}} & \tilde{L} \ar[d]^{(0,1_C)} \ar[r] & 
\coker(\underline{r}) \ar[d] \ar[r] & 0\\
0 \ar[r]& \tilde{N} \ar[r]^{\iota} & C \ar[r] & 
\tilde{N}[1] \ar[r] &0
}\end{xymatrix}
\]
where the top and the bottom rows are $\#$-exact. 
By snake Lemma we obtain 
 the $\#$-exact sequence $0 \to \tilde{L}' \to \coker(\underline{r}) \to \tilde{N}[1] \to 0$. 
Therefore by Lemma \ref{101 1250} we conclude that $\coker(\underline{r})$ is injectively fibrant. 
Since the cone $C$ is weakly contractible, we see that $\cohm(\underline{r}) = r$.  
\end{proof}
We choose an injectively fibrant representative $\tilde{J}$ of $J$. 
Then for $M \in \cM$ and $i \geq 0$, 
the module $\tilde{J}^i_M$ belongs to $(\langle  \tilde{J} \rangle^{\cC})^{\circ}$.

The above resolution is a main tool for the proof. 
We need to fix  notations. 
From now for simplicity
we denote $\delta^{i}_M$, $\tilde{J}^i_M$ and $\tilde{J}_M$ by $\delta^i_M$, $J^i$ and $J$ 
respectively. 
Since we prove that the evaluation map $\epsilon_M$ is a quasi-isomorphism for each $M$ separately, 
these modifications are harmless for our purpose. 
 We set $M^0 := M$ and $M^{n}=\coker(\delta^{n-2})$ for $n \geq 1$.  
We denote by $\lambda^n: M^{n} \hookrightarrow J^n$ and $\rho^: J^{n} \twoheadrightarrow M^{n+1}$ 
 the canonical morphisms. 
Therefore we have  $\delta^{n}: = \lambda^{n+1} \circ \rho^n$. 
\[
\begin{xymatrix}{
\ar[r]^{\delta^{n-1}} & J^n \ar[rr]^{\delta^n} \ar@{->>}[dr]_{\rho^n} &&
J^{n+1} \ar[r]^-{\delta^{n+1}} & \\ 
&& M^{n+1} \ar@{^{(}->}[ur]_{\lambda^{n+1}} &&
}\end{xymatrix}
\]
For $n \leq m$ 
we denote by $I^{[n,m]}$ the totalization of the complex 
\[
J^n \xrightarrow{\delta^n} J^{n+1} \xrightarrow{\delta^{n+1}} \cdots \xrightarrow{\delta^{m-1}} J^m.
\]
More precisely, 
the dg $\cA$-module $I^{[n,m]}$ is defined in the following way:    
the underlying graded module of $I^{[n,m]}$ is given by 
\[
 J^m[-(m-n)] \bigoplus J^{m-1}[-(m-n-1)] \bigoplus \cdots \bigoplus J^{n+1}[-1] \bigoplus J^n
\]
and the differential $d_{I^{[n,m]}}$ is given by 
\[
(-1)^n
\begin{pmatrix}
(-1)^{m-n}d_{J^m} & \delta^{m-1}[-(m-n-1)] & 0 & \cdots & 0 & 0 \\ 
0 & (-1)^{m-n-1}d_{J^{m-1}} & \delta^{m-2}[-(m-n-2)] & \cdots & 0& 0\\
0&0  & (-1)^{m-n-2}d_{J^{m-2}} &  \cdots  &0&0  \\ 
\vdots & \vdots & \vdots &  & & \\
0&0&0& \cdots &-d_{J^{n+1}} & \delta^{n}\\ 
0& 0& 0& \cdots & 0 & d_{J^n}. 
\end{pmatrix} 
\]
We denote by $\pi^m_n$ the morphism 
${}^{t}(0,0, \cdots 0,\lambda^n) : M^n \to I^{[n,m]}.$
For $\ell> m > n$ we denote by $\varphi^{\ell,m}_{n}$ 
the morphism  $I^{[n,\ell]} \to I^{[n,m]}$ induced from the morphism of complexes 
\[
\begin{xymatrix}{
J^{n} \ar[d]^{1_{J^n}} \ar[r] & J^{n+1} \ar[d]^{1_{J^{n+1}}} \ar[r] & *++++{\cdots} \ar[r] & 
 J^{m} \ar[d]^{1_{J^{m}}} \ar[r] & J^{m+1} \ar[d]^0 \ar[r] & 
*++++{\cdots} \ar[r] &  J^{\ell}\ar[d]^0\\
J^{n}  \ar[r] & J^{n+1}  \ar[r] & *++++{\cdots} \ar[r] &  J^{m} \ar[r] &
0 \ar[r] & *++++{\cdots} \ar[r]  &0
 }\end{xymatrix}    
\]  
Then we have $\pi^m_n = \varphi^{\ell,m} \circ \pi^\ell_n$. 
 We set $I^n:= I^{[0,n]}$,$\pi^n := \pi^n_0$ and $\varphi^{m,n}:= \varphi^{m,n}_0$.\[
\begin{xymatrix}{
& *++{M} \ar@{>->}[d]_{\pi^{n}} \ar@{>->}[drr]^{\pi^{n-1}} & \\
*+++{\cdots} \ar[r]_{\varphi^{n+1,n} }& I^{n} \ar[rr]_{\varphi^{n, n-1}} && I^{n-1} \ar[r]_{\varphi^{n-1,n-2}}& *+++{\cdots} &.  
}\end{xymatrix}\]
The limit $\lim_{n \to \infty} I^n$ is the totalization of $J^0 \to J^1 \to \cdots $.  
By \cite{ddc} the system $\{\pi^n \}_{n \geq 0} $ of morphisms induces a quasi-isomorphism
\[
\pi^{\infty}: M \xrightarrow{\sim} \lim_{n\to \infty }I^{n}. 
\]

We denote by $\cU$ 
the full subcategory $(\langle J \rangle^{\cC}_{M/})^{\circ}$. 
We denote by $\cI $ the category $(\ZZ_{\geq 0})^{\textup{op}}$. 
We define a functor $\Phi: \cI \to \cU$ 
to be a functor which sends an object $n$ to $\pi^n$ 
and a unique morphism $\textbf{geq}^{m,n}: m \to n$ to $\varphi^{m,n}$.

The key of the proof of Theorem \ref{loc thm 2} 
is the following proposition. 

\begin{proposition}\label{loc key prop}
The functor $\Phi$ is homotopy left cofinal. 
\end{proposition}

First we assume Proposition \ref{loc key prop} and prove Theorem \ref{loc thm 2}. 
We have the commutative diagram 
\begin{equation}\label{101 1134}
\begin{xymatrix}{
\ct_{\cI} M \ar[d]_{\kappa} \ar[r]^{\epsilon_M} & 
\ct_{\cI} S(M) \ar[d]^{S(\kappa)} \\ 
\G|_{\cI} \ar[r]_{\epsilon_\G} & 
S\circ \G|_{\cI} . 
}\end{xymatrix}\end{equation} 
The bottom arrow is a weak equivalence. 
By Proposition \ref{loc key prop} the right vertical arrow exhibits $S(M)$ as a homotopy limit of $S\circ \G$. 
We claim that the left vertical arrow exhibits $M$ as a homotopy limit of $\G|_{\cI}$. 
Indeed 
we see that there exists  a canonical  isomorphism $\lim_{\cI}\G|_{\cI} \cong \lim_{n \to \infty} I^{n}$ 
and under this isomorphism the morphism $\lim\kappa : M \to \lim_{\cI} \G|_{\cI} $ 
corresponds to the quasi-isomorphism $\pi^{\infty} : M \stackrel{\sim}{\rightarrow} \lim_{n \to \infty } I^n$.  
Since each $\G(\textbf{geq}^{n,n-1}) =\varphi^{n,n-1}$ is a projective fibration, 
the limit $\lim \G|_{\cI}$ computes  a homotopy  limit of $\G|_{\cI}$. 
Therefore we verify the claim.  

Taking limits of the above diagram (\ref{101 1134}), 
by the uniqueness of homotopy limit up to weak equivalence (Lemma \ref{923 017}),   
we conclude that 
the evaluation map $\epsilon_M: M \to S(M) $ is a quasi-isomorphism. 
This finishes the proof of Theorem \ref{loc thm 2} 
except Proposition \ref{loc key prop}.

In the rest of this section we devote to prove Proposition \ref{loc key prop}.
According to Theorem \ref{cofinal thm}, 
 it is enough to verify  that 
the simplicial set $\sner(\Phi)_{/k}$ is non-empty and  weakly contractible for each object $k: M \to K$ of $\cU$. 
First we prove  
\begin{lemma}\label{104 153}
The simplicial set $\sner(\Phi)_{/k}$ is non-empty. 
Namely  for some $m \geq 0$  
there exists a morphism $\psi: I^{m} \to K$ 
which satisfies $k= \psi \circ \pi^m$.
\end{lemma}

To prove this lemma and complete the proof of Proposition \ref{loc key prop}, 
we need little more informations of the  resolution (\ref{105 740}). 
We denote by $\xi^m_n$ the canonical projection $I^{[n,m]} \to \coker(\pi^m_n)$. 
\begin{lemma}\label{tri fib lem}
There exists a injective trivial fibration 
$\eta^{m}_n: \coker(\pi^m_n) \stackrel{\sim}{\twoheadrightarrow} M^{m+1}[-(m-n)]$ 
such that the composite morphism $\eta^{m}_n\circ \xi^m_n$ 
is equal to $(\rho^m[-(m-n)],0,\cdots 0)$. 
\end{lemma}

First we claim 
\begin{lemma}\label{loc cone fib lem}
(1) 
Assume that the following commutative diagram 
such that the top row is $\#$-exact 
is given in $\cC(\cA)$
\[
\begin{xymatrix}{
0 \ar[r] & X \ar@{-->}[d]_{\alpha} 
\ar[r]^{f} &
Y \ar@{=}[d] \ar[r] & Z \ar[d]^{h} \ar[r] & 0\\
& \cocone(g) \ar[r]^{(0,1_{Y'})} & Y' \ar[r]^{g} & Z' &.
}\end{xymatrix}\]
Then 
the morphism   $\alpha : = {}^{t}(0,f) :X \to  Z'[-1]\oplus Y'$ of graded modules 
become the morphism $\alpha : X \to \cocone(g)$  in $\cC(\cA)$ and completes the above diagram. 
Moreover we have a canonical isomorphism 
$\coker(\alpha) \cong \cone(h)[-1]$.

(2)
Assume that 
the commutative diagram such that the vertical arrows are injective fibrations is given in $\cC(\cA)$  
\[
\begin{xymatrix}{
X \ar@{->>}[d]^{\wr} \ar[r]^{f} & Y \ar@{->>}[d]^{\wr} \\
X' \ar[r]_{f'} & Y'. 
}\end{xymatrix} 
\]
Then the induced morphism $\cone(f) \to \cone(f')$ is an injective fibration. 

(3) 
Let $f:X \to Y$ be a morphism of dg $\cA$-modules. 
We denote by $g$ the canonical projection $Y \to \coker(f)$. 
Then we have the $\#$-exact sequence of dg $\cA$-modules 
\[
0\to \ker(f) \to \cone(1_X) \xrightarrow{
\textup{diag}(f,1_{X[1]})
} \cone(f) \xrightarrow{(g,0)} \coker(f) \to 0
\]
In particular, 
if $f$ is an  injective  cofibration and $X$ is injective fibrant, 
then the morphism $(g,0)$ is a trivial injective fibration. 

\end{lemma}
The proof is straightforward.
The last statement of (3) follows from Lemma \ref{922 142} and Lemma \ref{101 1250}.

\begin{proof}[Proof of Lemma \ref{tri fib lem}]
We prove by induction on $m \geq n$. 
For the case $m=n$, 
it is enough to set $\eta^{n}_n:= 1_{M^{n+1}}$, since $\coker(\pi^n_n) = M^{n+1}$.
We assume that 
the case when $m$ is verified. 
Observe that 
the module $I^{[n,m+1]}$ is the co-cone of 
$(\delta^{m}[-(m-n)],0,\dots ,0): I^{[n,m]} \to J^{m+1}[-(m-n)]$ 
and that 
the associated  morphism $I^{[n,m+1]}\to I^{[n,m]}$ coincides with $\varphi^{m+1,m}_n$. 
Then by Lemma \ref{loc cone fib lem}.(1), 
we obtain the following commutative diagram   
\[
\begin{xymatrix}{ 
*++{M^n} \ar@{>->}[rr]^{\pi^m_n} \ar@{>->}[dd]_{\pi^{m+1}_n} &&
I^{[n,m]} \ar@{=}[d] \ar@{->>}[rrr]^{\xi^m_n} &&
& \coker(\pi^m_n) \ar@{->>}[d]^{\wr \eta^m_n}  \\
&&
I^{[n,m]}\ar[rrr]_{(\rho^m[-(m-n)], 0,\cdots , 0 )\quad\quad} \ar@{=}[d]
&&& *++{M^{m+1}[-(m-n)]} \ar@{>->}[d]^{\lambda^{m+1}[-(m-n)]} \\
I^{[n,m+1]} \ar[rr]_{\varphi^{m+1,m}_n} && I^{[n,m]} \ar[rrr]_{(\delta^m[-(m-n)], 0,\cdots , 0 )\quad\quad} &&& 
J^{m+1}[-(m-n)].   
}\end{xymatrix}
\]

Then by Lemma \ref{loc cone fib lem}, 
we have the following sequence of morphisms 
\[
\begin{split}
\coker(\pi^{m+1}_n) &\cong{} \cone(\lambda^{m+1}[-(m-n)] \circ \eta^m_n)[-1] \\
                     & \stackrel{\sim}{\twoheadrightarrow}{} \cone(\lambda^{m+1}[-(m-n)])[-1]\\
                     & \stackrel{\sim}{\twoheadrightarrow}{} \coker(\lambda^{m+1}) [-(m-n+1)] 
                     \cong M^{m+2}[-(m+1-n)] 
\end{split}
\]
where for $a =1,2,3$ we obtain the $i-th$ morphisms  by Lemma \ref{loc cone fib lem} (a). 
Note that 
to show that the 3-ed morphism is a trivial injective fibration, 
we use the fact that $M^{m+2}$ is injectively fibrant.       
We denote by $\eta^{m+1}_n$ the above composition. 
Then it follows from Lemma \ref{loc cone fib lem} that 
$\eta^{m+1}_n \circ \xi^{m+1}_n =(\rho^{m+1}[-(m+1-n)], 0,\cdots , 0 )$. 
\end{proof}

Note that 
for $\ell > m>n$ we have the following commutative diagram 
\begin{equation}\label{105 833}
\begin{xymatrix}{ 
*++{M^n} \ar@{>->}[rr]^{\pi^m_n} \ar@{>->}[dd]_{\pi^{\ell}_n} &&
I^{[n,m]} \ar@{=}[d] \ar@{->>}[rrr]^{\xi^m_n} &&
& \coker(\pi^m_n) \ar@{->>}[d]^{\wr \eta^m_n}  \\
&&
I^{[n,m]}\ar[rrr]_{(\rho^m[-(m-n)], 0,\cdots , 0 )\quad\quad} \ar@{=}[d]
&&& *++{M^{m+1}[-(m-n)]} \ar@{>->}[d]^{\pi^{\ell}_{m+1}[-(m-n)]} \\
I^{[n,\ell ]} \ar[rr]_{\varphi^{\ell,m}_n} && I^{[n,m]} \ar[rrr] &&& 
I^{[m+1,\ell]}[-(m-n)]
}\end{xymatrix}
\end{equation} 
where the term $I^{[n,\ell]}$ is obtained as  the co-cone of the bottom right arrow 
$I^{[n,m]} \to I^{[m+1,\ell]}[-(m-n)]$.

\begin{proof}[Proof of Lemma \ref{104 153}] 
%
By Lemma \ref{model lem}, it is enough to show that  
in the derived category $\cD(\cA)$ 
there exists  a morphism $\psi: I^{n} \to K$ for some $n$  
such that $k= \psi \circ \pi^n$. 

First 
we recall the definition of the $n$-generated thick subcategory 
$\langle J \rangle_{n}^{\cD}$ introduced in  \cite{Ro}. 
Let $\cT$ be a triangulated category. 
For a full subcategory  $\cS$ of   $\cT$ 
we denote by $\langle \cI \rangle_1^{\cD}$ the smallest full  subcategory  of $\cT$ 
containing $\cI$ which is closed under taking shifts, 
finite direct sums, direct summands and isomorphisms. 
For full subcategories $\cS$ and $\cR$ of $\cT$ 
we denote by $\cS \ast \cR$ the full subcategory of $\cT$ 
consisting of those objects $T \in \cT$ 
such that 
there exists an exact triangle 
$S \to T \to R \to $ 
with $S \in \cS$ and $R \in \cR$. 
Set $\cS \diamond \cR := \langle \cS \ast \cR\rangle_{1}^{\cD}$.  
For $n \geq 2$  we define inductively  
\[
\langle \cS \rangle_{n}^{\cD} := 
\langle \cS \rangle_{n-1}^{\cD} \diamond \langle \cS \rangle_{1}^{\cD} 
. 
\]
Note that 
$\langle \cS \rangle^{\cD}  = \bigcup_{n \geq 1}\langle \cS \rangle^{\cD}_{n}$.

Let $k : M \to K$ be a morphism in $\cD(\cA)$ with $K \in \langle J \rangle^{\cD}$. 
We claim that if $K \in \langle J \rangle_{n+1}$ then there exists morphism $\pi^n \to k$. 
We prove this claim by induction on $n$. 
In the case when $n=0$, 
$K$ is a direct summand of a finite direct sum of $J$. 
Hence it is a  cohomologically  injective object.   
Since by definition the morphism $\pi^0 :M \to J^0$ is cohomologically  monomorphism, 
we have a morphism $\pi^0 \to k$. 
We assume that the claim is proved for $n$. 
We may assume that $K$ fits into an exact triangle 
$K \to K^{n} \to K^0\to K[1]$ 
for some $K^{n} \in \langle J \rangle^{\cD}_{n+1}$ and 
$K^0 \in \langle J \rangle^{\cD}_{1}$. 
By induction hypothesis, 
the composite morphism $M \xrightarrow{k} K \to K^{n}$ factors through $\pi^{n}$. 
We set $M^{n+1} := \image(d^{n})$. 
Then  we have the following commutative diagram 
except for the  dotted arrow:
\[\begin{xymatrix}{
M \ar[r]^{\pi^{n}} \ar[d]_{k} & I^{n} \ar[d] \ar[r] & M^{n+1} \ar@{-->}[d]^{k'} \ar[r] & M[1]\ar[d]^{k[1]} \\
K \ar[r]             & K^{n} \ar[r] & K^{0} \ar[r] & K[1].
}\end{xymatrix}\]
where the top and bottom rows are exact.  
By an axiom of triangulated category, 
we have a morphism $k': M^n \to K^0$ 
which complete the above commutative diagram. 
Since 
the morphism $\lambda^{n+1} :M^{n+1} \to J^{n+1}$ constructed 
in the resolution (\ref{105 740}) 
is cohomologically monomorphism, 
there exists a morphism $h': J^{n+1} \to K^0$ such that 
$k'= h' \circ \iota^{n+1}$.

Recall that $I^{n+1}$ is the co-cone of the composite morphism  $I^{n} \to M^{n+1} \xrightarrow{\lambda^{n+1}} J^{n+1}$. 
We have the following commutative diagram except for the dotted arrow: 
\[
\begin{xymatrix}{
 I^n \ar@{=}[d] \ar[r] & M^{n+1} \ar[d]_{\lambda^{n+1}} \ar[r]  & M[1] \ar[d]_{\pi^{n+1}[1]} \ar@/^2.7pc/[dd]^{k[1]}  \\
 I^{n} \ar[d] \ar[r] & J^{n+1} \ar[d]_{h'} \ar[r]& I^{n+1}[1] \ar@{-->}[d]_{h}\\
 K^{n} \ar[r] & K^0 \ar[r] & K[1]
}\end{xymatrix}
\]
Since $I^n[1]$ is a homotopy fiber co-product of $J^{n+1}$ and $M[1]$ under $M^{n+1}$, 
there exists a morphism $h: I^{n+1}[1] \to K[1] $ 
which complete the above commutative  diagram \cite[Section 1.3]{Neemantricat}. 
Therefore the morphism $ h[-1]: I^{n+1} \to K $ give a desired factorization.    
\end{proof}

In a similar way, 
we can prove more general statement 
by using the diagram (\ref{105 833}).

\begin{lemma}\label{loc cofinal lem}
Let $k: M^n \cofright K$ be an injectively cofibrant morphism of dg $\cA$-modules with $K \in \langle J \rangle^{\cC}$. 
Then there exists $m \geq n$ such that 
there exists a morphism $\psi: I^{[n,m]} \to K$ 
which satisfies $k= \psi \circ \pi^m_n$.
\end{lemma}

To finish the proof of Proposition \ref{loc key prop}, 
in the rest of this section 
we devote to prove that 
the simplicial set $\sner(\Phi)_{/k}$ is weakly contractible.

We  will show that 
 any simplicial map $f:X \to \sner(\Phi)_{/k}$ with finite simplicial set $X$ 
 factor $\ast$ in the homotopy category $\textup{Ho}(\sSet)$. 
By the same method of the proof of Lemma \ref{comp lem 1} 
 we may assume that 
 the composition morphism $f_1=\pr_1 \circ f : X \to \sner(\cI)$ 
 is a constant morphism with the value, say, $n$, 
 and that 
 the map $f$ factors $\Homr_{\sner(\cU)}(\pi^n,k)$. 
\[\begin{xymatrix}{
X \ar@/^1pc/[drr]^{f} \ar@/_1pc/[ddr]  \ar@{-->}[dr]^{f'} &&& \\
&\Homr_{\sner(\cU)}(\pi^n,k) \ar[d]\ar[r] &
 \sner(\Phi)_{/k} \ar[d]_{\textup{pr}_1} \ar[r]^{\textup{pr}_2}  & \sner(\cU)_{/k} \ar[d] \\
 & \{n\} \ar[r] & \sner(\cI) \ar[r]_{\Phi} & \sner(\cU)
 }\end{xymatrix}\]
 
Let $\ell$ be an integer greater  than $n$. 
We denote by $[\textbf{uni}^{\ell,n}]$   
the simplicial map $\Delta^1 \to \sner(\cI)$ corresponding to 
the morphism $\textbf{uni}^{\ell,n}$.  
 Let $H$ be the composite map 
 \[
 H: X \times\{1\} \xrightarrow{\alpha  \times \iota_1}  
 \Delta^0 \times \Delta^1  \cong \Delta^1 \xrightarrow{[\textbf{uni}^{\ell,n}]} \sner(\cI). 
\]
 where $\alpha : X \to \Delta^0$ is a unique map. 
We can check that   the following diagram is commutative except the dotted arrow  
 \[
 \begin{xymatrix}{
 X \times \{1\} \ar[d]_{\alpha \times \iota_1 } \ar@{=}[r]^-{\sim} & 
 X \ar[r]^-{f} & \sner(\Phi)_{/k} \ar[d]^{\pr_1} \\
 X \times \Delta^1 \ar@{-->}[urr]^{H'}  \ar[rr]_{H} && \sner(\cI) 
 }\end{xymatrix}
 \]
 By the same consideration as in the proof of Lemma \ref{comp lem 1}, 
 we see that 
 there exists a dotted arrow which completes above commutative diagram.
Observe that  the composite map 
\[\pr_1 \circ H' \circ (1_X \times \iota_0) :
X\times \{0\} \to X \times \Delta^1 \to \sner(\Phi)_{/k} \to \sner(\cI) 
\] 
factors through $\{\ell\} \to \sner(\cI)$. 
Since $\Homr(\pi^{\ell},k)$ is  the  fiber product $\{\ell\} \times_{\sner(\cU)} \sner(\cU)_{/k}$, 
the composite map $H' \circ (1_X \times \iota_0)$ is decomposed 
into a map $g': X \to \Homr(\pi^\ell,k) $ followed by a canonical map 
 $\Homr(\pi^\ell,k) \to \sner(\Phi)_{/k}$.  
 We define  $\Homr(\varphi^{\ell,n}, k)$ to be the fiber product
  $\{\textbf{uni}^{\ell,n}\} \times_{\sner(\cU)} \sner(\cU)_{/k}$. 
  Then  we have the commutative diagram 
\[
\begin{xymatrix}{
X \times \{0\} \ar[d]_{g'} \ar[r]^{ 1_X \times \iota_0} &
X\times \Delta^1  \ar[d]_{H'}  & 
Z \times \{1\}  \ar[d]^{f'} \ar[l]_{1_X\times \iota_1} \\
\Homr(\pi^\ell ,k ) \ar[r] & 
\Homr(\varphi^{\ell,n},k)  & 
\Homr(\pi^n,k) \ar[l] 
}\end{xymatrix}\]

By \cite{HTT} for $\ell \geq n$ we have  the commutative diagram 
\[
\begin{xymatrix}{ 
X \ar[rr]^-{f'} \ar@/_1.3pc/[drr]_-{g'} && \Homr(\pi^n,k)  \ar[d]^{(\varphi^{\ell,n})^{\ast}} \\
&& \Homr(\pi^\ell,k)  
}\end{xymatrix}
\]
where the vertical arrow $(\varphi^{\ell,n})^{\ast}$ 
is the morphism induced from the morphism $\varphi^{\ell,n}: \pi^\ell \to \pi^n$. 
Therefore it is enough to prove that 
for some $\ell\geq n$ we have the following homotopy commutative diagram in $\sSet$ 
\[
\begin{xymatrix}{
X\ar[r]^-{f'} \ar[d] & \Homr(\pi^n,k) \ar[d]^{(\varphi^{\ell,n})^\ast} \\
\ast \ar[r] & \Homr(\pi^m,k). 
}\end{xymatrix}
\]
(For simplicity we denote by $[-,+]= \Hom_{\textup{h}\sSet}(-,+)$ 
the Hom set of the homotopy category $\textup{h}\sSet$.) 
In other words, 
for every  element $f'$ of $[X,\Homr(\pi^n,k)]$ 
there exists a natural number $\ell\geq n$ 
such that 
the image of $f'$ by the induced morphism 
$[X,\Homr(\pi^n,k)] \to [X, \Homr(\pi^\ell,k)]$ 
lies in the image of the morphism 
$[\ast,\Homr(\pi^\ell,k)] \to [X,\Homr(\pi^\ell,k)]$ 
induced from a unique map $ X \to \ast$. 

By \cite[2.2]{HTT} we have a weak homotopy  equivalence $\Homr(\pi^n,k)\simeq \Map_{\cU}(\pi^n,k)$. 
On the other hand 
we have a  natural isomorphism
$\Map_{\sSet}(X, \Map_{\cU}(\pi^n,k)) \cong \Map_{\cU}(\pi^n,k^X)$
Therefore we have a natural isomorphism $[X,\Homr(\pi^n,k)] \cong \Hom_{\textup{h}\cU}(\pi^n,k^X)$.
Hence it is enough to show that 
for every morphism $\alpha: \pi^n \to k^X$ in $\cU$ 
there exists a natural number $\ell\geq n$ 
such that we have the following homotopy commutative diagram in $\cU$ 
\begin{equation}\label{loc hoshii diagram}
\begin{xymatrix}{
\pi^\ell \ar[d]_{\varphi^{\ell,n}} \ar[r] & k \ar[d]^{\beta}\\
\pi^n \ar[r]_{\alpha} & k^X
}\end{xymatrix}
\end{equation}
where we denote by $\beta$ the morphism induced from a unique map $X \to \ast$.

By Lemma \ref{loc cofinal lem} 
there exists a morphism $\beta':\pi^m \to k$ in $\cU$. 
Replacing $m$ with $\max\{m,n\}$,   
we may assume that $m \geq n$. 
We set $\alpha": = \alpha \circ \varphi^{m,n} $ and $\beta" := \beta \circ \beta'$. 
We also denote by $\alpha"$ and $\beta"$ the morphisms $I^{m} \to K^X$, $I^m \to K^X$ 
between the co-domains respectively. 
We define a object $k_1: M \to K_1$  of $\cU$ 
to be a morphism $(0,{}^{t}(\pi^m,\pi^m)): M \to \cocone(\alpha",-\beta")$. 
Since $(\alpha",-\beta")\circ {}^{t}(1_{I^m} ,1_{I^m}) \circ \pi^m= \alpha"\circ \pi^m -\beta"\circ \pi^m = 
k^X -k^X=0$, 
there exists a morphism $\zeta: \coker(\pi^m) \to K^X$ 
which completes the following commutative diagram 
\begin{equation}\label{loc diagram 1}
\begin{xymatrix}{
M \ar[d]_{k_1} \ar[r]^{\pi^m} & I^m \ar[d]^{{}^{t}(1_{I^m},1_{I^m})} \ar[r]^{\xi^m} & \coker(\pi^m)\ar[d]^{\zeta}\\
K_1 \ar[r]_{\gamma\quad } & I^{m} \bigoplus I^m \ar[r]_{(\alpha", -\beta")} & K^X.
}\end{xymatrix}\end{equation} 
where $\gamma$ is a canonical morphism for the  co-cone construction \ref{1010 101}.
We set $\gamma_a := \pr_a \circ \gamma: K_1 \to  I^m \oplus I^m \to I^a$ for $a =1,2$. 
Then we have an equation $\pi^m = \gamma_a \circ k_1$ of morphisms in $\cC(\cA)$. 
Therefore $\gamma_a$ induces a morphism $k_1 \to \pi^m$, 
which will also be denoted by $\gamma_a$. 
In the same way of the proof of Lemma \ref{106 1146}, 
from the above diagram (\ref{loc diagram 1}) 
we obtain the homotopy commutative diagram and an equation of morphisms in $\textup{h}\cU$ 
\begin{equation}\label{loc diagram 3} 
\begin{xymatrix}{
k_1 \ar[d]_{\gamma_1} \ar[r]^{\gamma_2} & \pi^m \ar[d]^{\beta"} \\
\pi^m \ar[r]_{\alpha"} & k^X,  
}\end{xymatrix}\qquad 
\alpha" \circ \gamma_1 = \beta" \circ \gamma_2. 
\end{equation}

We denote by $\pi^m \oplus \pi^m$ the object 
${}^{t}(\pi^m,\pi^m):M \to I^m \bigoplus I^m$ of $\cU$. 
We define a morphism  $\delta: \pi^m \to \pi^m\oplus \pi^m$ in $\cU$ 
to be the morphism induced from ${}^{t}(1_{I^m},1_{I^m}): I^m \to I^{m} \bigoplus I^m$. 
Now we have equations of morphisms in $\cU$. 
\begin{equation}\label{loc diagram 2}
\begin{split}
1_{\pi^m}= \pr_1 \circ \delta,\,\, & \pr_1 \circ \gamma = \gamma_1, \\
1_{\pi^m}= \pr_2 \circ \delta, \,\,& \pr_2 \circ \gamma = \gamma_2.  
\end{split}
\end{equation}

By Lemma \ref{tri fib lem} 
we have  a injective trivial fibration $\eta^m:\coker(\pi^m) \trivfibright M^{m+1}[-m]$. 
Since every object of $\cC(\cA)$ is injectively cofibrant, 
therefore using the lifting property,  
we can check that  the morphism $\eta^m$ splits. 
Hence $\coker(\pi^m)$ is isomorphic to $M^{m+1}[-m] \bigoplus N$ for some weakly contractible module $N$. 
From now we identify $\coker(\pi^m)$ with $M^{m+1}\bigoplus N$. 
By Lemma \ref{tri fib lem} the first component 
$\xi_1^m$ of $\xi^m : I^m \to \coker(\pi^m)\cong M^{m+1}[-m]\bigoplus N$ is 
equal to $(\rho^m,0,\cdots,0)$. 
Let $\zeta_1: M^{m+1}[-m] \to K^X$ be the first component of $\zeta: \coker(\pi^m) \to K^X$. 
Then by Lemma \ref{loc cofinal lem} 
there exists a natural number $\ell\geq m+1$ such that 
there exists a morphism $\psi: I^{[m+1,\ell ]}[-m] \to K^X$ 
which satisfies $\zeta_1= \psi\circ \pi^{\ell}_{m+1}[-m]$. 
We denote  by $\omega$ the morphism 
\[
\begin{pmatrix} \pi^\ell_{m+1}[-m]& 0 \\ 0 & 1_N \end{pmatrix}: M^{m+1}[-m]\bigoplus N \to I^{[m+1,\ell]}[-m]\bigoplus N
\]
Then we have 
the equation 
$\zeta = (\zeta_1,\zeta_2)= (\psi,\zeta_2)\circ \omega$. 
We denote by $K_2$ the co-cone of $\omega\circ \xi^m : I^m \to I^{[m+1,\ell]}[-m]\bigoplus N$.  
We set $ K_2: = \cocone(\omega \circ \xi^m)$. 
By Lemma \ref{loc cone fib lem} 
the morphism ${}^{t}(0,\pi^m): M \to K_2$ of graded modules 
become a morphism in $\cC(\cA)$, which will be denoted by $k_2$. 
Now we have the following commutative diagram except dotted arrow. 
\begin{equation}\label{loc diagram 4}
\begin{xymatrix}{
 &M\ar@/_2pc/[dd]_{k_1} \ar[d]^{k_2} \ar[r]^{\pi^m} & I^m \ar@{=}[d] \ar[r]^-{\xi^m} & 
M^{m+1}[-m]\bigoplus N \ar[d]^{\omega} \ar@/^5pc/[dd]^{\zeta}\\
&K_2 \ar@{-->}[d]^{\psi'} \ar[r]_{\gamma'} & I^m \ar[d]^{(1_{I^m},1_{I^m})} \ar[r]^-{\omega\circ \xi^m} 
& I^{[m+1,\ell]}[-m] \bigoplus N \ar[d]^{(\psi,\zeta_2)}&\\
&K_1 \ar[r]_{\gamma} & I^m \bigoplus I^m \ar[r]_{(\alpha",-\beta")} & K^X &
}\end{xymatrix}\end{equation}
Since $K_1$ and $K_2$ are defined to be the co-cones, 
a morphism $\psi': K_2 \to K_1$ which completes the above diagram 
 is induced from the lower right square.   
The above diagram 
 gives the following commutative diagram  and an equation of morphisms in $\cU$  
\begin{equation}\label{loc diagram 5}
\begin{xymatrix}{
k_2 \ar[d]_{\psi'} \ar[r]^{\gamma'} & \pi^m \ar[d]^{\delta} \\
k_1 \ar[r]_-{\gamma} & \pi^m \oplus \pi^m, 
}\end{xymatrix}
\qquad 
\delta \circ \gamma' = \gamma \circ\psi'. 
\end{equation}
On the other hand, 
since $(1,0) \circ \omega = (\pi^m[-m],0)$, 
we have the following commutative diagram except the dotted arrow in $\cC(\cA)$ 
\begin{equation}\label{loc diagram 6}
\begin{xymatrix}{
 &M\ar@/_2pc/[dd]_{\pi^{\ell}} \ar[d]^{k_2} \ar[r]^{\pi^\ell_{m+1}} & 
 I^m \ar@{=}[d] \ar[rr]^-{\xi^m} &&
M^{m+1}[-m]\bigoplus N \ar[d]^{\omega} \ar@/^5pc/[dd]^{(\pi^m[-m],0)}\\
&K_2 \ar@{-->}[d]^{\psi"} \ar[r]_{\gamma'} & I^m \ar@{=}[d] \ar[rr]^-{\omega\circ \xi^m }& 
& I^{[m+1,\ell]}[-m] \bigoplus N \ar[d]^{(1,0)}&\\
&I^{\ell} \ar[r]_{\varphi^{\ell,m}} & 
I^m  \ar[rr]_{((\pi^\ell_{m+1}\circ\rho^m)[-m],0,\cdots 0 )} & & 
I^{[m+1,\ell]}[-m] &
}\end{xymatrix}\end{equation}
Recall that 
$K_2$ and $I^\ell$ are given as the co-cones, 
a morphism $\psi": K_2 \to I^\ell $ which completes the above diagram 
is induced from the lower right square. 
Note that 
since $N$ is weakly contractible, the induced  morphism $\psi": K_2 \to I^\ell$ is quasi-isomorphism. 
By Lemma \ref{925 1053} 
the morphism $\psi": k_2 \to \pi^\ell$ is weak homotopy equivalence. 
The above diagram (\ref{loc diagram 6}) gives the following commutative diagram 
and an equation in $\cU$ 
\begin{equation}\label{loc diagram 7} 
\begin{xymatrix}{
k_2 \ar[r]^{\gamma'} \ar[d]_{\psi"} & \pi^m \ar@{=}[d] \\ 
I^\ell \ar[r]_{\varphi^{\ell,m} } & \pi^m 
}\end{xymatrix}
\qquad 
 \varphi^{\ell,m}\circ \psi"= \gamma'.
\end{equation}
Combining the equations  (\ref{loc diagram 2},\ref{loc diagram 3},\ref{loc diagram 5},\ref{loc diagram 7}), 
we obtain a homotopy commutative diagram in $\cU$ 
\begin{equation}\label{loc diagram 8}
\begin{xymatrix}{
 \pi^{\ell} \ar[d]_{\varphi^{\ell,m}} \ar[r] & \pi^m \ar[d]^{\beta"}\\
 \pi^m \ar[r]_{\alpha"} & k^X.  
 }\end{xymatrix}\end{equation} 
We explain  the combining process. 
First note that 
the morphism $\psi"$ is a weak equivalence in $\cU$. 
Therefore from the equation (\ref{loc diagram 7}),
 we obtain an equation $\varphi^{\ell,m} = \gamma'\circ \psi"^{-1}$ in the homotopy category $\textup{h}\cU$.
 From  the equations (\ref{loc diagram 2}, \ref{loc diagram 5}) 
 we deduce the  equation in  the first line below for $a= 1,2$.  
 Finally, using the equation (\ref{loc diagram 3})
  we obtain the  equation in the bottom line. 
  \[\begin{split}
 \varphi^{\ell,m}= \pr_a \circ\delta \circ \gamma'\circ \psi"^{-1} 
  = \pr_a \circ \gamma \circ \psi' \circ \psi"^{-1}
  = \gamma_{a} \circ \psi'\circ\psi"^{-1}, \\
\alpha" \circ \varphi^{\ell,m}=\alpha"\circ \gamma_{1} \circ \psi'\circ\psi"^{-1}
= \beta\circ \gamma_{2} \circ \psi'\circ\psi"^{-1} =\beta" \circ \varphi^{\ell,m}.
 \end{split}\]

Recalling the definitions of $\alpha",\beta"$, 
we see that the diagram (\ref{loc diagram 8}) gives a desired diagram (\ref{loc hoshii diagram}). 
Now  the proof of Theorem \ref{loc  thm 1} is completed. 

\appendix

\section{Conceptually this paper is very simple.}

This paper is lengthy, because we need to work with homotopy theory. 
However the ideas behind Main theorem and applications are so simple that 
we require the reader to have knowledge of elementary category theory 
and homological algebra.

\subsection{Main Theorem}\label{obs for bic holim thm} 
We explain that 
an elementary observation leads to the main theorem \ref{out main thm} of this paper. 

We use the notations in Introduction. 
So we have the duality over $J$. 
\[
(-)^{\circledast} :=\RR\Hom_{\cA}(-,J) :
 \cD(\cA) \rightleftarrows \cD(\cE)^{\textup{op}}: \RR\Hom_{\cE}(-,J)=: (-)^{\circledast}
\]
We denote by $\langle J \rangle$ 
the smallest thick subcategory containing $J$. 
We claim  that 
if $K$ belongs to $\langle J \rangle$,
 then the evaluation map  $\varepsilon_{K} : K \to K^{\wcdast} $ is an isomorphism.  
Indeed for the case $K=J$ is clear. 
Since the bi-dual $(-)^{\wcdast}$ is an exact functor, 
we can check the claim  for general $K \in \langle J \rangle$. 

We fix a dg $\cA$-module $M$. 
It follows  from the above claim that  
every morphism $k: M \to K$ with $K \in \langle J \rangle$ 
factors though $\epsilon_{M}: M \to M^{\wcdast}$. 
\[
\begin{xymatrix}{
M \ar[d]^{k} \ar[r]^{\epsilon_{M}} & M^{\wcdast} \ar[dl]^{k^{\vee}} \\
K  &     ,
}\end{xymatrix}
\quad \quad 
\begin{xymatrix}{
M \ar[d]^{k} \ar[r]^{\epsilon_{M}} & M^{\wcdast} \ar[d]^{k^{\wcdast}} \\
K  &     \ar[l]_{\epsilon_{K}^{-1}}^{\cong}  K^{\wcdast}
}\end{xymatrix}
\]
It seems that 
the derived bi-dual module $M^{\wcdast}$ 
satisfies one of the two conditions of the limit of the family $M \to K$ of morphisms. 
In the following way, 
we can catch a glimpse of the other condition 
that 
we can reach  from $K \in \langle J \rangle $ to $M^{\wcdast}$.

It is   well-known  that 
a dg-module is obtained as a filtered homotopy colimit perfect modules. 
Hence  the dg $\cE$-module $M^{\circledast}$ is  quasi-isomorphic to the homotopy colimit 
of   some family $\{P_\lambda \}_{\L}$ of perfect $\cE$-modules.
\[M^{\circledast} \simeq  \hocolim_{\L}P_{\lambda} \] 
Applying the dual functor $(-)^{\circledast}$ to this quasi-isomorphism, 
 we obtain  the quasi-isomorphisms 
\[
M^{\wcdast} \simeq (\hocolim_{\L}P_{\lambda})^{\circledast} \simeq \holim_{\L}(P_{\lambda}^{\circledast}). 
\]
It is clear that $\cE^{\circledast} \simeq J$. 
Therefore, since $P_{\lambda}$ is a perfect $\cE$-module, 
the dual $P_\lambda^{\circledast}$ belongs to $\langle J \rangle$.  
This shows that 
we can reach  from $K \in \langle J \rangle $ to $M^{\wcdast}$. 

These observations suggest  that $M^{\wcdast}$ is the  limit of $k:M \to K$ with $K \in \langle J \rangle$. 
Actuarially it become true after some modification.

\subsection{Localization Theorem}\label{App 2}
We explain that 
the view point that 
a bi-duality is a completion, naturally leads a  proof of Theorem \ref{out loc thm}.

We discuss the case when $\cA$ is an ordinary ring, $M$ an $\cA$-module and $J$ 
be an injective co-generator of $\Mod \cA$. 
Then the module $M$ has an injective resolution by the products of $J$
\[
0 \to M \to J^{\Pi\kappa_0} \to J^{\Pi \kappa_1} \to J^{\Pi \kappa_2} \to \cdots .
\]
We can reduce the problem  the following theorem 
by setting $\kappa : =\sup\{ \kappa_i \mid i\in \ZZ\}$. 

\begin{theorem}
We take an injective resolution $M \xrightarrow{\sim} J^{\bullet}$ of $M$. 
\begin{equation}\label{1015 1003}
0 \to M \to J^{0} \to J^{1} \to J^{2} \to \cdots .
\end{equation}
Assume that $J^{i}$ is a direct summand of $J$. 
Then the evaluation map  $\epsilon_M: M \to M^{\wcdast}$ is a quasi-isomorphism. 
\end{theorem}


We explain an outline of a proof. 
\begin{Assumption}
We assume that $\holim = \lim$.  
\end{Assumption}

We denote by $I^n$ the totalization of the $n$-th truncated resolution.  
\[
I^n := \textup{tot}[ J^{0} \to J^{1} \to \cdots \to J^{n}].
\]
Then
by assumption 
the complex  $I^n$ belongs to the thick subcategory  $\langle J \rangle$ 
generated by $J$. 
Therefore  the canonical morphism $\pi^n : M \to I^n$ 
belongs to the under category $\langle J \rangle_{M/}$. 
Moreover we have a canonical morphism $\varphi^{n+1,n} : I^{n+1} \to I^n$ for $n \geq 0$ 
which is compatible with $\pi^n$. 
\[\begin{xymatrix}{
& 
*++{M} \ar[d]_{\pi^{n}} \ar[drr]^{\pi^{n-1}} 
&&& \\
*+++{\cdots} \ar[r]_{\varphi^{n+1,n} } 
& 
I^{n} \ar[rr]_{\varphi^{n, n-1}}
 && I^{n-1} \ar[r]_{\varphi^{n-1,n-2}}& *+++{\cdots} &.  
}\end{xymatrix}\]
Note that since the limit $\lim_{n \to \infty } I^{n}$ is 
the totalization of the injective resolution (\ref{1015 1003}), 
the morphisms $\{ \pi^n\}$ induces a (quasi-)isomorphism $M \to \lim_{n \to \infty} I^n$.   

We denote by $\cI$ 
the subcategory of $\langle J \rangle_{M/}$ 
consisting of objects $\pi^n: M \to I^n$ and of morphisms $\phi^{m,n}:\pi^m\to \pi^n$ 
so that $\cI$ is isomorphic to $(\ZZ_{\geq 0})^{\textup{op}}$. 
Then we have 
\[
\lim_{\cI}\Gamma|_{\cI} \cong \lim_{n \to \infty}I^n \simeq M.  
\]
Therefore
by Theorem \ref{out main thm}
 it is enough to prove that the subcategory $\cI \subset \langle J \rangle_{M/}$ is left co-final. 
Namely for each $k \in \langle J \rangle_{M/}$ the over category $\cI_{/k}$ is non-empty and connected.

We recall an elements of Homological algebra: 
Let $M'$ be another $\cA$-module and $M \xrightarrow{\sim} J'^{\bullet}$  
an injective resolution. 
Assume that an $\cA$-homomorphism $f : M \to M'$ is given. 
Then (1) 
there exists a morphism $\psi : J^{\bullet} \to J'^{\bullet} $ of complexes 
which completes the commutative diagram 
\[
\begin{xymatrix}{
M \ar[d]_{f} \ar[r] & J^{\bullet} \ar[d]^{\psi} \\
M' \ar[r] & J'^{\bullet}.  
}\end{xymatrix}
\]
(2) This morphism $\psi$ is not uniquely determined. 
(3) However it is uniquely determined up to homotopy. 

Using the same methods of the proof of (1), 
we can check that $\cI_{/k}$ is non-empty. 
By the same reason with (2), 
the category $\cI_{/k}$ is \textit{not} connective. 
However 
 in the same way of the proof of (3) 
we can verify that $\cI_{/k}$ is \textit{homotopy connective}. 
We explain a little bit more about this 
in the special case when 
the co-domain $K$ of $k: M \to K$ is an injective module: 

Since the canonical morphism $\pi^0: M \to I^0 = J^0$ is injective, 
there exists an extension  $\psi: I^0 \to K$ of $\pi^0$. 
This shows that $\cI_{/k} \neq \emptyset$. 
However there is no canonical choice of an extension.  
Moreover 
since the degree $0$-part of the canonical morphism $\varphi^{n,0}:I^n \to I^0$ 
is the identity map $1_{J^0}: J^0 \to J^0$, 
two extensions $\psi$ and $\psi'$ are not connected to each other 
in $\cI_{/k}$, 
unless $\psi= \psi'$. 
Nevertheless 
we can  see that 
for any two extensions $\psi$ and $\psi'$ 
there exists a homotopy commutative diagram 
\[
\begin{xymatrix}{
\pi^1 \ar[d]_{\varphi^{1,0}} \ar[r]^{\varphi^{1,0}} & 
\pi^0 \ar[d]^{\psi} \\ 
\pi^0 \ar[r]_{\psi'} & k 
}\end{xymatrix} 
\]
This shows that the objects $\psi$ and $\psi'$ of $\cI_{/k}$ 
is \textit{homotopically connected} to each other in $\cI_{/k}$. 
Therefore it is inevitable to work with homotopy theory.

{Graduate School of Mathematics, Nagoya University, Chikusa-ku, Nagoya, 464-8602 Japan}

{x12003g@math.nagoya-u.ac.jp}

\end{document}